\documentclass{article}%
\usepackage{amsfonts}
\usepackage{amsmath}
\usepackage{amssymb}
\usepackage{xcolor}
\usepackage{graphicx}%
\providecommand{\U}[1]{\protect\rule{.1in}{.1in}}
\numberwithin{equation}{section}
\newtheorem{theorem}{Theorem}[section]

\newtheorem{corollary}[theorem]{Corollary}

\newtheorem{lemma}[theorem]{Lemma}

\newtheorem{proposition}[theorem]{Proposition}
\newtheorem{remark}[theorem]{Remark}

\newenvironment{proof}[1][Proof]{\noindent\textbf{#1.} }{\ \rule{0.5em}{0.5em}}

\newcommand{\EEE}{\color{black}} 
\begin{document}

\title{Second-Order $\Gamma$-Limit for the Cahn--Hilliard Functional with Dirichlet
Boundary Conditions, I}
\author{Irene Fonseca\\Department of Mathematical Sciences,\\Carnegie Mellon University,\\Pittsburgh PA 15213-3890, USA
\and Leonard Kreutz\\ School of Computation, Information and Technology,\\ Technical University of Munich \\  Garching bei M\"unchen, 85748, Germany 
\and Giovanni Leoni\\Department of Mathematical Sciences, \\Carnegie Mellon University, \\Pittsburgh PA 15213-3890, USA}
\maketitle

\begin{abstract}
This paper addresses the asymptotic development of order 2 by the $\Gamma$ -convergence of the Cahn--Hilliard functional with Dirichlet boundary
conditions. The Dirichlet data are assumed to be well separated from one of
the two wells. In the case where there are no interfaces, it is shown that
there is a transition layer near the boundary of the domain.

\end{abstract}

\section{Introduction}

In this paper, we study the second-order asymptotic development via $\Gamma$-convergence of the Cahn-Hilliard functional
\begin{equation}
F_{\varepsilon}(u):=\int_{\Omega}(W(u)+\varepsilon^{2}|\nabla u|^{2}%
)\,dx,\quad u\in H^{1}(\Omega), \label{functional cahn-hilliard}%
\end{equation}
subject to the Dirichlet boundary condition%
\begin{equation}
\operatorname*{tr}u=g_{\varepsilon}\quad\text{on }\partial\Omega.
\label{dirichlet boundary conditions}%
\end{equation}
Here $W:\mathbb{R}\rightarrow\lbrack0,\infty)$ is a double-well potential
with
\begin{equation}
W^{-1}(\{0\})=\{a,b\}, \label{wells}%
\end{equation}
$\Omega\subset\mathbb{R}^{N}$ is an open, bounded set with a smooth boundary,
$N\geq2$, and $g_{\varepsilon}\in H^{1/2}(\partial\Omega)$.

We recall that, given a metric space $X$ and a family of functions
$\mathcal{F}_{\varepsilon}:X\rightarrow\lbrack-\infty,\infty]$ for
$\varepsilon>0$, \emph{the asymptotic development of order} $n$   via $\Gamma$-convergence is written as:%
\begin{equation}
\mathcal{F}_{\varepsilon}=\mathcal{F}^{(0)}+\varepsilon\mathcal{F}%
^{(1)}+\cdots+\varepsilon^{n}\mathcal{F}^{(n)}+o(\varepsilon^{n}).
\label{asymptotic development}%
\end{equation}
This expansion holds if we can find $\mathcal{F}^{(i)}%
:X\rightarrow\lbrack-\infty,\infty]$, $i=0,\ldots,n$, such that the functions
\[
\mathcal{F}_{\varepsilon}^{(i)}:=\frac{\mathcal{F}_{\varepsilon}^{(i-1)}%
-\inf_{X}\mathcal{F}^{(i-1)}}{\varepsilon},\quad \mathcal{F}_{\varepsilon}^{(0)}:=\mathcal{F}_{\varepsilon}%
\]
 are  well-defined and the family $\{\mathcal{F}_{\varepsilon}^{(i)}%
\}_{\varepsilon}$ $\Gamma$-converges to $\mathcal{F}^{(i)}$ as $\varepsilon
\rightarrow0^{+}$. 

 The notion of asymptotic expansion was introduced by Anzellotti and Baldo in 1993 \cite{anzellotti-baldo1993}. Observe that if we define
\begin{align*}
\mathcal{U}_i := \{\text{minimizers of } \mathcal{F}^{(i)} \}, 
\end{align*} 
it can be shown that 
\begin{align*}
\mathcal{F}^{(i)} =\infty \text{ on } X \setminus \mathcal{U}_{i-1}
\end{align*}
and the sets of minimizers satisfy the nested relationship 
\begin{align*}
\mathcal{U}_n \subseteq \mathcal{U}_{n-1} \subseteq \cdots \subseteq \mathcal{U}_0 = \{\text{limits of minimizers of } \mathcal{F}_\varepsilon\}\,.
\end{align*}
 In general, the above set inclusions can be shown to be strict. Therefore, leveraging the hierarchical structure of functionals $\mathcal{F}^{(i)}$, this framework provides a systematic selection criterion for the limits of the minimizers of functionals $\mathcal{F}_\varepsilon$.

In many cases, the powers of $\varepsilon$ in the asymptotic development (\ref{asymptotic development}) may be replaced by more general scales $\delta_{\varepsilon}^{(i)}$, where $\delta_{\varepsilon}^{(i)}>0$ for all $i=1,\ldots,m$ and $\varepsilon>0$, $\delta_\varepsilon^{(0)}:=1$, and $\sigma_\varepsilon^{(i)}:=\delta_{\varepsilon}%
^{(i)}/\delta_{\varepsilon}^{(i-1)}\rightarrow0$ as $\varepsilon
\rightarrow0^{+}$ for all $i=1,\ldots,m$, and the asymptotic expansion takes
the form
 \[
\mathcal{F}_{\varepsilon}=\mathcal{F}^{(0)}+\delta_{\varepsilon}%
^{(1)}\mathcal{F}^{(1)}+\cdots+\delta_{\varepsilon}^{(n)}\mathcal{F}%
^{(n)}+o(\delta_{\varepsilon}^{(n)}).
\] 
In this setting, the functions $\mathcal{F}_{\varepsilon}^{(i)}$ are defined
by
\[
 \mathcal{F}_{\varepsilon}^{(i)}:=\frac{\mathcal{F}_{\varepsilon}^{(i-1)}%
-\inf_{X}\mathcal{F}^{(i-1)}}{\sigma_\varepsilon^{(i)}}, \quad \mathcal{F}_{\varepsilon}^{(0)}:=\mathcal{F}_{\varepsilon}.
\]

The second-order asymptotic expansion of the Cahn-Hilliard functional
(\ref{functional cahn-hilliard}) subject to a mass constraint
\begin{equation}
\int_{\Omega}u(x)\,dx=m \label{mass constraint}%
\end{equation}
was studied by the third author and Murray in \cite{leoni-murray2016},
\cite{leoni-murray2019} in dimension $N\geq2$. With $X := L^{1}(\Omega)$ and
\begin{equation}
\mathcal{G}_{\varepsilon}(u):=\left\{
\begin{array}
[c]{ll}%
\int_{\Omega}(W(u)+\varepsilon^{2}|\nabla u|^{2})\,dx & \text{if }u\in
H^{1}(\Omega),\,\int_{\Omega}u\,dx=m,\\
\infty & \text{otherwise in }L^{1}(\Omega),
\end{array}
\right.  \label{functional LM}%
\end{equation}
they proved that, under appropriate hypotheses on $\Omega$ and $W$, if $W$ is
quadratic near the wells, then
\[
\mathcal{G}_{\varepsilon}^{(2)}(u)=\frac{1}{2}\frac{C_{W}^{2}(N-1)^{2}%
}{W^{\prime\prime}(a)(b-a)^{2}}\kappa_{u}^{2}+(C_{\operatorname*{sym}}%
+C_{W}\tau_{u})\kappa_{u}\operatorname*{P}(\{u=a\};\Omega),
\]
where $u$ is a minimizer of the first-order functional (see \cite{baldo1990},
\cite{fonseca-tartar1989}, \cite{modica-mortola1977}, \cite{modica1987},
\cite{sternberg1988}),
\[
\mathcal{G}^{(1)}(v):=\left\{
\begin{array}
[c]{ll}%
C_{W}\operatorname*{P}(\{v=a\};\Omega) & \text{if }v\in BV(\Omega
;\{a,b\}),\,\int_{\Omega}v\,dx=m,\\
\infty & \text{otherwise in }L^{1}(\Omega),
\end{array}
\right.
\]
$\tau_{u}\in\mathbb{R}$ is a constant related to the mass constraint
(\ref{mass constraint}), $\kappa_{u}$ and $\operatorname*{P}(\{u=a\};\Omega)$
are the constant mean curvature and the perimeter of the set $\{u=a\}$ in
$\Omega$, respectively, the constants $C_{W}$ and $C_{\operatorname*{sym}}$
are given by\footnote{Note that our constants $c_{W}$ and
$c_{\operatorname*{sym}}$ differ by correspond to the constants $2c_{W}$ and
$2c_{\operatorname*{sym}}$ in \cite{leoni-murray2016}, \cite{leoni-murray2019}%
.}%
\begin{equation}
C_{W}:=2\int_{a}^{b}W^{1/2}(\rho)\,d\rho, \label{cW definition}%
\end{equation}
and
\[
C_{\operatorname*{sym}}:=2\int_{\mathbb{R}}W(z_{c}(t))t\,dt,
\]
where $c$ is the central zero of $W^{\prime}$ (see (\ref{W' three zeroes})),
and for $\alpha\in\mathbb{R}$, $z_{\alpha}$ solves to the Cauchy problem
\begin{equation}
\left\{
\begin{array}
[c]{l}%
z_{\alpha}^{\prime}=W^{1/2}(z_{\alpha}),\\
z_{\alpha}(0)=\alpha.
\end{array}
\right.  \label{cauchy problem z alpha}%
\end{equation}

The third author and Murray in \cite{leoni-murray2016},
\cite{leoni-murray2019} also considered the case where $W$ exhibits
subquadratic growth near the wells. This scenario had been previously
analyzed by the first and third authors together with Dal Maso in
\cite{dalmaso-fonseca-leoni2015}, where they assumed both zero Dirichlet
boundary conditions and the mass constraint (\ref{mass constraint}).

In the case of Dirichlet boundary conditions
(\ref{dirichlet boundary conditions}), we take $X:= L^{1}(\Omega)$ and define
\[
\mathcal{F}_{\varepsilon}(u):=\left\{
\begin{array}
[c]{ll}%
\int_{\Omega}(W(u)+\varepsilon^{2}|\nabla u|^{2})\,dx & \text{if }u\in
H^{1}(\Omega),\,\operatorname*{tr}u=g_{\varepsilon}\text{ on }\partial
\Omega,\\
\infty & \text{otherwise in }L^{1}(\Omega).
\end{array}
\right.
\]
Under suitable assumptions on $\Omega$, $W$, and $g_{\varepsilon}$, Owen,
Rubinstein, and Sternberg \cite{owen-rubinstein-sternberg1990} showed that the first non-trivial scale is $\delta_{\varepsilon}^{(1)}=\varepsilon$, i.e.,
\[
\mathcal{F}_{\varepsilon}^{(1)}(u)=\left\{
\begin{array}
[c]{ll}%
\int_{\Omega}(\frac{1}{\varepsilon}W(u)+\varepsilon|\nabla u|^{2})\,dx &
\text{if }u\in H^{1}(\Omega),\,\operatorname*{tr}u=g_{\varepsilon}\text{ on
}\partial\Omega,\\
\infty & \text{otherwise in }L^{1}(\Omega),
\end{array}
\right.
\]
and that the functionals $\{\mathcal{F}_{\varepsilon}^{(1)}\}_{\varepsilon}$
$\Gamma$-converge as $\varepsilon\rightarrow0^{+}$ to%
\begin{equation}
\mathcal{F}^{(1)}(u):=\left\{
\begin{array}
[c]{ll}%
\begin{split}
&C_{W}\operatorname*{P}(\{u= a\};\Omega)\\&\quad+\int_{\partial\Omega}\operatorname*{d}%
\nolimits_{W}(\operatorname*{tr}u,g)\,d\mathcal{H}^{N-1} \end{split} & \text{if }u\in
BV(\Omega;\{a,b\}),\\
\infty & \text{otherwise in }L^{1}(\Omega),
\end{array}
\right.  \label{firstOrderFormalDefinition}%
\end{equation}
where $g_{\varepsilon}\rightarrow g$ in $L^{1}(\partial\Omega)$,
$\operatorname*{d}\nolimits_{W}$ is the geodesic distance determined by $W$
\begin{equation}
\operatorname*{d}\nolimits_{W}(r,s):=\left\{
\begin{array}
[c]{ll}%
2\left\vert \int_{r}^{s}W^{1/2}(\rho)\,d\rho\right\vert  & \text{if }%
r\in\{a,b\}\text{ or }s\in\{a,b\},\\
\infty & \text{otherwise,}%
\end{array}
\right.  \label{distance definition}%
\end{equation}
and the constant $C_{W}$ is given in (\ref{cW definition}). We also refer to
the recent work by Cristoferi and Gravina \cite{cristoferi-gravina2021}, who
addressed the vectorial case and considered potentials where the wells depend
on the spatial variable $x$, and to the work by Gazoulis \cite{gazoulis2024},
who studied the vectorial case under different settings.

We aim to extend the results of Owen, Rubinstein, and Sternberg
\cite{owen-rubinstein-sternberg1990} by determining the second-order
asymptotic expansion of $\mathcal{F}_{\varepsilon}$ via $\Gamma$-convergence,
assuming the boundary data $g_{\varepsilon}:\overline{\Omega}\rightarrow
\mathbb{R}$ stay away from one of the two wells $a$, $b$:%
\begin{equation}
a<\alpha_{-}\leq g_{\varepsilon}(x)\leq b\label{bounds g}%
\end{equation}
for all $x\in\overline{\Omega}$, all $\varepsilon\in(0,1)$, and some constant
$\alpha_{-}$. In this article, we consider only the case where the Dirichlet boundary datum is close to the value 
$b$ and far from $a$. The case where the boundary datum is close to $a$ and far from $b$ can be addressed using the same arguments. Under this hypothesis, when the constant $\alpha_{-}$ is
sufficiently close to $b$, the only minimizer of $\mathcal{F}^{(1)}$ is the
constant function $b$ (see Proposition \ref{proposition b minimizer} below).
Hence, we assume that
\begin{equation}
u_{0}\equiv b\quad\text{is the unique minimizer of }\mathcal{F}^{(1)}%
\text{.}\label{u0=b}%
\end{equation}
In this case, due to \eqref{firstOrderFormalDefinition}, we have 
\begin{align*}
\min \mathcal{F}^{(1)} = \int_{\partial\Omega}\operatorname*{d}%
\nolimits_{W}(b,g)\,d\mathcal{H}^{N-1}
\end{align*}
and we define%
\begin{align}
&  \mathcal{F}_{\varepsilon}^{(2)}(u):=\frac{\mathcal{F}_{\varepsilon}%
^{(1)}(u)-\min\mathcal{F}^{(1)}}{\varepsilon}\label{F 2 epsilon}\\
&  =\int_{\Omega}\left(  \frac{1}{\varepsilon^{2}}W(u)+|\nabla u|^{2}\right)
\,dx-\frac{1}{\varepsilon}\int_{\partial\Omega}\operatorname*{d}%
\nolimits_{W}(b,g)\,d\mathcal{H}^{N-1}\nonumber
\end{align}
if $u\in H^{1}(\Omega)$ and $\operatorname*{tr}u=g_{\varepsilon}$ on
$\partial\Omega$, and $\mathcal{F}_{\varepsilon}^{(2)}(u):=\infty$ otherwise
in $L^{1}(\Omega)$.

The main result of this paper is the following theorem:

\begin{theorem}
\label{theorem main}Let $\Omega\subset\mathbb{R}^{N}$ be an open, bounded,
connected set with boundary of class $C^{2,d}\ $, $0<d\leq1$. Assume that $W$
satisfies  \eqref{W_Smooth}-\eqref{W' three zeroes} and that
$g_{\varepsilon}$ satisfy \eqref{bounds g},
\eqref{g epsilon smooth}-\eqref{g epsilon -g bound}. Suppose also that
\eqref{u0=b} holds. Then
\begin{equation}
\mathcal{F}^{(2)}(u)=\int_{\partial\Omega}\kappa(y)\int_{0}^{\infty}%
2W^{1/2}(z_{g(y)}(s))z_{g(y)}^{\prime}(s)s\,ds\,d\mathcal{H}^{N-1}(y)
\label{second order}%
\end{equation}
if $u=b$ and $\mathcal{F}^{(2)}(u)=\infty$ otherwise in $L^{1}(\Omega)$. Here,
$\kappa$ is the mean curvature of $\partial\Omega$ and $z_{\alpha}$ is the
solution to the Cauchy problem \eqref{cauchy problem z alpha} with
$\alpha=g(y)$.

In particular, if $u_{\varepsilon}\in H^{1}(\Omega)$ is a minimizer of
\eqref{functional cahn-hilliard} subject to the Dirichlet boundary condition
\eqref{dirichlet boundary conditions}, then
\begin{align}
\int_{\Omega}(W(u_{\varepsilon})  &  +\varepsilon^{2}|\nabla u_{\varepsilon
}|^{2})\,dx=\varepsilon\int_{\partial\Omega}\operatorname*{d}\nolimits_{W}%
(b,g)\,d\mathcal{H}^{N-1}\label{sharp bound}\\
&  +\varepsilon^{2}\int_{\partial\Omega}\kappa(y)\int_{0}^{\infty}%
2W^{1/2}(z_{g(y)}(s))z_{g(y)}^{\prime}(s)s\,dsd\mathcal{H}^{N-1}%
(y)+o(\varepsilon^{2}).\nonumber
\end{align}

\end{theorem}

\begin{remark}
In the case where $g$ is allowed to take the value $a$ but \eqref{u0=b}
continues to hold, the scaling
\[
\mathcal{F}_{\varepsilon}^{(2)}(u):=\frac{\mathcal{F}_{\varepsilon}%
^{(1)}(u)-\min\mathcal{F}^{(1)}}{\varepsilon|\log\varepsilon|}%
\]
should replace the scaling in \eqref{F 2 epsilon}, as the latter becomes
incorrect in this context. We address this problem in the paper \cite{fonseca-kreutz-leoni2025II}.
\end{remark}

\begin{remark}
The case where the minimizer $u_{0}$ of the functional $\mathcal{F}^{(1)}$ in
\eqref{firstOrderFormalDefinition} is not constant, the analysis becomes
considerably more complex. By leveraging recent results from De Phillipis and
Maggi \cite{dephilippis-maggi2015}, it can be shown that if $\Omega$ and $g$
are sufficiently regular, then by modifying $E_{0}:=\{u_{0}=a\}$ on a set of
Lebesgue measure zero, $E_{0}$ is open and its trace $\partial E_{0}%
\cap\partial\Omega$ has finite perimeter in $\partial\Omega$. Moreover, if
$M=\overline{\partial E_{0}\cap\Omega}$, then $\partial_{\partial\Omega
}(\partial E_{0}\cap\partial\Omega)=M\cap\partial\Omega$, and there exists a
closed set $\Sigma\subseteq M$, with $\mathcal{H}^{N-2}(M\setminus\Sigma)=0$,
such that $M\setminus\Sigma$ is a $C^{1,1/2}$ hypersurface with boundary,
$M\setminus\Sigma$ has zero mean curvature in $\Omega$ and satisfies the
Young's law%
\begin{equation}
\nu_{E_{0}}(x)\cdot\nu_{\partial\Omega}(x)=\frac{1}{C_{W}}(\operatorname*{d}%
\nolimits_{W}(a,g(x))-\operatorname*{d}\nolimits_{W}(b,g(x)))
\label{young's law 1}%
\end{equation}
for all $x\in(M\cap\partial\Omega)\setminus\Sigma$. Here $\nu_{E_{0}}$ and
$\nu_{\partial\Omega}$ are the outward unit normals to $E_{0}$ and $\Omega$,
respectively. We are currently investigating this problem in dimension two. In
this setting,
\[
u_{0}=a\chi_{E_{0}}+b\chi_{\Omega\setminus E_{0}},
\]
where%
\[
\partial E_{0}\cap\Omega=\bigcup_{i=1}^{m}\Sigma_{i},
\]
with $\Sigma_{i}$ being disjoint segments that have endpoints $P_{i}$ and
$Q_{i}$ on $\partial\Omega$ and form angles $\theta_{i}$ that satisfy Young's
law (\ref{young's law 1}). By adapting the techniques presented in this paper,
we have constructed $u_{\varepsilon}\in H^{1}(\Omega)$ satisfying the
Dirichlet boundary conditions (\ref{dirichlet boundary conditions}) and
converging to $u_{0}$ in $L^{1}(\Omega)$, such that%
\begin{align*}
&  \limsup_{\varepsilon\rightarrow0^{+}}\mathcal{F}_{\varepsilon}%
^{(2)}(u_{\varepsilon})\leq\int_{\partial\Omega\cap\overline{\Omega}^{b}%
}\kappa(y)\int_{0}^{\infty}2W^{1/2}(z_{g(y)}(s))z_{g(y)}^{\prime
}(s)s\,dsd\mathcal{H}^{1}(y)\\
&  \quad+\int_{\partial\Omega\cap\overline{\Omega}^{a}}\kappa(y)\int_{-\infty
}^{0}2W^{1/2}(z_{g(y)}(s))z_{g(y)}^{\prime}(s)|s|\,dsd\mathcal{H}^{1}(y)\\
&  \quad-\sum_{i=1}^{m}\frac{1+\cos\theta_{i}}{\sin\theta_{i}}C_{i}-\sum
_{i=1}^{m}\frac{1-\cos\theta_{i}}{\sin\theta_{i}}D_{i},
\end{align*}
where
\begin{align*}
C_{i}  &  :=\int_{0}^{\infty}2W^{1/2}(z_{g(P_{i})}(s))z_{g(P_{i})}^{\prime
}(s)s\,ds+\int_{0}^{\infty}2W^{1/2}(z_{g(Q_{i})}(s))z_{g(Q_{i})}^{\prime
}(s)s\,ds,\\
D_{i}  &  :=\int_{-\infty}^{0}2W^{1/2}(z_{g(P_{i})}(s))z_{g(P_{i})}^{\prime
}(s)|s|\,ds+\int_{-\infty}^{0}2W^{1/2}(z_{g(Q_{i})}(s))z_{g(Q_{i})}^{\prime
}(s)|s|\,ds,
\end{align*}
and $z_{\alpha}$ solves the Cauchy problem \eqref{cauchy problem z alpha} with
$\alpha=g(y)$, and $\Omega^{r}:=\{x\in\Omega:\,u_{0}(x)=r\}$, $r\in\{a,b\}$.
\end{remark}

Theorem \ref{theorem main} is in the same spirit as the work by Anzellotti,
Baldo, and Orlandi \cite{anzellotti-baldo-orlandi1996}, who considered the
case $W(\rho)=\rho^{2}$ and derived a formula similar to (\ref{second order}).
Our proof, however, takes a different approach and relies on the asymptotic
development of order two by $\Gamma$-convergence of the weighted
one-dimensional functional%
\begin{equation}
G_{\varepsilon}(v):=\int_{0}^{T}(W(v(t))+\varepsilon^{2}(v^{\prime}%
(t))^{2})\omega(t)\,dt,\quad v\in H^{1}(I), \label{functional 1d}%
\end{equation}
subject to the Dirichlet boundary conditions
\begin{equation}
v(0)=\alpha_{\varepsilon},\quad v(T)=\beta_{\varepsilon},
\label{dirichlet boundary conditions 1d}%
\end{equation}
where $\omega$ is a smooth positive weight, and
\begin{equation}
a<\alpha_{\varepsilon},\,\beta_{\varepsilon}\leq b. \label{initial values 1d}%
\end{equation}
The second-order asymptotic expansion of this functional was studied by the
third author and Murray (\cite{leoni-murray2016}, \cite{leoni-murray2019}) in
the case where the Dirichlet boundary conditions
(\ref{dirichlet boundary conditions 1d}) were replaced by the mass constraint:%
\begin{equation}
\int_{0}^{T}v(t)\omega(t)\,dt=m. \label{mass 1d}%
\end{equation}
The key difference in our proof of the $\Gamma$-liminf inequality is that in
\cite{leoni-murray2016}, \cite{leoni-murray2019}, the authors utilized a
rearrangement technique based on the isoperimetric function to reduce the
functional (\ref{functional LM}) to the one-dimensional weighted problem
(\ref{functional 1d}) and (\ref{mass 1d}). This approach, however, is not
feasible in our case (except in the case of trivial boundary conditions).
Instead, we adapt techniques from Sternberg and Zumbrum
\cite{sternberg-zumbrun1998} and Caffarelli and Cordoba
\cite{caffarelli-cordoba1995} to study the behavior of minimizers of
(\ref{functional cahn-hilliard}) and (\ref{dirichlet boundary conditions})
near the boundary and use slicing arguments.

The case $N=1$ was previously addressed by Anzellotti and Baldo
\cite{anzellotti-baldo1993} under the assumption that $W$ is zero in a
neighborhood of $a$ and $b$, and by Bellettini, Nayam, and Novaga
\cite{bellettini-nayam-novaga2015} in the periodic case.

This paper is organized as follows. In Section \ref{section 1d functional}, we
characterize the asymptotic development of order two by the $\Gamma$ convergence
of the weighted one-dimensional family of functionals $G_{\varepsilon}$
defined in (\ref{functional 1d}). Section \ref{section minimizers} explores the
qualitative properties of critical points and minimizers of functional
\ref{functional cahn-hilliard}. Finally, in Section
\ref{section main theorems}, we prove Theorem \ref{theorem main}.

\section{Preliminaries}

We assume that the double-well potential $W:\mathbb{R}\rightarrow
\lbrack0,\infty)$ satisfies the following hypotheses:%
\begin{align}
 \begin{split} &W\text{ is of class $C^{2,\alpha_{0}}(\mathbb{R})$, }\alpha_{0}%
\in(0,1)\text{, and has precisely two zeros} \\&\text{at } a \text{ and } b, \text{ with } a<b,\end{split}\label{W_Smooth}\\
&  W^{\prime\prime}(a)>0,\quad W^{\prime\prime}(b)>0,\label{WPrime_At_Wells}\\
&  \lim_{s\rightarrow-\infty}W^{\prime}(s)=-\infty,\quad\lim_{s\rightarrow
\infty}W^{\prime}(s)=\infty,\label{WGurtin_Assumption}\\
& \text{$W^{\prime}$ has exactly 3 zeros at $a,b,c$ with $a<c<b$,}\quad W^{\prime\prime
}(c)<0, \label{W' three zeroes}%
\end{align}
Let
\begin{equation}
a<\alpha_{-}<\min\left\{  c,\frac{a+b}{2}\right\}  \leq\max\left\{
c,\frac{a+b}{2}\right\}  <\beta_{-}<b. \label{alpha and beta minus}%
\end{equation}

\begin{remark}
\label{remark W near b}Since $W\in C^{2}(\mathbb{R})$, $W(a)=W^{\prime}(a)=0$,
$W(b)=W^{\prime}(b)=0$, and $W^{\prime\prime}(a)$, $W^{\prime\prime}(b)>0$,
there exists a constant $\sigma>0$ depending on $\alpha_{-}$ and $\beta_{-}$
such that%
\begin{align}
\sigma^{2}(b-s)^{2}  &  \leq W(s)\leq\frac{1}{\sigma^{2}}(b-s)^{2}%
\quad\text{for all }\alpha_{-}\leq s\leq b+1,\label{W near b}\\
\sigma^{2}(s-a)^{2}  &  \leq W(s)\leq\frac{1}{\sigma^{2}}(s-a)^{2}%
\quad\text{for all }a-1\leq s\leq\beta_{-}. \label{W near a}%
\end{align}

\end{remark}

\bigskip

\begin{proposition}
\label{proposition asymptotic behavior}For $a<\alpha_{-}<b$ and  $0 <\delta\leq\sigma^{-1}$, we
have%
\begin{align}
-  &  \frac{\sigma^{-1}}{2}\log(\sigma^{-2}\delta)+\sigma^{-1}\log
(b-\alpha)-\sigma^{-1}\log(1+2\sigma^{-1}(b-\beta)/\delta^{1/2})\nonumber\\
&  \leq\int_{\alpha}^{\beta}\frac{1}{(\delta+W(s))^{1/2}}\,ds\leq-\frac
{\sigma}{2}\log(\sigma^{2}\delta)+\sigma\log(1+2(b-a)) \label{near b log}%
\end{align}
for every $\alpha_{-}\leq\alpha\leq\beta\leq b$, where $\sigma>0$ is the
constant given in \eqref{W near b}.
\end{proposition}

\begin{proof}
By (\ref{W near b}),
\[
\frac{\sigma^{-1}}{(\sigma^{-2}\delta+(b-s)^{2})^{1/2}}\leq\frac{1}%
{(\delta+W(s))^{1/2}}\leq\frac{\sigma}{(\sigma^{2}\delta+(b-s)^{2})^{1/2}}.
\]
Hence, it suffices to estimate%
\[
\mathcal{A}:=\int_{\alpha}^{\beta}\frac{1}{(r+(b-s)^{2})^{1/2}}\,ds.
\]
Consider the change of variables $r^{1/2}t=b-s$, so that $- r^{1/2}dt=ds$. Then%
\begin{align*}
\mathcal{A}  &  =\int_{\alpha}^{\beta}\frac{1}{(r+(b-s)^{2})^{1/2}}%
\,ds=\frac{r^{1/2}}{r^{1/2}}\int_{(b-\beta)/r^{1/2}}^{(b-\alpha)/r^{1/2}}%
\frac{1}{\left(  1+t^{2}\right)  ^{1/2}}\,dt\\
&  =[\log[t+(t^{2}+1)^{1/2}]]_{(b-\beta)/r^{1/2}}^{(b-\alpha)/r^{1/2}} \\
&  =-\frac{1}{2}\log r+\log(b-\alpha+[r+(b-\alpha)^{2}]^{1/2})\\
&  \quad-\log((b-\beta)/r^{1/2}+[1+(b-\beta)^{2}/r]^{1/2}).
\end{align*}
Hence, for $0<r\leq 1$, we have 
\begin{align*}
-\frac{1}{2}  &  \log r+\log(b-\alpha)-\log(1+2(b-\beta)/r^{1/2})\\
&  \leq\mathcal{A}\leq-\frac{1}{2}\log r+\log(1+2(b-a)).
\end{align*}\hfill
\end{proof}

\begin{proposition}
\label{proposition difference}Let $a\leq\alpha_{\varepsilon}\leq
\beta_{\varepsilon}\leq b$. Then there exists a constant $C>0$  depending on $\sigma$  such that
\begin{equation}
\int_{\alpha_{\varepsilon}}^{\beta_{\varepsilon}}\left[  \frac{2}%
{(\delta+W(s))^{1/2}+W^{1/2}(s)}-\frac{1}{(\delta+W(s))^{1/2}}\right]
\,ds\leq C \label{2d 40}%
\end{equation}
for all $0<\delta<1$, where $\sigma>0$ is the constant given in (\ref{W near b}).
\end{proposition}

\begin{proof}
For $A\geq0$, we have%
\begin{align*}
\frac{2}{(\delta+A)^{1/2}+A^{1/2}}-\frac{1}{(\delta+A)^{1/2}}  &
=\frac{(\delta+A)^{1/2}-A^{1/2}}{[(\delta+A)^{1/2}+A^{1/2}](\delta+A)^{1/2}}\\
&  =\frac{\delta}{[(\delta+A)^{1/2}+A^{1/2}]^{2}(\delta+A)^{1/2}}\geq0.
\end{align*}
Hence, the left  side of (\ref{2d 40}) can be bounded from above by%
\begin{align*}
&  \int_{a}^{b}\frac{\delta}{[(\delta+W(s))^{1/2}+W^{1/2}(s)]^{2}%
(\delta+W(s))^{1/2}}ds\\
&  =\int_{a}^{c}\frac{\delta}{[(\delta+W(s))^{1/2}+W^{1/2}(s)]^{2}%
(\delta+W(s))^{1/2}}\\
&  \quad+\int_{c}^{b}\frac{\delta}{[(\delta+W(s))^{1/2}+W^{1/2}(s)]^{2}%
(\delta+W(s))^{1/2}}\\
&  =:\mathcal{A}+\mathcal{B}.
\end{align*}
By (\ref{W near a}) we have
\begin{align*}
\mathcal{A}  &  \leq\int_{a}^{c}\frac{\delta}{[(\delta+\sigma^{2}%
(s-a)^{2})^{1/2}+\sigma(s-a)]^{2}(\delta+\sigma^{2}(s-a)^{2})^{1/2}}ds\\
&  =\int_{0}^{(c-a)/\delta^{1/2}}\frac{\delta}{[(\delta+\sigma^{2}\delta
t^{2})^{1/2}+\sigma\delta^{1/2}t]^{2}(\delta+\sigma^{2}\delta t^{2})^{1/2}%
}\delta^{1/2}dt\\
&  \leq\int_{0}^{\infty}\frac{1}{[(1+\sigma^{2}t^{2})^{1/2}+\sigma
t]^{2}(1+\sigma^{2}t^{2})^{1/2}}dt,
\end{align*}
where we have made the change of variables $s-a=\delta^{1/2}t$, so that
$ds=\delta^{1/2}dt$, and used the fact that $0<\delta<1$\EEE. A similar estimate holds for $\mathcal{B}$.\hfill
\end{proof}

Next, we study the properties of the solutions to the Cauchy problem
(\ref{cauchy problem z alpha}).

\begin{proposition}
\label{proposition z epsilon}Assume that $W$ satisfies
\eqref{W_Smooth}-\eqref{W' three zeroes} and let $a<\alpha<b$. Then the
Cauchy problem \eqref{cauchy problem z alpha} admits a unique global solution
$z_{\alpha}:\mathbb{R}\rightarrow\mathbb{R}$. The function $z_{\alpha}$ is
increasing with%
\[
a<z_{\alpha}(t)<b\quad\text{for all }t\in\mathbb{R},
\]
and
\begin{equation}
\lim_{t\rightarrow-\infty}z_{\alpha}(t)=a,\quad\lim_{t\rightarrow\infty
}z_{\alpha}(t)=b. \label{limits infinity}%
\end{equation}
Moreover, if $\alpha_{-}\leq\alpha<b$, where $\alpha_{-}$ is given in
\eqref{alpha and beta minus}, then
\begin{equation}
(b-\alpha)e^{-\sigma^{-1}t}\leq b-z_{\alpha}(t)\leq(b-a)e^{-\sigma t}
\label{estimate z alpha}%
\end{equation}
for all $t\geq0$.
\end{proposition}

\begin{proof}
 As $\sqrt{W}$ is Lipschitz continuous in $[a-1,b+1]$, the Cauchy problem
(\ref{cauchy problem z alpha}) admits a unique local solution. As the
constant functions $a$ and $b$ are solutions to the differential equation, by
uniqueness, $a<z_{\alpha}(t)<b$ for all $t$ in the interval of existence of
$z_{\alpha}$. This implies that $z_{\alpha}$ can be uniquely extended to the
entire real line. Standard ODEs techniques show that (\ref{limits infinity}) is valid.

If $\alpha_{-}\leq\alpha<b$, since $z_{\alpha}$ is increasing, using
(\ref{limits infinity}), we can find $T_{\alpha}<0$ such that $z_{\alpha
}(T_{\alpha})=\alpha_{-}$ and $z_{\alpha}(t)>\alpha_{-}$ for all $t>T_{\alpha
}$. In turn, by (\ref{W near b}),
\[
\sigma(b-z_{\alpha}(t))\leq z_{\alpha}^{\prime}(t)\leq\sigma^{-1}(b-z_{\alpha
}(t))\quad\text{for all }t\geq T_{\alpha}.
\]
Dividing by $b-z_{\alpha}(t)$ and integrating from $0$ to $t$ gives
\[
-\sigma^{-1}t\leq\log\left(  \frac{b-z_{\alpha}(t)}{b-\alpha}\right)
\leq-\sigma t,
\]
which implies (\ref{estimate z alpha}).\hfill
\end{proof}

We assume that $g_{\varepsilon}:\partial\Omega\rightarrow\mathbb{R}$ and
$g:\partial\Omega\rightarrow\mathbb{R}$ satisfy the following hypotheses:
\begin{align}
g_{\varepsilon}  &  \in H^{1}(\partial\Omega),\label{g epsilon smooth}\\
(\varepsilon|\log\varepsilon|)^{1/2}\int_{\partial\Omega}|\nabla_{\tau
}g_{\varepsilon}|^{2}d\mathcal{H}^{N-1}  &  =o(1)\quad\text{as }%
\varepsilon\rightarrow0^{+},\label{g epsilon bound derivatives}\\
|g_{\varepsilon}(x)-g(x)|  &  \leq C\varepsilon^{\gamma},\quad x\in
\partial\Omega,\quad\gamma>1 \label{g epsilon -g bound}%
\end{align}
for all $\varepsilon\in(0,1)$ and for some constant $C>0$. Here, $\nabla
_{\tau}$ denotes the tangential gradient. 

 Condition \eqref{g epsilon bound derivatives} is of a technical nature and ensures that, in the energy estimates for the recovery sequence in the $\Gamma$-limsup inequality, the tangential component of the gradient near the boundary of $\Omega$ does not contribute to the limiting energy (see \eqref{299c} below). In particular, this condition is satisfied if $g_\varepsilon=g$ for all $\varepsilon>0$ for some $g \in H^1(\partial \Omega)$.

 Observe that the hypotheses \eqref{g epsilon bound derivatives} and \eqref{g epsilon -g bound} imply some regularity of $g$. In particular, when $N=2$, we see that the functions $g_\varepsilon$ are continuous, and since \eqref{g epsilon -g bound} implies uniform convergence, it follows that $g$ must be continuous.

For $a\leq\alpha\leq b$, let%
\begin{align*}
\phi(\alpha)  &  :=\operatorname*{d}\nolimits_{W}(a,\alpha)-\operatorname*{d}%
\nolimits_{W}(\alpha,b)\\
&  =2\int_{a}^{\alpha}W^{1/2}(\rho)\,d\rho-2\int_{\alpha}^{b}W^{1/2}%
(\rho)\,d\rho,
\end{align*}
where $\operatorname*{d}\nolimits_{W}$ is defined in
(\ref{distance definition}). Since $\phi(a)=- C_{W}$, $\phi(b)=C_{W}$,
$\phi^{\prime}(\alpha)=4W^{1/2}(\alpha)>0$ for $\alpha\in(a,b)$, there exists
a unique $\bar{\alpha}\in(a,b)$ such that
\begin{equation}
\phi(\bar{\alpha})=0\quad\text{and\quad}\phi(\alpha) >0\quad\text{for all
}\bar{\alpha}<\alpha\leq b. \label{phi bar}%
\end{equation}

\begin{proposition}
\label{proposition b minimizer}Let $\Omega\subset\mathbb{R}^{N}$ be an open,
bounded, connected set with boundary of class $C^{2,d}$, $0<d\leq1$. Assume
that $W$ satisfies \eqref{W_Smooth}-\eqref{W' three zeroes} and
that $g_{\varepsilon}$ satisfy \eqref{bounds g},
\eqref{g epsilon smooth}-\eqref{g epsilon -g bound}. Suppose that
\begin{equation}
g_{-}>\bar{\alpha} \label{g minus}%
\end{equation}
where $\bar{\alpha}$ is given in \eqref{phi bar}. Then the constant function
$b$ is the unique minimizer of the functional $\mathcal{F}^{(1)}$ defined in \eqref{firstOrderFormalDefinition}.
\end{proposition}

\begin{proof}
Let $u\in BV(\Omega;\{a,b\})$. We have  $\operatorname*{tr}u (x)\in\{a,b\}$ for $\mathcal{H}^{N-1}$-a.e.~$x \in \partial \Omega$ and thus 
\begin{align*}
\mathcal{F}^{(1)}(u)  &  \geq C_{W}\operatorname*{P}(\{u=b\};\Omega
)+\int_{\partial\Omega}\operatorname*{d}\nolimits_{W}(\operatorname*{tr}%
u,g)\,d\mathcal{H}^{N-1}\\
&  \geq\int_{\partial\Omega}\operatorname*{d}\nolimits_{W}(b,g)\,d\mathcal{H}%
^{N-1}=\mathcal{F}^{(1)}(b)
\end{align*}
provided
\[
\int_{\partial\Omega\cap\{\operatorname*{tr}u=a\}}\operatorname*{d}%
\nolimits_{W}(a,g)\,d\mathcal{H}^{N-1}\geq\int_{\partial\Omega\cap
\{\operatorname*{tr}u=a\}}\operatorname*{d}\nolimits_{W}(g,b)\,d\mathcal{H}%
^{N-1}.
\]
By (\ref{bounds g}), (\ref{phi bar}), and (\ref{g minus}), we obtain
\begin{align*}
\operatorname*{d}\nolimits_{W}(a,g(x))  >\operatorname*{d}\nolimits_{W}(g(x),b).
\end{align*} 
%\begin{align*}
%\operatorname*{d}\nolimits_{W}(a,g(x))  &  =2\int_{a}^{g(x)}W^{1/2}%
%(\rho)\,d\rho\geq2\int_{a}^{g_{-}}W^{1/2}(\rho)\,d\rho>2\int_{a}^{\bar{\alpha
%}}W^{1/2}(\rho)\,d\rho\\
%&  =2\int_{\bar{\alpha}}^{b}W^{1/2}(\rho)\,d\rho\geq2\int_{g(x)}^{b}%
%W^{1/2}(\rho)\,d\rho=\operatorname*{d}\nolimits_{W}(g(x),b).
%\end{align*} \EEE
Hence, if $\operatorname*{P}(\{u=b\};\Omega)>0$ or $\mathcal{H}^{N-1}%
(\partial\Omega\cap\{\operatorname*{tr}u=a\})>0$, we have that $\mathcal{F}%
^{(1)}(u)>\mathcal{F}^{(1)}(b)$, which shows that the constant function $b$ is
the unique minimizer of $\mathcal{F}^{(1)}$.\hfill
\end{proof}

In what follows, given $z\in\mathbb{R}^{N}$, with a slight abuse of notation,
we write%
\begin{equation}
z=(z^{\prime},z_{N})\in\mathbb{R}^{N-1}\times\mathbb{R}, \label{z' notation}%
\end{equation}
where $z^{\prime}:=(z_{1},\ldots,z_{N-1})$. We also write
\begin{equation}
\nabla^{\prime}:=\left(  \frac{\partial}{\partial z_{1}},\ldots,\frac
{\partial}{\partial z_{N-1}}\right)  . \label{grad' notation}%
\end{equation}
Also, given $\delta>0$ we define%
\begin{equation}
\Omega_{\delta}:=\{x\in\Omega:\,\operatorname*{dist}(x,\partial\Omega
)<\delta\}. \label{Omega delta}%
\end{equation}

The following result is classical. We recall it and its proof for the reader's convenience.

\begin{lemma}
\label{lemma diffeomorphism}Assume that $\Omega\subset\mathbb{R}^{N}$ is an
open, bounded, connected set and that its boundary $\partial\Omega$ is of
class $C^{2,d}$, $0<d\leq1$. If $\delta>0$ is sufficiently small, the mapping
\[
\Phi:\partial\Omega\times\lbrack0,\delta]\rightarrow\overline{\Omega}_{\delta}%
\]
given by%
\[
\Phi(y,t)=y+t\nu(y),
\]
where $\nu(y)$ is the unit inward normal vector to $\partial\Omega$ at $y$ and
$\Omega_{\delta}$ is defined in \eqref{Omega delta}, is a diffeomorphism of
class $C^{1,d}$. Moreover, $\Omega\setminus\Omega_{\delta}$ is connected for
all $\delta>0$ sufficiently small. Finally,%
\begin{equation}
\det J_{\Phi}(y,0)=1\quad\text{for all }y\in\partial\Omega\label{det=1}%
\end{equation}
and%
\begin{equation}
\frac{\partial}{\partial t}\left.  \det J_{\Phi}(y,t)\right\vert _{t=0}%
=\kappa(y)\quad\text{for all }y\in\partial\Omega, \label{curvature}%
\end{equation}
where $\kappa(y)$ is the mean curvature of $\partial\Omega$ at $y$.
\end{lemma}

\begin{proof}
The fact that $\Phi:\partial\Omega\times\lbrack0,\delta]\rightarrow
\overline{\Omega}_{\delta}$ is a diffeomorphism $\delta>0$ is sufficiently
small is classical (see, e.g. \cite[Theorem 6.17]{lee-book2013}). Its inverse
is given by%
\[
\Phi^{-1}(x)=(y(x),\operatorname*{dist}(x,\partial\Omega)),
\]
where we denote by $y(x)\in\partial\Omega$ the unique projection of $x$ onto
$\partial\Omega$, with
\[
\operatorname*{dist}(x,\partial\Omega)=|y(x)-x|.
\]

Next, we show that $\Omega\setminus\Omega_{\delta}$ is pathwise connected. Let
$x_{0}$ and $x_{1}$ be two points in $\Omega\setminus\Omega_{\delta}$. Since
$\Omega$ is open and connected, there exists a continuous function
$f:[0,1]\rightarrow\Omega$ such that $f(0)=x_{0}$ and $f(1)=x_{1}$. Since
$\Phi$ is a diffeomorphism and $\Phi(\partial\Omega\times\{\delta
\})=\partial(\Omega\setminus\Omega_{\delta})$, the function%
\[
h(x)=y(x)+ \delta  \nu(y(x)),\quad x\in\overline{\Omega}_{\delta},
\]
is continuous, with $h(\overline{\Omega}_{\delta})=\partial(\Omega
\setminus\Omega_{\delta})$. Note that if $x\in\partial(\Omega\setminus
\Omega_{\delta})$, then $h(x)=x$. Hence, if we extend $h$ to be the identity
in $\Omega\setminus\Omega_{\delta}$, we have a continuous function
$h:\overline{\Omega}\rightarrow\Omega\setminus\Omega_{\delta}$. Then $h\circ
f:[0,1]\rightarrow\Omega\setminus\Omega_{\delta}$ is continuous and $(h\circ
f)(0)=x_{0}$ and $(h\circ f)(1)=x_{1}$, which shows that $\Omega
\setminus\Omega_{\delta}$ is pathwise connected.

To prove (\ref{det=1}) and (\ref{curvature}), we fix $y_{0}\in\partial\Omega$
and find a rigid motion $T:\mathbb{R}^{N}\rightarrow\mathbb{R}^{N}$, with
$T(y_{0})=0$, $r>0$, and a function $f:B_{N-1}(0,r)\rightarrow\mathbb{R}$ of
class $C^{2,d}$ such that $f(0)=0$, $\nabla^{\prime}f(0)=0$, and
\[
T(B(y_{0},r)\cap\Omega)=\{z\in\mathbb{R}^{N}:\,z_{N}>f(z^{\prime}%
),\,z^{\prime}\in B_{N-1}(0,r)\}=:V,
\]
where we are using the notations (\ref{z' notation}) and (\ref{grad' notation}%
) and $B_{N-1}(0,r)$ is the open ball centered at $0$ and radius $r$ in
$\mathbb{R}^{N-1}$. The unit inward normal to $\partial V$ at a point
$(z^{\prime},f(z^{\prime}))$ is the vector%
\[
\nu=\frac{(-\nabla^{\prime}f(z^{\prime}),1)}{(1+|\nabla^{\prime}f(z^{\prime
})|_{N-1}^{2})^{1/2}}%
\]
Hence, if we consider%
\[
\Psi(z^{\prime},t):=(z^{\prime},f(z^{\prime}))+t\frac{(-\nabla^{\prime
}f(z^{\prime}),1)}{(1+|\nabla^{\prime}f(z^{\prime})|_{N-1}^{2})^{1/2}},
\]
we have that for $i,j=1,\ldots,N-1$,
\begin{align*}
\frac{\partial\Psi_{j}}{\partial z_{i}}\left(  z^{\prime},t\right)   &
=\delta_{i,j}+t\frac{\partial}{\partial z_{i}}\left(  \frac{\frac{\partial
f}{\partial z_{j}}(z^{\prime})}{\sqrt{1+|\nabla^{\prime}f(z^{\prime}%
)|_{N-1}^{2}}}\right)  ,\\
\frac{\partial\Psi_{N}}{\partial z_{i}}\left(  z^{\prime},t\right)   &
=\frac{\partial f}{\partial z_{i}}(z^{\prime})-t\frac{\partial}{\partial
z_{i}}\left(  \frac{1}{\sqrt{1+|\nabla^{\prime}f(z^{\prime})|_{N-1}^{2}}%
}\right)  ,\\
\frac{\partial\Psi_{j}}{\partial t}(z^{\prime},t)  &  =\frac{-\frac{\partial
f}{\partial z_{j}}(z^{\prime})}{\sqrt{1+|\nabla^{\prime}f(z^{\prime}%
)|_{N-1}^{2}}},\quad\frac{\partial\Psi_{N}}{\partial t}\left(  z^{\prime
},t\right)  =\frac{1}{\sqrt{1+|\nabla^{\prime}f(z^{\prime})|_{N-1}^{2}}}.
\end{align*}
In particular, since $\nabla^{\prime}f(0)=0$,%
\[
J_{\Psi}(0,0)=I_{N-1}.
\]
This proves (\ref{det=1}). As 
\begin{align*}
\frac{\partial^{2}\Psi_{j}}{\partial t\partial z_{i}}\left(  z^{\prime
},t\right)   &  =\frac{\partial}{\partial z_{i}}\left(  \frac{\frac{\partial
g}{\partial z_{j}}(z^{\prime})}{\sqrt{1+|\nabla^{\prime}g(z^{\prime}%
)|_{N-1}^{2}}}\right)  ,\\
\frac{\partial^{2}\Psi_{N}}{\partial t\partial z_{i}}\left(  z^{\prime
},t\right)   &  =-\frac{\partial}{\partial z_{i}}\left(  \frac{1}%
{\sqrt{1+|\nabla^{\prime}g(z^{\prime})|_{N-1}^{2}}}\right)  ,\\
\frac{\partial^{2}\Psi_{j}}{\partial^{2}t}(z^{\prime},t)  &  =0,\quad
\frac{\partial^{2}\Psi_{N}}{\partial t^{2}}\left(  z^{\prime},t\right)  =0,
\end{align*}
using Jacobi's formula, we obtain%
\[
\frac{\partial\det J_{\Psi}}{\partial t}(z^{\prime},t)=\det J_{\Psi}%
(z^{\prime},t)\operatorname*{tr}\left(  J_{\Psi}^{-1}(z^{\prime}%
,t)\frac{\partial J_{\Psi}(z^{\prime},t)}{\partial t}\right)
\]
In particular, taking $z^{\prime}=0$ and using the fact that $J_{\Psi
}(0,t)=I_{N-1}$  we get
\[
\frac{\partial\det J_{\Psi}}{\partial t}(0,t)=\sum_{i=1}^{N-1}\left.
\frac{\partial}{\partial z_{i}}\left(  \frac{\frac{\partial g}{\partial z_{j}%
}(z^{\prime})}{\sqrt{1+|\nabla^{\prime}g(z^{\prime})|_{N-1}^{2}}}\right)
\right\vert _{z^{\prime}=0}=\kappa(y_{0}).
\]
By the arbitrariness of $y_{0}$, this concludes the proof of \eqref{curvature}.\hfill
\end{proof}

\section{A 1D Functional Problem}

\label{section 1d functional}Let
\[
I:=(0,T)
\]
and consider a weight function%
\begin{equation}
\omega\in C^{1,d}([0,T]),\quad\min_{\lbrack0,T]}\omega>0. \label{etaSmooth}%
\end{equation}
The prototype we have in mind is
\[
\omega(t):= 1+t\kappa(t).
\]
In this section, we study the second-order $\Gamma$-convergence of the family of functionals
\[
G_{\varepsilon}(v):=\int_{I}(W(v(t))+\varepsilon^{2}(v^{\prime}(t))^{2}%
)\omega(t)\,dt,\quad v\in\,H^{1}(I),
\]
subject to the Dirichlet boundary condition
\begin{equation}
v(0)=\alpha_{\varepsilon},\quad v(T)=\beta_{\varepsilon}, \label{1d dirichlet}%
\end{equation}
where $\alpha_\varepsilon,\beta_\varepsilon \in \mathbb{R}$.
In what follows, we will use the weighted BV space $BV_{\omega}(I)$ given by
all functions $v\in BV_{\operatorname*{loc}}(I)$ for which the norm
\[
\Vert v\Vert_{BV_{\omega}}:=\int_{I}|v(t)|\omega(t)\,dt+\int_{I}%
\omega(t)\,d|Dv|(t)
\]
is finite. For $v\in BV_{\omega}(I)$ we will also write the weighted total
variation of the derivative as
\[
|Dv|_{\omega}(E) :=\int_{E}\omega(t)\,d|Dv|(t).
\]
For a more detailed introduction to weighted BV spaces and their applications to phase-field models, we refer to \cite{baldi2001,fonseca-liu2017}.

We will study the second-order $\Gamma$-convergence with respect to
the metric in $L^{1}(I)$. This choice is motivated by the following
compactness result.

\begin{theorem}
[Compactness]\label{theorem 1d compactness}Assume that $W$ satisfies \eqref{W_Smooth}-\eqref{W' three zeroes}, that $\omega$
satisfies \eqref{etaSmooth}, and that $\alpha_{\varepsilon
}\rightarrow\alpha$ and $\beta_{\varepsilon}\rightarrow\beta$ as
$\varepsilon\rightarrow0^{+}$ for some $\alpha,\beta\in\mathbb{R}$. Let
$\varepsilon_{n}\rightarrow0^{+}$ and $v_{n}\in$\thinspace$H^{1}(I)$ be such
that%
\[
\sup_{n}\int_{I}\left(  \frac{1}{\varepsilon_{n}}W(v_{n}(t))+\varepsilon
_{n}(v_{n}^{\prime}(t))^{2}\right)  \omega(t)\,dt<\infty.
\]
Then there exist a subsequence $\{v_{n_{k}}\}_{k}$ of $\{v_{n}\}_{n}$ and
$v\in BV_{\omega}(I;\{a,b\})$ such that $v_{n_{k}}\rightarrow v$ in
\thinspace$L^{1}(I)$.
\end{theorem}

The proof is identical to the one of \cite[Proposition 4.3]{leoni-murray2016}
and so we omit it. In view of the previous theorem, we extend $G_{\varepsilon
}$ to \thinspace$L^{1}(I)$ by setting
\begin{equation}
G_{\varepsilon}(v):=\left\{
\begin{array}
[c]{ll}%
\int_{I}(W(v(t))+\varepsilon^{2}(v^{\prime}(t))^{2})\omega(t)\,dt & \text{if
}v\in H^{1}(I)\text{ satisfies \eqref{1d dirichlet}},\\
\infty & \text{otherwise in }L^{1}(I).
\end{array}
\right.  \label{1d functional}%
\end{equation}

\subsection{Zeroth and First-Order $\Gamma$-limit of $G_{\varepsilon}$}

We begin by establishing the zeroth-order $\Gamma$-limit of the functional
$G_{\varepsilon}$.

\begin{theorem}
\label{theorem 1d zero gamma} Assume that $W$ satisfies 
\eqref{W_Smooth}-\eqref{W' three zeroes}, that $\omega$ satisfies \eqref{etaSmooth}, and that $\alpha_{\varepsilon}\rightarrow\alpha$
and $\beta_{\varepsilon}\rightarrow\beta$ as $\varepsilon\rightarrow0^{+}$ for
some $\alpha,\beta\in\mathbb{R}$. Then the family $\{G_{\varepsilon
}\}_{\varepsilon}$ $\Gamma$-converges to $G^{(0)}$ in \thinspace$L^{1}(I)$ as
$\varepsilon\rightarrow0^{+}$, where
\[
G^{(0)}(v):=\int_{I}W(v(t))\omega(t)\,dt.
\]

\end{theorem}

\begin{proof}
To prove the liminf inequality, let $\varepsilon_{n}\rightarrow0^{+}$ and
$v_{n}\rightarrow v$ in \thinspace$L^{1}(I)$. Write $\alpha_{n}:=\alpha
_{\varepsilon_{n}}$ and $\beta_{n}:=\beta_{\varepsilon_{n}}$. Consider a
subsequence $\{\varepsilon_{n_{k}}\}_{k}$ of $\{\varepsilon_{n}\}_{n}$ such
that%
\[
\lim_{k\rightarrow\infty}G_{\varepsilon_{n_{k}}}(v_{n_{k}})=\liminf
_{n\rightarrow\infty}G_{\varepsilon_{n}}(v_{n}).
\]
Since $v_{n}\rightarrow v$ in \thinspace$L^{1}(I)$ and $\inf_{I}\omega>0$, by
selecting a further subsequence, not relabeled, we can assume that $v_{n_{k}%
}(t)\rightarrow v(t)$ for $\mathcal{L}^{1}$-a.e. $t\in I$. Hence, by Fatou's
lemma and the continuity and nonnegativity of $W$, we have%
\[
\lim_{k\rightarrow\infty}G_{\varepsilon_{n_{k}}}(v_{n_{k}})\geq\liminf
_{k\rightarrow\infty}\int_{I}W(v_{n_{k}}(t))\omega(t)\,dt\geq\int%
_{I}W(v(t))\omega(t)\,dt.
\]
To prove the limsup inequality, let $\varepsilon_{n}\rightarrow0^{+}$ and
$v\in$\thinspace$L^{1}(I)$. Assume first that $v$ is bounded. Let $\bar{v}$ be
a representative of $v$ and let $\varphi_{\delta}$ be a standard mollifier,
where $\delta>0$. Let $\delta_{n}\rightarrow0^{+}$ to be chosen later on and
define%
\[
\bar{v}_{n}(t):=\left\{
\begin{array}
[c]{ll}%
\alpha_{n} & \text{if }-1<t<2\delta_{n},\\
\bar{v}(t) & \text{if }2\delta_{n}\leq t\leq T-2\delta_{n},\\
\beta_{n} & \text{if }T-2\delta_{n}<t<T+1,
\end{array}
\right.
\]
and $v_{n}:=\varphi_{\delta_{n}}\ast\bar{v}_{n}$. Assuming that
$\operatorname*{supp}\varphi\subseteq(-1,1)$, we have that $v_{n}%
(0)=(\varphi_{\delta_{n}}\ast\alpha_{n})(0)=\alpha_{n}$ and $v_{n}%
(T)=(\varphi_{\delta_{n}}\ast\beta_{n})(T)=\beta_{n}$. On the the other hand,
if $0<t_{0}<1$ is a Lebesgue point of $\bar{v}$ then for all $n$ sufficiently
large, we have that $v_{n}(t_{0})=(\varphi_{\delta_{n}}\ast\bar{v}%
)(t_{0})\rightarrow\bar{v}(t_{0})$ by standard properties of mollifiers. Using
the continuity of $W$, we may apply the Lebesgue dominated
convergence theorem to obtain that $v_{n}\rightarrow v$ in \thinspace
$L^{1}(I)$ and
\[
\lim_{n\rightarrow\infty}\int_{I}W(v_{n}(t))\omega(t)\,dt=\int_{I}%
W(v(t))\omega(t)\,dt.
\]
On the other hand,%
\[
|v_{n}^{\prime}(t)|=|(\varphi_{\delta_{n}}^{\prime}\ast\bar{v}_{n}%
)(t)|\leq\frac{C}{\delta_{n}}.
\]
Hence,
\[
\int_{I}\varepsilon_{n}^{2}(v_{n}^{\prime}(t))^{2}\omega(t)\,dt\leq
C\frac{\varepsilon_{n}^{2}}{\delta_{n}}\int_{I}\omega(t)\,dt\rightarrow0
\]
provided we choose $\delta_{n}$ such that $\frac{\varepsilon_{n}^{2}}%
{\delta_{n}}\rightarrow0$. It follows that
\[
\lim_{n\rightarrow\infty}G_{\varepsilon_{n}}(v_{n})=\int_{I}W(v(t))\omega
(t)\,dt.
\]
To complete the proof, we use the fact that the $\Gamma$-lim sup is lower semicontinuous (see \cite[Proposition 6.8]{dalmaso-book1993}) and that every \thinspace$v \in L^{1}(I)$ can be approximated by an increasing sequence of bounded functions for which the convergence of the integral on the right-hand side above follows from the monotone convergence theorem. \hfill
\end{proof}

Since $W^{-1}(\{0\})=\{a,b\}$, we have
\[
\inf_{v\in\,L^{1}(I)}G^{(0)}(v)=0
\]
and therefore%
\begin{align}
G_{\varepsilon}^{(1)}(v)  &  :=\frac{G_{\varepsilon}(v)-\inf_{\,L^{1}%
(I)}G^{(0)}}{\varepsilon}\label{G1 epsilon}\\
&  =\int_{I}\left(  \frac{1}{\varepsilon}W(v(t))+\varepsilon(v^{\prime
}(t))^{2}\right)  \omega(t)\,dt\nonumber
\end{align}
if $v\in$\thinspace$H^{1}(I)$ satisfies \eqref{1d dirichlet} and
$G_{\varepsilon}^{(1)}(v):=\infty$ if $v\in$\thinspace$L^{1}(I)\backslash
$\thinspace$H^{1}(I)$ or if the boundary condition \eqref{1d dirichlet} fails.

We now characterize the first-order $\Gamma$-limit of the family $\{G_{\varepsilon}\}_{\varepsilon}$.

\begin{theorem}
\label{theorem 1d first gamma}Assume that $W$ satisfies
\eqref{W_Smooth}-\eqref{W' three zeroes}, that $\omega$ satisfies \eqref{etaSmooth}, and that $\alpha_{\varepsilon}\rightarrow\alpha$
and $\beta_{\varepsilon}\rightarrow\beta$ as $\varepsilon\rightarrow0^{+}$ for
some $\alpha,\beta\in\mathbb{R}$. Then the family $\{G_{\varepsilon}%
^{(1)}\}_{\varepsilon}$ $\Gamma$-converges to $G^{(1)}$ in \thinspace
$L^{1}(I)$ as $\varepsilon\rightarrow0^{+}$, where%
\begin{equation}
G^{(1)}(v):=\frac{C_{W}}{b-a}|Dv|_{\omega}(I)+\operatorname*{d}\nolimits_{W}%
(v(0),\alpha)\omega(0)+\operatorname*{d}\nolimits_{W}(v(T),\beta)\omega(T)
\label{G1}
\end{equation}
if $v\in BV_{\omega}(I;\{a,b\})$ 
and $G^{(1)}(v):=\infty$ otherwise in $L^{1}(I)$, where $\operatorname*{d}\nolimits_{W}$ and $C_{W}$ are defined in
\eqref{distance definition} and \eqref{cW definition}, respectively.
\end{theorem}

\begin{proof}
\textbf{Step 1:} To prove the $\Gamma$-liminf inequality, let $\varepsilon
_{n}\rightarrow0^{+}$ and $v_{n}\rightarrow v$ in \thinspace$L^{1}(I)$. Write
$\alpha_{n}:=\alpha_{\varepsilon_{n}}$ and $\beta_{n}:=\beta_{\varepsilon_{n}%
}$. Assume that%
\[
\liminf_{n\rightarrow\infty}G_{\varepsilon_{n}}^{(1)}(v_{n})<\infty,
\]
since otherwise there is nothing to prove,  and consider a subsequence
$\{\varepsilon_{n_{k}}\}_{k}$ of $\{\varepsilon_{n}\}_{n}$ such that%
\[
\lim_{k\rightarrow\infty}G_{\varepsilon_{n_{k}}}^{(1)}(v_{n_{k}}%
)=\liminf_{n\rightarrow\infty}G_{\varepsilon_{n}}^{(1)}(v_{n}).
\]
Then for all $k$ sufficiently large, $v_{n_{k}}\in$\thinspace$H^{1}(I)$,
$v_{n_{k}}(0)=\alpha_{n_{k}}$, and $v_{n_{k}}(T)=\beta_{n_{k}}$. Extend
$\omega$ and $v_{n_{k}}$ to $(-1,T+1)$, by setting
\[
\bar{\omega}(t):=\left\{
\begin{array}
[c]{ll}%
\omega(0) & \text{if }t\leq0,\\
\omega(t) & \text{if }t\in I,\\
\omega(T) & \text{if }t\geq T,
\end{array}
\right.  \quad\bar{v}_{n_{k}}(t):=\left\{
\begin{array}
[c]{ll}%
\alpha_{n_{k}} & \text{if }t\leq0,\\
v_{n_{k}}(t) & \text{if }t\in I,\\
\beta_{n_{k}} & \text{if }t\geq T.
\end{array}
\right.
\]
Define
\[
W_{1}(s):=\min\{W(s),K\},\quad\Phi_{1}(s):=\int_{a}^{s}2W_{1}^{1/2}%
(\rho)\,d\rho,
\]
where
\[
K:=\max_{J}W
\]
and $J$ is the smallest closed interval that contains $[a,b]$, $\{\alpha
_{n_{k}}\}_{k}$, and $\{\beta_{n_{k}}\}_{k}$. Then
\begin{align*}
G_{\varepsilon_{n_{k}}}^{(1)}(v_{n_{k}})  &  =\int_{I}\left(  \frac
{1}{\varepsilon_{n_{k}}}W(v_{n_{k}}(t))+\varepsilon_{n_{k}}(v_{n_{k}}^{\prime
}(t))^{2}\right)  \omega(t)\,dt\\
&  \geq\int_{I}2W_{1}^{1/2}(v_{n_{k}}(t))|v_{n_{k}}^{\prime}(t)|\omega
(t)\,dt=\int_{I}|(\Phi_{1}\circ v_{n_{k}})^{\prime}(t)|\omega(t)\,dt\\
&  =\int_{-1}^{T+1}|(\Phi_{1}\circ\bar{v}_{n_{k}})^{\prime}(t)|\bar{\omega
}(t)\,dt,
\end{align*}
where in the last equality we used the fact that $(\Phi_{1}\circ\bar{v}%
_{n_{k}})(t)\equiv\Phi(\alpha_{n_{k}})$ in $(-1,0)$ and $(\Phi_{1}\circ\bar
{v}_{n_{k}})(t)\equiv\Phi(\beta_{n_{k}})$ in $(T,T+1)$ . Since $\Phi_{1}$ is
Lipschitz continuous, we have that $\Phi_{1}\circ\bar{v}_{n_{k}}%
\rightarrow\Phi_{1}\circ\bar{v}$ in \thinspace$L^{1}((-1,T+1))$, where
\[
\bar{v}(t):=\left\{
\begin{array}
[c]{ll}%
\alpha & \text{if }t\leq0,\\
v(t) & \text{if }t\in I,\\
\beta & \text{if }t\geq T.
\end{array}
\right.
\]
Hence, by standard lower semicontinuity results,%
\begin{align*}
\lim_{k\rightarrow\infty}G_{\varepsilon_{n_{k}}}^{(1)}(v_{n_{k}})  &
\geq\liminf_{k\rightarrow\infty}\int_{-1}^{T+1}|(\Phi_{1}\circ\bar{v}_{n_{k}%
})^{\prime}(t)|\bar{\omega}(t)\,dt\\
&  \geq\int_{-1}^{T+1}\bar{\omega}\,d|D(\Phi_{1}\circ\bar{v})|=\int_{-1}%
^{T+1}\bar{\omega}\,d|D(\Phi\circ\bar{v})|\\
&  =\int_{I}\omega\,d|D(\Phi\circ v)|+\operatorname*{d}\nolimits_{W}%
(\alpha,v(0))\omega(0)+\operatorname*{d}\nolimits_{W}(\beta,v(T))\omega(T)\\
&  =\frac{C_{W}}{b-a}\int_{I}\omega\,d|Dv|+\operatorname*{d}\nolimits_{W}%
(\alpha,v(0))\omega(0)+\operatorname*{d}\nolimits_{W}(\beta,v(T))\omega(T)\\
&  =G^{(1)}(v).
\end{align*}
\textbf{Step 2:} To prove the $\Gamma$-limsup inequality, assume first that $v$
is of the form
\[
v(t)=%
\begin{cases}
a & \text{ if }t\in\lbrack t_{2k},t_{2k+1}),\\
b & \text{ otherwise,}%
\end{cases}
\]
where $0=t_{0}<t_{1}<\cdots<t_{2\ell}=T$. Observe that
\begin{equation}
v(t)=\operatorname*{sgn}\nolimits_{a,b}(f(t)), \label{1d 98}%
\end{equation}
where
\begin{equation}
\operatorname*{sgn}\nolimits_{a,b}(t):=\left\{
\begin{array}
[c]{ll}%
a & \text{ if }t\leq0,\\
b & \text{ if }t>0,
\end{array}
\right.  \label{sgnabDefinition}%
\end{equation}
and%
\begin{equation}
f(t):=%
\begin{cases}
t-t_{1} & \text{ if }t\in\lbrack t_{0},t_{1}),\\
-\min\{t-t_{2k},t_{2k+1}-t\} & \text{ if }t\in\lbrack t_{2k},t_{2k+1}),\text{
and }k\geq1,\\
\min\{t-t_{2k+1},t_{2k+2}-t\} & \text{ if }t\in(t_{2k+1},t_{2k+2}],\text{ and
}k<\ell-1,\\
t-t_{2\ell-1} & \text{ if }t\in\lbrack t_{2\ell-1},t_{2\ell})
\end{cases}
\label{1d 99}%
\end{equation}
is the signed distance function of the set $E:=\{t\in I:\,v(t)=a\}$ relative
to $I$. We will construct smooth approximations of the function
$\operatorname*{sgn}\nolimits_{a,b}$ that almost minimize the energy
$G_{\varepsilon}^{(1)}$.

Since we expect each transition to happen in an infinitesimal interval and
$\omega(t)\sim\omega(t_{k})$ for $t$ close to $t_{k}$, to construct an
approximate solution, we consider minimizers of the functional $\int%
_{0}^{T_{\varepsilon}}\left(  \frac{1}{\varepsilon}W(\phi)+\varepsilon
(\phi^{\prime})^{2}\right)  \,dt$ with appropriate boundary conditions and
where $T_{\varepsilon}\rightarrow0^{+}$. The minimizers of this functional
satisfy the Euler--Lagrange equations $2\varepsilon^{2}\phi^{\prime\prime
}=W^{\prime}(\phi)$. If we multiply each side by $\phi^{\prime}$ and
integrate, we get%
\[
\varepsilon^{2}(\phi^{\prime}(t))^{2}=\varepsilon^{2}(\phi^{\prime}%
(0))^{2}-W(\phi(0))+W(\phi(t)).
\]
Thus,
\[
\varepsilon\phi^{\prime}(t)=\pm(\varepsilon^{2}(\phi^{\prime}(0))^{2}%
-W(\phi(0))+W(\phi(t)))^{1/2}.
\]
Setting $\delta_{\varepsilon}:=\varepsilon^{2}(\phi^{\prime}(0))^{2}%
-W(\phi(0))$ we solve the differential equation
\[
\varepsilon\phi^{\prime}(t)=\pm(\delta_{\varepsilon}+W(\phi(t)))^{1/2},
\]
where we determine the sign according to each transition. Provided
$\delta_{\varepsilon}+W(\phi(t))\neq0$, this gives%
\[
\pm\int_{\phi(0)}^{\phi(t)}\frac{\varepsilon}{(\delta_{\varepsilon}%
+W(\rho))^{1/2}}\,d\rho=t.
\]
We consider first the transition from $\alpha_{\varepsilon}$ to $a$. Let $\delta_{\varepsilon}>0$ and introduce the function%
\begin{equation}
\bar{\Psi}_{\varepsilon}(r):=\int_{r}^{\alpha_{\varepsilon}}\frac{\varepsilon
}{(\delta_{\varepsilon}+W(\rho))^{1/2}}\,d\rho. \label{1d 99a}%
\end{equation}
Define
\[
\bar{T}_{\varepsilon}:=\bar{\Psi}_{\varepsilon}(a)
\]
and observe that since $W\geq0$,
\begin{equation}
0<\bar{T}_{\varepsilon}=\int_{a}^{\alpha_{\varepsilon}}\frac{\varepsilon
}{(\delta_{\varepsilon}+W(\rho))^{1/2}}\,d\rho\leq(b-a)\frac{\varepsilon
}{\delta_{\varepsilon}^{1/2}}. \label{1d 100}%
\end{equation}
The function $\bar{\Psi}_{\varepsilon}$ is strictly decreasing and
differentiable. Let $\bar{\phi}_{\varepsilon}:[0,\bar{T}_{\varepsilon
}]\rightarrow\lbrack a,\alpha_{\varepsilon}]$ be the inverse of $\bar{\Psi
}_{\varepsilon}$ on the interval $[a,\alpha_{\varepsilon}]$. Then $\bar{\phi
}_{\varepsilon}\left(  0\right)  =\alpha_{\varepsilon}$, $\bar{\phi
}_{\varepsilon}(\bar{T}_{\varepsilon})=a$, and
\begin{equation}
\bar{\phi}_{\varepsilon}^{\prime}\left(  s\right)  =\frac{1}{\bar{\Psi
}_{\varepsilon}^{\prime}(\bar{\phi}_{\varepsilon}\left(  s\right)  )}%
=-\frac{(\delta_{\varepsilon}+W(\bar{\phi}_{\varepsilon}\left(  s\right)
))^{1/2}}{\varepsilon}. \label{1d 100a}%
\end{equation}
Extend $\bar{\phi}_{\varepsilon}$ to be $a$ for $t>\bar{T}_{\varepsilon}$.

Similarly, to transition from $a$ to $b$, we define
\[
\Psi_{\varepsilon}(r):=\int_{a}^{r}\frac{\varepsilon}{(\delta_{\varepsilon
}+W(\rho))^{1/2}}\,d\rho,
\]
and
\begin{equation}
0\leq T_{\varepsilon}:=\Psi_{\varepsilon}(b)\leq(b-a)\frac{\varepsilon}%
{\delta_{\varepsilon}^{1/2}}. \label{1d 101}%
\end{equation}
Let $\phi_{\varepsilon}:[0,T_{\varepsilon}]\rightarrow\lbrack a,b]$ be the
inverse of $\Psi_{\varepsilon}$ on the interval $[a,b]$. Then $\phi
_{\varepsilon}\left(  0\right)  =a$, $\phi_{\varepsilon}(T_{\varepsilon})=b$,
and
\begin{equation}
\phi_{\varepsilon}^{\prime}(s)=\frac{(\delta_{\varepsilon}+W(\phi
_{\varepsilon}\left(  s\right)  ))^{1/2}}{\varepsilon}. \label{1d 101a}%
\end{equation}
Extend $\phi_{\varepsilon}$ to be equal to $a$ for $s<0$ and $b$ for
$s>T_{\varepsilon}$.

Finally, to transition from $\beta_{\varepsilon}$ to $b$, define%
\[
\hat{\Psi}_{\varepsilon}(r):=\int_{\beta_{\varepsilon}}^{r}\frac{\varepsilon}%
{(\delta_{\varepsilon}+W(\rho))^{1/2}}\,d\rho,
\]
and
\begin{equation}
0\leq\hat{T}_{\varepsilon}:=\hat{\Psi}_{\varepsilon}(b)\leq(b-a)\frac
{\varepsilon}{\delta_{\varepsilon}^{1/2}}. \label{1d 102}%
\end{equation}
Let $\hat{\phi}_{\varepsilon}:[0,\hat{T}_{\varepsilon}]\rightarrow\lbrack
\beta_{\varepsilon},b]$ be the inverse of $\hat{\Psi}_{\varepsilon}$ on the interval
$[\beta_{\varepsilon},b]$. Then $\hat{\phi}_{\varepsilon}\left(  0\right)  =b$,
$\hat{\phi}_{\varepsilon}(\hat{T}_{\varepsilon})=\beta_{\varepsilon}$, and
\begin{equation}
\hat{\phi}_{\varepsilon}^{\prime}(s)=-\frac{(\delta_{\varepsilon}+W(\hat{\phi
}_{\varepsilon}\left(  s\right)  ))^{1/2}}{\varepsilon}. \label{1d 103}%
\end{equation}
Extend $\hat{\phi}_{\varepsilon}$ to be equal to $b$ for $s<0$. Assume that%
\begin{equation}
\delta_{\varepsilon}\rightarrow0,\quad\frac{\varepsilon}{\delta_{\varepsilon
}^{1/2}}\rightarrow0. \label{1d delta epsilon choice}%
\end{equation}
Taking $\varepsilon$ be so small that transition layers do not overlap or
leave $\overline{I}$, we can define%
\[
v_{\varepsilon}(t):=\left\{
\begin{array}
[c]{ll}%
\bar{\phi}_{\varepsilon}\left(  t\right)  & \text{if }0<t<\bar{T}%
_{\varepsilon},\\
\phi_{\varepsilon}\left(  f(t)\right)  & \text{if }\bar{T}_{\varepsilon}\leq
t\leq T-\hat{T}_{\varepsilon},\\
\hat{\phi}_{\varepsilon}(t-T+\hat{T}_{\varepsilon}) & \text{if }T-\hat
{T}_{\varepsilon}<t<T.
\end{array}
\right.
\]
Since $\bar{T}_{\varepsilon}\rightarrow0$, $T_{\varepsilon}\rightarrow0$, and
$\hat{T}_{\varepsilon}\rightarrow0$\ as $\varepsilon\rightarrow0^{+}$ by
(\ref{1d 100}), (\ref{1d 101}), and (\ref{1d 102}), respectively, in view of
(\ref{1d 98}), (\ref{1d 99}), we have that $v_{\varepsilon}(t)\rightarrow
v(t)$ for all $t\in I$. Moreover, $|v_{\varepsilon}(t)|\leq C$ for all $t\in
I$, all $\varepsilon>0$, and for some constant $C>0$. Hence, by the Lebesgue
dominated convergence theorem, $v_{\varepsilon}\rightarrow v$ in
\thinspace$L^{1}(I)$.

By (\ref{1d 100a}) and the change of variables $\rho=\bar{\phi}_{\varepsilon
}(t)$, we have%
\begin{align*}
&  \int_{0}^{\bar{T}_{\varepsilon}}\left(  \frac{1}{\varepsilon}%
W(v_{\varepsilon})+\varepsilon(v_{\varepsilon}^{\prime})^{2}\right)
\omega\,dt=\int_{0}^{\bar{T}_{\varepsilon}}\left(  \frac{1}{\varepsilon}%
W(\bar{\phi}_{\varepsilon})+\varepsilon(\bar{\phi}_{\varepsilon}^{\prime}%
)^{2}\right)  \omega\,dt\\
&  \leq\max_{\lbrack0,\bar{T}_{\varepsilon}]}\omega\int_{0}^{\bar
{T}_{\varepsilon}}\left(  \frac{1}{\varepsilon}(\delta_{\varepsilon}%
+W(\bar{\phi}_{\varepsilon}))+\varepsilon(\bar{\phi}_{\varepsilon}^{\prime
})^{2}\right)  \,dt\\
&  =\max_{[0,\bar{T}_{\varepsilon}]}\omega\int_{0}^{\bar{T}_{\varepsilon}%
}2(\delta_{\varepsilon}+W(\bar{\phi}_{\varepsilon}))^{1/2}|\bar{\phi
}_{\varepsilon}^{\prime}|\,dt=\max_{[0,\bar{T}_{\varepsilon}]}\omega\int%
_{a}^{\alpha_{\varepsilon}}2(\delta_{\varepsilon}+W(\rho))^{1/2}\,d\rho.
\end{align*}
On the other hand, by (\ref{1d 103}) and the changes of variables $t-T+\hat
{T}_{\varepsilon}=s$, $\rho=\hat{\phi}_{\varepsilon}(s)$,%
\begin{align*}
&  \int_{T-\hat{T}_{\varepsilon}}^{T}\left(  \frac{1}{\varepsilon
}W(v_{\varepsilon})+\varepsilon(v_{\varepsilon}^{\prime})^{2}\right)
\omega\,dt=\int_{0}^{\hat{T}_\varepsilon}\left(  \frac{1}{\varepsilon
}W(\hat{\phi}_{\varepsilon})+\varepsilon(\hat{\phi}_{\varepsilon}^{\prime
})^{2}\right)  \omega(T-\hat{T}_{\varepsilon}+s)\,ds\\
&  \leq \max_{\lbrack T-\hat{T}_{\varepsilon},T]}\omega\int_{0}^{\hat
{T}_{\varepsilon}} \left(  \frac{1}{\varepsilon}(\delta_{\varepsilon}%
+W(\hat{\phi}_{\varepsilon}))+\varepsilon(\hat{\phi}_{\varepsilon}^{\prime
})^{2}\right)  \,dt\\
&  =\max_{\lbrack T-\hat{T}_{\varepsilon},T]}\omega  \int_{0}^{\hat
{T}_{\varepsilon}}2(\delta_{\varepsilon}+W(\hat{\phi}_{\varepsilon}%
))^{1/2}|\hat{\phi}_{\varepsilon}^{\prime}|\,dt=\max_{\lbrack T-\hat{T}_{\varepsilon},T]}\omega \int_{\beta_{\varepsilon}}^{b}2(\delta_{\varepsilon}%
+W(\rho))^{1/2}d\rho.
\end{align*}
Similarly, by (\ref{1d 101a}),
\begin{align*}
&  \sum_{k=1}^{2\ell-1}\int_{t_{k-1}}^{t_{k}}\left(  \frac{1}{\varepsilon
}W(v_{\varepsilon})+\varepsilon(v_{\varepsilon}^{\prime})^{2}\right)
\omega\,dt\\
&  =\sum_{k=1}^{2\ell-1}\int_{0}^{T_{\varepsilon}}\left(  \varepsilon
(\phi_{\varepsilon}^{\prime}(s))^{2}+\varepsilon^{-1}W(\phi_{\varepsilon
}(s))\right)  \omega(t_{k}+(-1)^{k+1}s)\,ds\\
&  \leq\sum_{k=1}^{2\ell-1}\int_{0}^{T_{\varepsilon}}2(\delta_{\varepsilon
}+W(\phi_{\varepsilon}(s)))^{1/2}\phi_{\varepsilon}^{\prime}(s)\omega
(t_{k}+(-1)^{k+1}s)\,ds\\
&  \leq\sum_{k=1}^{2\ell-1}\sup\{\omega(t_{k}+(-1)^{k+1}r):\,r\in
(0,T_{\varepsilon})\}\int_{0}^{T_{\varepsilon}}2(\delta_{\varepsilon}%
+W(\phi_{\varepsilon}(s)))^{1/2}\phi_{\varepsilon}^{\prime}(s)\,ds\\
&  =\sum_{k=1}^{2\ell-1}\sup\{\omega(t_{k}+(-1)^{k+1}r):\,r\in
(0,T_{\varepsilon})\}\int_{a}^{b}2(\delta_{\varepsilon}+W(s))^{1/2}ds.
\end{align*}
Thus taking the limit as $\varepsilon\rightarrow0^{+}$ we find that
\begin{align*}
\limsup_{\varepsilon\rightarrow0^{+}}G_{\varepsilon}^{(1)}(v_{\varepsilon})
&  \leq\int_{a}^{\alpha}2W^{1/2}(s)\,ds\omega(0)+\int_{\beta}^{b}%
2W^{1/2}(s)ds\omega(T)+C_{W}\sum_{k=1}^{2\ell-1}\omega(t_{k})\\
&  =G^{(1)}(v).
\end{align*}
The cases where $v$ starts or ends at values different from those we assumed above are treated analogously. \hfill
\end{proof}

Next, we show that if $\omega$ is sufficiently close to $\omega(0)$,
$a<\alpha<b$, and $\beta=b$, then the unique minimizer of $G^{(1)}$ is the
constant function $b$. We recall that $\alpha$ and $\beta$ appear in the definition of $G^{(1)}$ (see \eqref{G1}).

\begin{corollary}
\label{corollary 1d minimizer}Assume that $W$ satisfies
\eqref{W_Smooth}-\eqref{W' three zeroes} and let $a<\alpha<b$
and\ $\beta=b$. Suppose that $\omega$ satisfies \eqref{etaSmooth}
and that
\begin{equation}
\omega(t)>\omega(0)-\omega_{0}\quad\text{for all }t\in(0,T],
\label{eta close to eta0}%
\end{equation}
where
\begin{equation}
0\leq\omega_{0}<\frac{1}{2}\frac{C_{W}-\operatorname*{d}\nolimits_{W}%
(\alpha,b)}{C_{W}}\omega(0). \label{eta0}%
\end{equation}
Then the unique minimizer of $G^{(1)}$ is the constant function $b$, with
\[
\min_{\,L^{1}(I)}G^{(1)}(v)=G^{(1)}(b)=\operatorname*{d}\nolimits_{W}%
(\alpha,b)\omega(0).
\]
\end{corollary}

\begin{proof}
Let $v\in BV_{\omega}(I;\{a,b\})$. If $v$ has at least one jump point at
$t_{0}\in I$, then by (\ref{eta close to eta0}) and (\ref{eta0}),%
\[
G^{(1)}(v)\geq\frac{C_{W}}{b-a}|Dv|_{\omega}(I)\geq C_{W}\omega(t_{0}%
)>C_{W}(\omega(0)-\omega_{0})\geq\operatorname*{d}\nolimits_{W}(\alpha
,b)\omega(0).
\]
Hence, either $v\equiv b$ or $v\equiv a$. If $v=a$, then again by
(\ref{eta close to eta0}) and (\ref{eta0})
\[
G^{(1)}(a)=\operatorname*{d}\nolimits_{W}(a,\alpha)\omega(0)+C_{W}%
\omega(T)>C_{W}(\omega(0)-\omega_{0})\geq\operatorname*{d}\nolimits_{W}%
(\alpha,b)\omega(0).
\]
This completes the proof. \hfill
\end{proof}

\begin{remark}
Note that condition \eqref{eta close to eta0} holds if either $\omega$ is
strictly increasing, with $\omega_{0}=0$, or if $T$ is sufficiently small, by
continuity of $\omega$.
\end{remark}

\subsection{Second-Order $\Gamma$-limsup}

Under the hypotheses of Corollary \ref{corollary 1d minimizer}, we have
\[
\min_{\,L^{1}(I)}G^{(1)}(v)=G^{(1)}(b)=\operatorname*{d}\nolimits_{W}%
(\alpha,b)\omega(0).
\]
We define%
\begin{align}
G_{\varepsilon}^{(2)}(v)  &  :=\frac{G_{\varepsilon}^{(1)}(v)-\inf
_{\,L^{1}(I)}G^{(1)}}{\varepsilon}\label{G 2 epsilon}\\
&  =\int_{I}\left(  \frac{1}{\varepsilon^{2}}W(v(t))+(v^{\prime}%
(t))^{2}\right)  \omega(t)\,dt-\operatorname*{d}\nolimits_{W}(\alpha
,b)\omega(0)\frac{1}{\varepsilon}\nonumber
\end{align}
if $v\in$\thinspace$H^{1}(I)$ satisfies \eqref{1d dirichlet} and
$G_{\varepsilon}^{(2)}(v):=\infty$ if $v\in$\thinspace$L^{1}(I)\backslash
$\thinspace$H^{1}(I)$ or if the boundary condition \eqref{1d dirichlet} fails.

We study the second-order $\Gamma$-limsup of the family $\{G_{\varepsilon
}\}_{\varepsilon}$.

\begin{theorem}
[Second-Order $\Gamma$-Limsup]\label{theorem 1d second gamma limsup}Assume that $W$
satisfies \eqref{W_Smooth}-\eqref{W' three zeroes}, that
$\alpha_{-}$ satisfies \eqref{alpha and beta minus}, and that
$\omega$ satisfies \eqref{etaSmooth}, \eqref{eta close to eta0},
where%
\begin{equation}
0\leq\omega_{0}<\frac{1}{2}\frac{\operatorname*{d}\nolimits_{W}(a,\alpha_{-}%
)}{C_{W}}\omega(0).\label{eta 0 alpha-}%
\end{equation}
Let
\begin{equation}
\alpha_{-}\leq\alpha_{\varepsilon},\,\beta_{\varepsilon}\leq
b,\label{1d alpha epsilon beta epsilon}%
\end{equation}
with
\begin{equation}
|\alpha_{\varepsilon}-\alpha|\leq A_{0}\varepsilon^{\gamma},\quad
|\beta_{\varepsilon}-b|\leq B_{0}\varepsilon^{\gamma}%
\label{alpha epsilon and beta epsilon}%
\end{equation}
for some $\alpha$,$\,\beta$ and where $A_{0}$, $B_{0}>0$, and $\gamma>1$. Then
there exist constants $0<\varepsilon_{0}<1$, $C,C_{0}>0$, and $\gamma
_{0},\gamma_{1}>0$, depending only on $\alpha_{-}$, $A_{0}$, $B_{0}$, $T$,
$\omega$, and $W$, and functions $v_{\varepsilon}\in$\thinspace$H^{1}(I)$
satisfying \eqref{1d dirichlet}, $a\leq v_{\varepsilon}\leq b$, and
$v_{\varepsilon}\rightarrow b$ in \thinspace$L^{1}(I)$, such that
\begin{equation}
G_{\varepsilon}^{(2)}(v_{\varepsilon})\leq\int_{0}^{l}2W(p_{\varepsilon
}(t))t\,dt\,\omega^{\prime}(0)+Ce^{-2\sigma l}\left(  2\sigma l+1\right)
+C\varepsilon^{2\gamma}l+C\varepsilon^{\gamma_{1}}|\log\varepsilon
|^{\gamma_{0}}\label{estimate limsup G1}%
\end{equation}
for all $0<\varepsilon<\varepsilon_{0}$ and all $l>0$, where $p_{\varepsilon
}(t)=v_{\varepsilon}(\varepsilon t)$ is such that $p_{\varepsilon}\rightarrow
z_{\alpha}$ pointwise in $[0,\infty)$, and where $z_{\alpha}$ solves the Cauchy
problem \eqref{cauchy problem z alpha}. In particular,%
\begin{equation}
\limsup_{\varepsilon\rightarrow0^{+}}G_{\varepsilon}^{(2)}(v_{\varepsilon
})\leq\int_{0}^{\infty}2W^{1/2}(z_{\alpha}(t))z_{\alpha}^{\prime
}(t)t\,dt\,\omega^{\prime}(0).\label{limsup G2}%
\end{equation}

\end{theorem}

\begin{proof}
Let $\delta_{\varepsilon}\rightarrow0^{+}$ as $\varepsilon\rightarrow0^{+}$,
and define
\begin{equation}
\Psi_{\varepsilon}(r):=\left\{
\begin{array}
[c]{cc}%
\int_{\alpha_{\varepsilon}}^{r}\frac{\varepsilon}{(\delta_{\varepsilon
}+W(s))^{1/2}}\,ds & \text{if }\alpha_{\varepsilon}\leq\beta_{\varepsilon},\\
\int_{r}^{\alpha_{\varepsilon}}\frac{\varepsilon}{(\delta_{\varepsilon
}+W(s))^{1/2}}\,ds & \text{if }\beta_{\varepsilon}<\alpha_{\varepsilon},
\end{array}
\right.  \label{1d 554}%
\end{equation}
and
\begin{equation}
0\leq T_{\varepsilon}:=\Psi_{\varepsilon}(\beta_{\varepsilon}). \label{1d 555}%
\end{equation}
Since $\alpha_{-}\leq\alpha_{\varepsilon},\,\beta_{\varepsilon}\leq b$, by
Proposition \ref{proposition asymptotic behavior} we have%
\begin{align*}
T_{\varepsilon}  &  \leq\int_{\alpha_{-}}^{b}\frac{\varepsilon}{(\delta
_{\varepsilon}+W(s))^{1/2}}\,ds\\
&  \leq-\frac{\sigma}{2}\varepsilon\log(\sigma^{2}\delta_{\varepsilon}%
)+\sigma\varepsilon\log(1+2(b-a)).
\end{align*}
Since $\delta_{\varepsilon}\rightarrow0^{+}$, there exist $C_{0}>0$ and
$\varepsilon_{0}>0$, depending only on $W$, $\gamma$, and $B$, such that
\begin{equation}
T_{\varepsilon}\leq C_{0}\varepsilon|\log\delta_{\varepsilon}|
\label{1d T epsilon above}%
\end{equation}
for all $0<\varepsilon<\varepsilon_{0}$.

On the other hand, if $\alpha<b$, by Proposition
\ref{proposition asymptotic behavior}, and
\eqref{alpha epsilon and beta epsilon}, we obtain%
\begin{align*}
T_{\varepsilon}  &  \geq-\varepsilon\log(\sigma^{-2}\delta_{\varepsilon
})+\sigma^{-1}\varepsilon\log(b-\alpha-A\varepsilon^{\gamma})\\
&  \quad-\sigma^{-1}\varepsilon\log(1+2\sigma^{-1}B\varepsilon^{\gamma}%
/\delta_{\varepsilon}^{1/2})
\end{align*}
for all $0<\varepsilon<\varepsilon_{\alpha}$. Since $\delta_{\varepsilon
}\rightarrow0^{+}$, by taking $\varepsilon_{\alpha}$ smaller if necessary, we can find
$C_{\alpha}>0$ such that
\begin{equation}
C_{\alpha}\varepsilon|\log\delta_{\varepsilon}|\leq T_{\varepsilon}%
\quad\text{if }\alpha<b \label{1d T epsilon below a}%
\end{equation}
for all $0<\varepsilon<\varepsilon_{\alpha}$.

Let $v_{\varepsilon}:[0,T_{\varepsilon}]\rightarrow\lbrack\alpha_{\varepsilon
},\beta_{\varepsilon}]$ be the inverse of $\Psi_{\varepsilon}$. Then
$v_{\varepsilon}\left(  0\right)  =\alpha_{\varepsilon}$, $v_{\varepsilon
}(T_{\varepsilon})=\beta_{\varepsilon}$, and
\begin{equation}
v_{\varepsilon}^{\prime}(t)=\pm\frac{(\delta_{\varepsilon}+W(v_{\varepsilon
}\left(  t\right)  ))^{1/2}}{\varepsilon}, \label{1d 556}%
\end{equation}
where we take the plus sign if $\alpha_{\varepsilon}\leq\beta_{\varepsilon}$
and the minus sign if $\beta_{\varepsilon}<\alpha_{\varepsilon}$. Extend
$v_{\varepsilon}$ to be equal to $\beta_{\varepsilon}$ for $t>T_{\varepsilon}$.

Since $\omega\in C^{1,d}(I)$, by Taylor's formula, for $t\in\lbrack0,T]$ we have
\[
\omega(t)=\omega(0)+\omega^{\prime}(0)t+R_{1}(t),
\]
where%
\begin{equation}
|R_{1}(t)|=|\omega^{\prime}(\theta t)-\omega^{\prime}(0)|t\leq|\omega^{\prime
}|_{C^{0,d}}t^{1+d}. \label{1d 103a}%
\end{equation}
Write%
\begin{align}
G_{\varepsilon}^{(2)}(v_{\varepsilon})  &  =\left[  \int_{0}^{T_{\varepsilon}%
}\left(  \frac{1}{\varepsilon}W(v_{\varepsilon})+\varepsilon(v_{\varepsilon
}^{\prime})^{2}\right)  \,dt-\operatorname*{d}\nolimits_{W}(b,\alpha)\right]
\frac{\omega(0)}{\varepsilon}\nonumber\\
&  \quad+\int_{0}^{T_{\varepsilon}}\left(  \frac{1}{\varepsilon}%
W(v_{\varepsilon})+\varepsilon(v_{\varepsilon}^{\prime})^{2}\right)
t\,dt\frac{\omega^{\prime}(0)}{\varepsilon}\label{1d 557}\\
&  \quad+\int_{0}^{T_{\varepsilon}}\left(  \frac{1}{\varepsilon}%
W(v_{\varepsilon})+\varepsilon(v_{\varepsilon}^{\prime})^{2}\right)
R_{1}\,dt\frac{1}{\varepsilon}\nonumber\\
&  \quad+\int_{T_{\varepsilon}}^{T}\left(  \frac{1}{\varepsilon}%
W(v_{\varepsilon})+\varepsilon(v_{\varepsilon}^{\prime})^{2}\right)
\omega\,dt\frac{1}{\varepsilon}=:\mathcal{A}+\mathcal{B}+\mathcal{C}%
+\mathcal{D}.\nonumber
\end{align}
\textbf{Step 1:} We estimate $\mathcal{A}$ in (\ref{1d 557}). Assume first
that $\alpha_{\varepsilon}\leq\beta_{\varepsilon}$. By (\ref{1d 554}),
(\ref{1d 555}), (\ref{1d 556}), the change of variables $s=v_{\varepsilon}%
(t)$, the fact that $a<\alpha_{\varepsilon},\,\beta_{\varepsilon}\leq b$, and
the equality
\[
(A+B)^{1/2}-B^{1/2}=\frac{A}{(A+B)^{1/2}+B^{1/2}},
\]
we have%
\begin{align*}
\int_{0}^{T_{\varepsilon}}  &  \left(  \frac{1}{\varepsilon}W(v_{\varepsilon
})+\varepsilon(v_{\varepsilon}^{\prime})^{2}\right)  \,dt=\int_{0}%
^{T_{\varepsilon}}\left(  \frac{1}{\varepsilon}(\delta_{\varepsilon
}+W(v_{\varepsilon}))+\varepsilon(v_{\varepsilon}^{\prime})^{2}\right)
\,dt-\frac{T_{\varepsilon}\delta_{\varepsilon}}{\varepsilon}\\
&  =\int_{0}^{T_{\varepsilon}}2(\delta_{\varepsilon}+W(v_{\varepsilon}%
))^{1/2}v_{\varepsilon}^{\prime}\,dt-\frac{T_{\varepsilon}\delta_{\varepsilon
}}{\varepsilon}\\
&  =\int_{\alpha_{\varepsilon}}^{\beta_{\varepsilon}}2(\delta_{\varepsilon
}+W(s))^{1/2}\,ds-\int_{\alpha_{\varepsilon}}^{\beta_{\varepsilon}}%
\frac{\delta_{\varepsilon}}{(\delta_{\varepsilon}+W(s))^{1/2}}\,ds\\
&  =\int_{\alpha_{\varepsilon}}^{\beta_{\varepsilon}}2W^{1/2}(s)\,ds+\int%
_{\alpha_{\varepsilon}}^{\beta_{\varepsilon}}\left[  \frac{2\delta
_{\varepsilon}}{(\delta_{\varepsilon}+W(s))^{1/2}+W^{1/2}(s)}-\frac
{\delta_{\varepsilon}}{(\delta_{\varepsilon}+W(s))^{1/2}}\right]  \,ds.
\end{align*}
By Proposition \ref{proposition difference},%
\[
\int_{\alpha_{\varepsilon}}^{\beta_{\varepsilon}}\left[  \frac{2\delta
_{\varepsilon}}{(\delta_{\varepsilon}+W(s))^{1/2}+W^{1/2}(s)}-\frac
{\delta_{\varepsilon}}{(\delta_{\varepsilon}+W(s))^{1/2}}\right]  \,ds\leq
C\delta_{\varepsilon}%
\]
for all $0<\varepsilon<\varepsilon_{0}$, while, since $a\leq\alpha_{\varepsilon},$ $\beta_{\varepsilon}\leq b$,%
\[
\int_{\alpha_{\varepsilon}}^{\beta_{\varepsilon}}2W^{1/2}(s)\,ds\leq
\int_{\alpha}^{b}2W^{1/2}(s).
\]
Hence, we obtain%
\begin{equation}
\int_{0}^{T_{\varepsilon}}\left(  \frac{1}{\varepsilon}W(v_{\varepsilon
})+\varepsilon(v_{\varepsilon}^{\prime})^{2}\right)  \,dt\leq\int_{\alpha}%
^{b}2W^{1/2}(s)\,ds+C\delta_{\varepsilon}.
\label{1d 200}%
\end{equation}
The case $\beta_{\varepsilon}<\alpha_{\varepsilon}$ is similar. We omit the details.

It follows from (\ref{1d 557}) that%
\[
\mathcal{A}\leq C\frac{\delta_{\varepsilon}}{\varepsilon}%
\]
for all $0<\varepsilon<\varepsilon_{0}$.

\noindent\textbf{Step 2:} We estimate $\mathcal{B}$ in (\ref{1d 557}). By
(\ref{1d 556}), (\ref{1d 557}), and the change of variables $t=\varepsilon s$
for $l>0$,
\begin{align*}
\mathcal{B}  &  =\int_{0}^{T_{\varepsilon}}\frac{2}{\varepsilon^{2}%
}W(v_{\varepsilon})t\,dt\,\omega^{\prime}(0)+\frac{\delta_{\varepsilon}%
}{\varepsilon^{2}}\int_{0}^{T_{\varepsilon}}t\,dt\,\omega^{\prime}(0)\\
&  =\int_{0}^{l}2W(p_{\varepsilon}(s))s\,ds\omega^{\prime}(0)+\int%
_{l}^{T_{\varepsilon}\varepsilon^{-1}}2W(p_{\varepsilon}(s))s\,ds\omega
^{\prime}(0)+\frac{1}{2}\omega^{\prime}(0)\frac{\delta_{\varepsilon
}T_{\varepsilon}^{2}}{\varepsilon^{2}}\\
&  :=\mathcal{B}_{1}+\mathcal{B}_{2}+\mathcal{B}_{3}.
\end{align*}
where $p_{\varepsilon}(s):=v_{\varepsilon}(\varepsilon s)$ solves the initial value problem
\begin{equation}
\left\{
\begin{array}
[c]{l}%
p_{\varepsilon}^{\prime}(s)=\pm(\delta_{\varepsilon}+W(p_{\varepsilon}\left(
s\right)  ))^{1/2},\\
p_{\varepsilon}(0)=\alpha_{\varepsilon}%
\end{array}
\right.  \label{1d ode p epsilon}%
\end{equation}
in $[0,T_{\varepsilon}\varepsilon^{-1}]$, while $p_{\varepsilon}\left(
s\right)  =\beta_{\varepsilon}$ for $s>T_{\varepsilon}\varepsilon^{-1}$.

Let $q_{\varepsilon}$ be the unique solution to
(\ref{1d ode p epsilon}) in $\mathbb{R}$. Since $\delta_{\varepsilon
}\rightarrow0$, and $\alpha_{\varepsilon}\rightarrow\alpha$, by standard
results on the continuous dependence of solutions on a parameter (see, e.g.
\cite[Section 2.4]{gerald-book2012}), it follows that $q_{\varepsilon
}\rightarrow z_{\alpha}$ pointwise in $\mathbb{R}$, where $z_{\alpha}$ is
given in (\ref{cauchy problem z alpha}).

If $\alpha<b$, then $T_{\varepsilon}\varepsilon^{-1}\geq C_{\alpha}|\log
\delta_{\varepsilon}|\rightarrow\infty$ by (\ref{1d T epsilon below a}), and
so $p_{\varepsilon}\rightarrow z_{\alpha}$ pointwise in $[0,l]$. On the other
hand, if $\alpha=b$, then $q_{\varepsilon}\rightarrow z_{b}=b$ pointwise in
$\mathbb{R}$, while $p_{\varepsilon}\left(  s\right)  =\beta_{\varepsilon}$
for $s>T_{\varepsilon}\varepsilon^{-1}$. Hence,
\[
|p_{\varepsilon}\left(  s\right)  -z_{b}(s)|=|p_{\varepsilon}\left(  s\right)
-b|\leq|q_{\varepsilon}\left(  s\right)  -b|+|\beta_{\varepsilon
}-b|\rightarrow0
\]
for all $s\in\lbrack0,l]$.

Since $a\leq p_{\varepsilon}(s)\leq b$, by the Lebesgue dominated convergence
theorem%
\[
\lim_{\varepsilon\rightarrow0^{+}}\mathcal{B}_{1}=\int_{0}^{l}2W(z_{\alpha
}(s))s\,ds\omega^{\prime}(0).
\]
To estimate $\mathcal{B}_{2}$, observe that since $\delta_{\varepsilon}>0$ and
$\alpha_{\varepsilon}\leq p_{\varepsilon}(s)\leq\beta_{\varepsilon}<b$ for
$0\leq s\leq T_{\varepsilon}\varepsilon^{-1}$, by (\ref{W near b}) we have
\[
p_{\varepsilon}^{\prime}(s)\geq(W(p_{\varepsilon}\left(  s\right)
))^{1/2}\geq\sigma(b-p_{\varepsilon}(s))>0,
\]
and so
\[
-\sigma\geq\frac{(b-p_{\varepsilon}(s))^{\prime}}{b-p_{\varepsilon}(s)}%
=(\log(b-p_{\varepsilon}(s)))^{\prime}.
\]
Upon integration, we get%
\[
0\leq b-p_{\varepsilon}(s)\leq(b-\alpha_{\varepsilon})e^{-\sigma s}%
\leq(b-\alpha)e^{-\sigma s}.
\]
In turn, again by (\ref{W near b}), for $s\in\lbrack0,T_{\varepsilon
}\varepsilon^{-1}]$,%
\begin{equation}
W(p_{\varepsilon}(s))\leq\sigma^{-2}(b-p_{\varepsilon}(s))^{2}\leq\sigma
^{-2}(b-\alpha)^{2}e^{-2\sigma s}. \label{1d 204}%
\end{equation}
On the other hand, if $s\in\lbrack T_{\varepsilon}\varepsilon^{-1}%
,T\varepsilon^{-1}]$, then by (\ref{W near b}) and
(\ref{alpha epsilon and beta epsilon}),
\begin{equation}
W(p_{\varepsilon}(s))=W(\beta_{\varepsilon})\leq\sigma^{-2}(b-\beta
_{\varepsilon})^{2}\leq C\varepsilon^{2\gamma}. \label{1d 204a}%
\end{equation}
Hence, if $T_{\varepsilon}\varepsilon^{-1}\leq l$, by (\ref{1d 204}),
\[
\mathcal{B}_{2}\leq C\int_{l}^{\infty}e^{-2\sigma s}s\,ds=Ce^{-2\sigma
l}\left(  2\sigma l+1\right)  ,
\]
while if $T_{\varepsilon}\varepsilon^{-1}\geq l$, by (\ref{1d 204a}),%
\[
\mathcal{B}_{2}\leq C\varepsilon^{2\gamma}l.
\]
Therefore,
\[
\mathcal{B}_{2}\leq Ce^{-2\sigma l}\left(  2\sigma l+1\right)  +C\varepsilon
^{2\gamma}l
\]
for all $0<\varepsilon<\varepsilon_{0}$.

On the other hand, again by (\ref{1d T epsilon above}),
\[
\mathcal{B}_{3}\leq C\frac{\delta_{\varepsilon}T_{\varepsilon}^{2}%
}{2\varepsilon^{2}}\leq C\delta_{\varepsilon}\log^{2}\delta_{\varepsilon}.
\]

\noindent\textbf{Step 3:} To estimate $\mathcal{C}$ in (\ref{1d 557}), observe
that by (\ref{1d 200}), (\ref{1d 103a}), (\ref{1d T epsilon above}), and
(\ref{1d 557}),
\begin{align*}
\mathcal{C}  &  \leq|\omega^{\prime}|_{C^{0,d}}\int_{0}^{T_{\varepsilon}%
}\left(  \frac{1}{\varepsilon}W(v_{\varepsilon})+\varepsilon(v_{\varepsilon
}^{\prime})^{2}\right)  \,dt\,T_{\varepsilon}^{1+d}\frac{1}{\varepsilon}\\
&  \leq C\left(  C_{W}+C\delta_{\varepsilon}|\log\delta_{\varepsilon
}|+C\varepsilon^{\gamma}\right)  \varepsilon^{d}|\log\delta_{\varepsilon
}|^{1+d}\leq C\varepsilon^{d}|\log\delta_{\varepsilon}|^{1+d}%
\end{align*}
for all $0<\varepsilon<\varepsilon_{0}$.

\noindent\textbf{Step 4:} We estimate $\mathcal{D}$ in (\ref{1d 557}). By
(\ref{W near b}), (\ref{alpha epsilon and beta epsilon}), and (\ref{1d 557}),
for $t\geq T_{\varepsilon}$
\[
\mathcal{D}=W(\beta_{\varepsilon})\int_{T_{\varepsilon}}^{T}\omega\,dt\frac
{1}{\varepsilon^{2}}\leq\sigma^{-2}(b-\beta_{\varepsilon})^{2}\int_{0}%
^{T}\omega\,dt\frac{1}{\varepsilon^{2}}\leq C\varepsilon^{2\gamma-2}.
\]

Combining the estimates for $\mathcal{A}$, $\mathcal{B}_{2}$, $\mathcal{B}%
_{3}$, $\mathcal{C}$, and $\mathcal{D}$ and using (\ref{1d 557}) gives
\begin{align}
G_{\varepsilon}^{(2)}(v_{\varepsilon})  &  \leq C\frac{\delta_{\varepsilon}%
}{\varepsilon}+\int_{0}^{l}2W(p_{\varepsilon}(s))s\,ds\omega^{\prime
}(0)+Ce^{-2\sigma l}\left(  2\sigma l+1\right) \label{1d limsup with delta}\\
&  \quad+C\varepsilon^{2\gamma}l+C\delta_{\varepsilon}\log^{2}\delta
_{\varepsilon}+C\varepsilon^{d}|\log\delta_{\varepsilon}|^{1+d}+C\varepsilon
^{2\gamma-2}.\nonumber
\end{align}
By taking
\begin{equation}
\delta_{\varepsilon}:=\varepsilon^{m}\quad\text{or\quad}\delta_{\varepsilon
}:=\frac{\varepsilon}{\log^{m}\varepsilon},\quad m\geq2,
\label{1d delta epsilon}%
\end{equation}
we obtain (\ref{estimate limsup G1}). In turn, letting $\varepsilon
\rightarrow0^{+}$ in (\ref{estimate limsup G1}), we have
\[
\limsup_{\varepsilon\rightarrow0^{+}}G_{\varepsilon}^{(2)}(v_{\varepsilon
})\leq\int_{0}^{l}2W(z_{\alpha}(t))t\,dt\,\omega^{\prime}(0)+Ce^{-2\sigma
l}\left(  2\sigma l+1\right)  .
\]
Since $b-z_{\alpha}$ decays exponentially as $t\rightarrow\infty$, by
(\ref{estimate z alpha}), using (\ref{W near b}) and the Lebesgue dominated
convergence theorem, we let $l\rightarrow\infty$ to obtain
(\ref{limsup G2}).\hfill
\end{proof}

\begin{remark}
Observe that if we take $\delta_{\varepsilon}:=\varepsilon$, it follows from
(\ref{1d limsup with delta}) that%
\begin{align*}
G_{\varepsilon}^{(2)}(v_{\varepsilon})  &  \leq C+\int_{0}^{l}%
2W(p_{\varepsilon}(s))s\,ds\omega^{\prime}(0)+Ce^{-2\sigma l}\left(  2\sigma
l+1\right) \\
&  \quad+C\varepsilon^{2\gamma}l+C\varepsilon\log^{2}\varepsilon
+C\varepsilon^{d}|\log\varepsilon|^{1+d}+C\varepsilon^{2\gamma-2}.
\end{align*}

\end{remark}

\subsection{Properties of Minimizers of $G_{\varepsilon}$}

In this subsection we study qualitative properties of the minimizers of the
functional $G_{\varepsilon}$ defined in \eqref{1d functional}:%
\begin{equation}
G_{\varepsilon}(v):=\int_{I}(W(v(t))+\varepsilon^{2}(v^{\prime}(t))^{2}%
)\omega(t)\,dt,\quad v\in\,H^{1}(I), \label{1d functional G}%
\end{equation}
subject to the Dirichlet boundary conditions
\begin{equation}
v_{\varepsilon}(0)=\alpha_{\varepsilon},\quad v_{\varepsilon}(T)=\beta
_{\varepsilon}. \label{1d dirichlet 1}%
\end{equation}

\begin{theorem}
\label{theorem 1d EL}Assume that $W$ satisfies
\eqref{W_Smooth}-\eqref{W' three zeroes}, that $\omega$ satisfies
 \eqref{etaSmooth}, and that $a\leq\alpha_{\varepsilon},$
$\beta_{\varepsilon}\leq b$. Then the functional $G_{\varepsilon}$ admits a
minimizer $v_{\varepsilon}\in$\thinspace$H^{1}(I)$. Moreover, $v_{\varepsilon
}\in C^{2}([0,T])$, $v_{\varepsilon}$ satisfies the Euler--Lagrange equations
\begin{equation}
2\varepsilon^{2}(v_{\varepsilon}^{\prime}(t)\omega(t))^{\prime}-W^{\prime
}(v_{\varepsilon}(t))\omega(t)=0, \label{1d euler lagrange}%
\end{equation}
and $v_{\varepsilon}\equiv a$, or $v_{\varepsilon}\equiv b$, or
\begin{equation}
a<v_{\varepsilon}(t)<b\quad\text{for all }t\in(0,T). \label{1d truncated}%
\end{equation}

\end{theorem}

\begin{proof}
Since $\int_{I}(v^{\prime})^{2}\omega\,dt$ is convex, \eqref{etaSmooth} holds, and $W\geq0$, the
existence of minimizers $v_{\varepsilon}\in$\thinspace$H^{1}(I)$ of
$G_{\varepsilon}$ subject to the Dirichlet boundary conditions
(\ref{1d dirichlet 1}) follows from the direct method of the calculus of variations. Since $a\leq\alpha_{\varepsilon},$ $\beta_{\varepsilon}\leq b$, by
replacing $v_{\varepsilon}$ with the truncation
\[
\bar{v}_{\varepsilon}(t):=\left\{
\begin{array}
[c]{ll}%
a & \text{if }v_{\varepsilon}(t)\leq a,\\
v_{\varepsilon}(t) & \text{if }a<v_{\varepsilon}(t)<b,\\
b & \text{if }v_{\varepsilon}(t)\geq b,
\end{array}
\right.
\]
without loss of generality, we may assume that $v_{\varepsilon}$ satisfies
$a\leq v_{\varepsilon}\leq b$.

As $dG_{\varepsilon}(v_{\varepsilon})=0$, we have
\[
\int_{I}(W^{\prime}(v_{\varepsilon}(t))\varphi(t)+2\varepsilon^{2}%
v_{\varepsilon}^{\prime}(t)\varphi^{\prime}(t))\omega(t)\,dt=0
\]
for all $\varphi\in C_{c}^{1}(I)$. This implies that $W^{\prime}%
(v_{\varepsilon})\omega$ is the weak derivative of $2\varepsilon
^{2}v_{\varepsilon}^{\prime}\omega$. Hence,%
\[
2\varepsilon^{2}v_{\varepsilon}^{\prime}(t)\omega(t)=c+\int_{a}^{s}W^{\prime
}(v_{\varepsilon})\omega\,ds.
\]
Since the right-hand side is of class $C^{1}$ and $\omega\in C^{1,d}(I)$, we
have that $v_{\varepsilon}^{\prime}$ is of class $C^{1}$ and so we can
differentiate to obtain (\ref{1d euler lagrange}).

To prove (\ref{1d truncated}), observe that if there exists $t_{0}\in(0,T)$
such that $v_{\varepsilon}(t_{0})=b$, then since $v_{\varepsilon}\leq b$, the
point $t_{0}$ is a point of local maximum, and so $v_{\varepsilon}^{\prime
}(t_{0})=0$. Since $W^{\prime}(b)=0$, it follows by uniqueness of the Cauchy
problem (\ref{1d euler lagrange}) with initial data $v_{\varepsilon}(t_{0}%
)=b$, $v_{\varepsilon}^{\prime}(t_{0})=0$, that the unique solution if
$v_{\varepsilon}\equiv b$. Similarly, if $v_{\varepsilon}(t_{0})=a$ for some
$t_{0}\in(0,T)$, then $v_{\varepsilon}\equiv a$.\hfill
\end{proof}

\begin{corollary}
\label{corollary bounded derivative}Assume that $W$ satisfies hypotheses
\eqref{W_Smooth}-\eqref{W' three zeroes}, that $\omega$ satisfies
hypothesis \eqref{etaSmooth}, and that that $a\leq\alpha_{\varepsilon},$
$\beta_{\varepsilon}\leq b$. Let $v_{\varepsilon}$ be the minimizer of
$G_{\varepsilon}$ obtained in Theorem \ref{theorem 1d EL}. Then there exists a
constant $C_{0}>0$, depending only on $\omega$, $T$, $a$, $b$, and $W$, such
that%
\[
|v_{\varepsilon}^{\prime}(t)|\leq\frac{C_{0}}{\varepsilon}\quad\text{for all
}t\in I
\]
and for every $0<\varepsilon<1$.
\end{corollary}

\begin{proof}
In what follows $C_{0}>0$ is a constant that changes from line to line and
depends only on $\omega$, $T$, $a$, $b$, and $W$. Consider the function
$v_{0}(t):=\alpha_{\varepsilon}\frac{(T-t)}{T}+\beta_{\varepsilon}\frac{t}{T}$.
We have $a\leq v_{0}\leq b$ and $|v_{0}^{\prime}(t)|\leq\frac{b-a}{T}$. Since
$v_{\varepsilon}$ is a minimizer of $G_{\varepsilon}$, it follows
\begin{align*}
\int_{I}(W(v_{\varepsilon})+\varepsilon^{2}(v_{\varepsilon}^{\prime}%
)^{2})\omega\,dt  &  \leq\int_{I}(W(v_{0})+\varepsilon^{2}(v_{0}^{\prime}%
)^{2})\omega\,dt\\
&  \leq\left(  \max_{\lbrack a,b]}W+\frac{(b-a)^{2}}{T^{2}}\right)  \int%
_{I}\omega\,dt\leq C_{0}.
\end{align*}
As $v_{\varepsilon}^{\prime}$ and $\omega$ are continuous, by the mean
value theorem for integrals, there exists $t_{\varepsilon}\in\lbrack0,T]$ such
that%
\[
\varepsilon^{2}(v_{\varepsilon}^{\prime}(t_{\varepsilon})^{2}\omega
(t_{\varepsilon})=\frac{1}{T}\int_{I}\varepsilon^{2}(v_{\varepsilon}^{\prime
})^{2}\omega\,dt\leq C_{0}.
\]
In turn, by (\ref{etaSmooth}),%
\[
\varepsilon|v_{\varepsilon}^{\prime}(t_{\varepsilon})|\leq C_{0}.
\]
By (\ref{1d euler lagrange}),
\begin{equation}
2\varepsilon^{2}v_{\varepsilon}^{\prime}(t)\omega(t)=2\varepsilon
^{2}v_{\varepsilon}^{\prime}(t_{\varepsilon})\omega(t_{\varepsilon}%
)+\int_{t_{\varepsilon}}^{t}W^{\prime}(v_{\varepsilon})\omega\,ds
\label{integral identity 1}%
\end{equation}
for every $t\in\overline{I}$. Since $v_{\varepsilon}$ is bounded by
(\ref{1d truncated}), by \eqref{etaSmooth} this implies that
\begin{equation}
2\varepsilon^{2}|v_{\varepsilon}^{\prime}(t)|\leq C_{0}\varepsilon+\frac
{C_{0}}{\omega(t)}\int_{0}^{T}\omega\,ds\leq C_{0} \label{1d 109}%
\end{equation}
for all $t\in\overline{I}$. Rewrite (\ref{1d euler lagrange}) as
\begin{equation}
2\varepsilon^{2}v_{\varepsilon}^{\prime\prime}(t)+2\varepsilon^{2}\frac
{\omega^{\prime}(t)}{\omega(t)}v_{\varepsilon}^{\prime}(t)=W^{\prime
}(v_{\varepsilon}(t)). \label{1d euler lagrange expanded}%
\end{equation}
Using \eqref{etaSmooth} and (\ref{1d 109}), we obtain
\begin{equation}
2\varepsilon^{2}|v_{\varepsilon}^{\prime\prime}(t)|\leq\frac{|\omega^{\prime
}(t)|}{\omega(t)}2\varepsilon^{2}|v_{\varepsilon}^{\prime}(t)|+C_{0}\leq
C_{0}. \label{1d 110}%
\end{equation}
Next, we use a classical interpolation result. Let $t\in I$ and consider
$t_{1}\in I$ such that $|t-t_{1}|=\varepsilon$. By the mean value theorem
there exists $\theta$ between $t$ and $t_{1}$ such that
\[
v_{\varepsilon}(t)-v_{\varepsilon}(t_{1})=v_{\varepsilon}^{\prime}%
(\theta)(t-t_{1}).
\]
In turn, by the fundamental theorem of calculus
\[
v_{\varepsilon}^{\prime}(t)=v_{\varepsilon}^{\prime}(\theta)+\int_{\theta}%
^{t}v_{\varepsilon}^{\prime\prime}(s)\,ds=\frac{v_{\varepsilon}%
(t)-v_{\varepsilon}(t_{1})}{t-t_{1}}+\int_{\theta}^{t}v_{\varepsilon}%
^{\prime\prime}(s)\,ds.
\]
Using (\ref{1d truncated}) and (\ref{1d 110}), we obtain
\[
|v_{\varepsilon}^{\prime}(t)|\leq\frac{C_{0}}{\varepsilon}+\sup_{I}%
|v_{\varepsilon}^{\prime\prime}||t-\theta|\leq\frac{C_{0}}{\varepsilon}%
+\frac{C_{0}}{\varepsilon^{2}}\varepsilon.
\]
This concludes the proof.\hfill
\end{proof}

Another consequence of Theorem \ref{theorem 1d EL} is the following estimate.

\begin{theorem}
\label{thorem barrier 1d}Assume that $W$ satisfies
\eqref{W_Smooth}-\eqref{W' three zeroes}, that $\omega$ satisfies \eqref{etaSmooth}, and that $a\leq\alpha_{\varepsilon},$
$\beta_{\varepsilon}\leq b$. Let $v_{\varepsilon}$ be the minimizer of
$G_{\varepsilon}$ obtained in Theorem \ref{theorem 1d EL}, and  for $k \in \mathbb{N}$ let
\begin{equation}
 B^k_{\varepsilon}:=\{t\in\lbrack0,T]:~\beta_{-}\leq v_{\varepsilon}(t)\leq
\beta_{\varepsilon}-\varepsilon^{k}\}. \label{1d set B epsilon}%
\end{equation}
Then there exist $\mu>0$ and $0<\varepsilon_{0}<1$ depending only on
$\beta_{-}$, $T$, $\omega$, $W$, such that if $I_{\varepsilon}=[p_{\varepsilon
},q_{\varepsilon}]$ is a maximal subinterval of  $B_{\varepsilon}^k$, then
\begin{equation}
b-v_{\varepsilon}(t)\leq(b-v_{\varepsilon}(p_{\varepsilon}))e^{-\mu
(t-p_{\varepsilon})\varepsilon^{-1}}+(b-v_{\varepsilon}(q_{\varepsilon
}))e^{-\mu(q_{\varepsilon}-t)\varepsilon^{-1}} \label{1d b decay}%
\end{equation}
for all $t\in I_{\varepsilon}$ and all $0<\varepsilon<\varepsilon_{0}$. In
particular,
\begin{equation}
\operatorname*{diam}I_{\varepsilon}\leq C\varepsilon|\log\varepsilon|
\label{1d diam I epsilon}%
\end{equation}
for all $0<\varepsilon<\varepsilon_{0}$, where the constants $0<\varepsilon
_{0}<1$ and $C>0$ depend only on $\beta_{-}$, $T$, $\omega$, $W$ and $k$.
\end{theorem}

\begin{proof}
We claim that there exists $\mu>0$ such that
\begin{equation}
-W^{\prime}(s)\geq2\mu^{2}(b-s)\quad\text{for all }s\in\lbrack\beta_{-},b].
\label{eqn:WPrimeLinearized}%
\end{equation}
Since $W^{\prime\prime}(b)>0$, by the continuity of $W^{\prime\prime}$, we
have that $W^{\prime\prime}(s)\geq2\mu^{2}>0$ for all $s\in B(b,R_{1})$ and
for some $\mu>0$ and $R_{1}>0$. Upon integration, it follows that
\[
W^{\prime}(s)=-\int_{s}^{b}W^{\prime\prime}(r)\,dr\leq-2\mu^{2}(b-s)
\]
for all $s\in B(b,R_{1})$, with $s<b$. Using the fact that $W^{\prime}<0$ in
$(c,b)$, and by taking $\mu$ smaller, if necessary, we can assume that
\[
W^{\prime}(s)\leq-2\mu^{2}(b-s)
\]
for all $s\in\lbrack\beta_{-},b]$. Note that $\mu$ depends upon $\beta_{-}$
but not on $\varepsilon$. This proves the claim.

Write $I_{\varepsilon}:=[p_{\varepsilon},q_{\varepsilon}]$ and define
\begin{equation}
\phi(t):=(b-v_{\varepsilon}(p_{\varepsilon}))e^{-\mu(t-p_{\varepsilon
})\varepsilon^{-1}}+(b-v_{\varepsilon}(q_{\varepsilon}))e^{-\mu(q_{\varepsilon
}-t)\varepsilon^{-1}} \label{phiBarrierDefinition}%
\end{equation}
with $\mu$ fixed by \eqref{eqn:WPrimeLinearized}. We note that $\phi$
satisfies the following differential inequality:
\begin{align*}
(\phi^{\prime}\omega)^{\prime}  &  =\frac{\mu^{2}}{\varepsilon^{2}}\phi
\omega+\frac{\mu}{\varepsilon}\omega^{\prime}\left(  -(b-v_{\varepsilon
}(p_{\varepsilon}))e^{-\mu(t-p_{\varepsilon})\varepsilon^{-1}}%
+(b-v_{\varepsilon}(q_{\varepsilon}))e^{-\mu(q_{\varepsilon}-t)\varepsilon
^{-1}}\right) \\
&  \leq\frac{1}{\varepsilon^{2}}\left(  \mu^{2}+\varepsilon\frac
{|\omega^{\prime}|}{\omega}\mu\right)  \phi\omega.
\end{align*}
On the other hand, $\omega(t)\geq\omega_{0}>0$ for all $t\in I$. Thus,
\[
\varepsilon\frac{|\omega^{\prime}(t)|}{\omega(t)}\leq\varepsilon\frac
{\max|\omega^{\prime}|}{\omega_{0}}\leq\mu
\]
for all $t\in I$ and all $\varepsilon$ sufficiently small. Therefore in
$I_{\varepsilon}$
\begin{equation}
(\phi^{\prime}\omega)^{\prime}\leq2\varepsilon^{-2}\mu^{2}\phi\omega.
\label{eqn:Barrier1}%
\end{equation}
We then set $g(t):=b-v_{\varepsilon}(t)$,  and using (\ref{1d euler lagrange})
and \eqref{eqn:WPrimeLinearized} we have
\begin{equation}
(g^{\prime}\omega)^{\prime}=-\varepsilon^{-2}(W^{\prime}(v_{\varepsilon
}))\omega\geq2\varepsilon^{-2}\mu^{2}g\omega. \label{eqn:Barrier2}%
\end{equation}
We define $\Psi:=g-\phi$. By \eqref{phiBarrierDefinition},
\eqref{eqn:Barrier1} and \eqref{eqn:Barrier2}, for $\varepsilon$ small we obtain 
the following:
\begin{align*}
&  (\Psi^{\prime}\omega)^{\prime}\geq2\varepsilon^{-2}\mu^{2}\Psi\omega,\\
&  \Psi(p_{\varepsilon})\leq0,\quad\Psi(q_{\varepsilon})\leq0.
\end{align*}
The maximum principle implies that $\Psi\leq0$ for all $t\in I_{\varepsilon}$.
Thus
\[
b-v_{\varepsilon}(t)\leq(b-v_{\varepsilon}(p_{\varepsilon}))e^{-\mu
(t-p_{\varepsilon})\varepsilon^{-1}}+(b-v_{\varepsilon}(q_{\varepsilon
}))e^{-\mu(q_{\varepsilon}-t)\varepsilon^{-1}},
\]
which proves (\ref{1d b decay}). In turn, for $t:=\frac{p_{\varepsilon
}+q_{\varepsilon}}{2}$, we have
\[
\varepsilon^{k}\leq\beta_{\varepsilon}-v_{\varepsilon}(t)\leq b-v_{\varepsilon
}(t)\leq2be^{-\mu2^{-1}(q_{\varepsilon}-p_{\varepsilon})\varepsilon^{-1}},
\]
which implies that $-\frac{\mu}{2}(q_{\varepsilon}-p_{\varepsilon}%
)\varepsilon^{-1}\geq k\log\varepsilon-\log2b$, that is,
\[
0\leq q_{\varepsilon}-p_{\varepsilon}\leq2\mu^{-1}k\varepsilon|\log
\varepsilon|+2\mu^{-1}\varepsilon\log2b.
\]
This asserts (\ref{1d diam I epsilon}). \hfill
\end{proof}

\begin{remark}
With a similar proof, one can show that if $I_{\varepsilon}$ is a maximal
subinterval of
\[
 A_{\varepsilon}^k:=\{t\in\lbrack0,T]:~\alpha_{\varepsilon}+\varepsilon^{k}\leq
v_{\varepsilon}(t)\leq\alpha_{-}\},
\]
then
\[
\operatorname*{diam}I_{\varepsilon}\leq C\varepsilon|\log\varepsilon|
\]
for all $0<\varepsilon<\varepsilon_{0}$, where the constants $0<\varepsilon
_{0}<1$ and $C>0$ depend only on $\alpha_{-}$, $T$, $\omega$, $W$, and $k$.
\end{remark}

Next, we prove some differential inequalities for $v_{\varepsilon}$.

\begin{theorem}
\label{theorem monotonicity}Assume that $W$ satisfies
\eqref{W_Smooth}-\eqref{W' three zeroes}, that $\omega$ satisfies \eqref{etaSmooth}, and that that $a\leq\alpha_{\varepsilon},$
$\beta_{\varepsilon}\leq b$. Let $v_{\varepsilon}$ be the minimizer of
$G_{\varepsilon}$ obtained in Theorem \ref{theorem 1d EL} and let $\alpha
_{-},\beta_{-}$ be given as in \eqref{alpha and beta minus}. Then there exists
a constant $C>0$ such that
\[
\varepsilon(v_{\varepsilon}^{\prime}(0))^{2}-\frac{1}{\varepsilon}%
W(\alpha_{\varepsilon})\leq C
\]
for all $0<\varepsilon<1$. Moreover, there exist a constant $\tau_{0}>0$,
depending only on $\omega$, $T$, $a$, $b$, $\alpha_{-}$, $\beta_{-}$ and $W$,
such that
\begin{equation}
\frac{1}{2}\sigma^{2}(v_{\varepsilon}(t)-a)^{2}\leq\varepsilon^{2}%
(v_{\varepsilon}^{\prime}(t))^{2}\leq\frac{3}{2}\sigma^{-2}(v_{\varepsilon
}(t)-a)^{2} \label{1d v' near a}%
\end{equation}
whenever $a+\tau_{0}\varepsilon^{1/2}\leq v_{\varepsilon}(t)\leq\beta_{-}$,
and
\begin{equation}
\frac{1}{2}\sigma^{2}(b-v_{\varepsilon}(t))^{2}\leq\varepsilon^{2}%
(v_{\varepsilon}^{\prime}(t))^{2}\leq\frac{3}{2}\sigma^{-2}(b-v_{\varepsilon
}(t))^{2} \label{1d v' near b}%
\end{equation}
whenever $\alpha_{-}\leq v_{\varepsilon}(t)\leq b-\tau_{0}\varepsilon^{1/2}$,
where $\sigma>0$ is the constant given in Remark \ref{remark W near b}.
\end{theorem}

\begin{proof}
\textbf{Step 1: }We claim that
\[
\varepsilon(v_{\varepsilon}^{\prime}(0))^{2}-\frac{1}{\varepsilon}%
W(\alpha_{\varepsilon})\leq C
\]
for all $0<\varepsilon<1$ and for some constant $C>0$ independent of
$\varepsilon$. By Theorem \ref{theorem 1d first gamma},%
\begin{equation}
\sup_{0<\varepsilon<1}\int_{I}\left(  \frac{1}{\varepsilon}W(v)+\varepsilon
(v^{\prime})^{2}\right)  \omega\,dt\leq C \label{1d 69}%
\end{equation}
for all $0<\varepsilon<1$ and for some constant $C>0$ independent of
$\varepsilon$. Subdivide $I$ into $\lfloor\varepsilon^{-1}\rfloor$ equal
subintervals $I_{i}$ of equal length. Since%
\[
\sum_{i=1}^{\lfloor\varepsilon^{-1}\rfloor}\int_{I_{i}}\left(  \frac
{1}{\varepsilon}W(v)+\varepsilon(v^{\prime})^{2}\right)  \omega\,dt\leq C
\]
there exists $i_{\varepsilon}\in\{1,\ldots.,\lfloor\varepsilon^{-1}\rfloor\}$
such that%
\[
\int_{I_{i_{\varepsilon}}}\left(  \frac{1}{\varepsilon}W(v)+\varepsilon
(v^{\prime})^{2}\right)  \omega\,dt\leq C\varepsilon.
\]
In turn, there exists $t_{\varepsilon}\in I_{i_{\varepsilon}}$ such that%
\begin{equation}
\left(  \frac{1}{\varepsilon}W(v(t_{\varepsilon}))+\varepsilon(v^{\prime
}(t_{\varepsilon}))^{2}\right)  \omega(t_{\varepsilon})\leq C. \label{1d 70}%
\end{equation}
Multiply (\ref{1d euler lagrange expanded}) by $\frac{1}{\varepsilon
}v_{\varepsilon}^{\prime}(t)$ to get%
\begin{equation}
\varepsilon((v_{\varepsilon}^{\prime}(t))^{2})^{\prime}-\frac{1}{\varepsilon
}(W(v_{\varepsilon}(t)))^{\prime}=-2\varepsilon\frac{\omega^{\prime}%
(t)}{\omega(t)}(v_{\varepsilon}^{\prime}(t))^{2}. \label{1d 71}%
\end{equation}
Integrating between $0$ and $t_{\varepsilon}$, we have%
\[
\varepsilon(v_{\varepsilon}^{\prime}(0))^{2}-\frac{1}{\varepsilon}%
W(\alpha_{\varepsilon})=\varepsilon(v_{\varepsilon}^{\prime}(t_{\varepsilon
}))^{2}-\frac{1}{\varepsilon}W(v_{\varepsilon}(t_{\varepsilon}))-2\varepsilon
\int_{0}^{t_{\varepsilon}}\frac{\omega^{\prime}(t)}{\omega(t)}(v_{\varepsilon
}^{\prime}(t))^{2}dt\leq C
\]
where we used (\ref{1d 70}) and the fact that%
\begin{equation}
2\varepsilon\int_{0}^{t_{\varepsilon}}\frac{|\omega^{\prime}(t)|}{\omega
(t)}(v_{\varepsilon}^{\prime}(t))^{2}dt\leq\int_{I}\varepsilon(v_{\varepsilon
}^{\prime}(t))^{2}\omega(t)dt\leq C \label{1d 72}%
\end{equation}
since $\omega\in C^{1}([0,T])$, $\inf_{I}\omega>0$, and (\ref{1d 69}).

\textbf{Step 2: }Integrating (\ref{1d 71}) between $t$ and $0$ and using Step
1 and (\ref{1d 72}) gives%
\[
\left\vert \varepsilon(v_{\varepsilon}^{\prime}(t))^{2}-\frac{1}{\varepsilon
}W(v_{\varepsilon}(t))\right\vert \leq\varepsilon(v_{\varepsilon}^{\prime
}(0))^{2}+\frac{1}{\varepsilon}W(\alpha_{\varepsilon})+2\varepsilon\int%
_{0}^{T}\frac{|\omega^{\prime}|}{\omega}(v_{\varepsilon}^{\prime})^{2}dt\leq
C_{1}.
\]
In turn,%
\begin{equation}
W(v_{\varepsilon}(t))-C_{1}\varepsilon\leq\varepsilon^{2}(v_{\varepsilon
}^{\prime}(t))^{2}\leq W(v_{\varepsilon}(t))+C_{1}\varepsilon\label{1d 73}%
\end{equation}
for all $t\in I$, for all $0<\varepsilon<1$, and for some constant $C_{1}>0$
independent of $\varepsilon$. By Remark \ref{remark W near b},%
\[
\sigma^{2}(s-a)^{2}\leq W(s)\leq\frac{1}{\sigma^{2}}(s-a)^{2}%
\]
for all $s\in\lbrack a,\beta_{-}]$. Hence,
\begin{align*}
\frac{1}{2}\sigma^{2}(s-a)^{2}  &  \leq\sigma^{2}(s-a)^{2}-C_{1}%
\varepsilon\leq W(s)-C_{1}\varepsilon,\\
W(s)+C_{1}\varepsilon &  \leq\frac{1}{\sigma^{2}}(s-a)^{2}+C_{1}%
\varepsilon\leq\frac{3}{2}\frac{1}{\sigma^{2}}(s-a)^{2}%
\end{align*}
for all $s\in\lbrack a+\tau_{0}\varepsilon^{1/2},\beta_{-}]$, where $\tau
_{0}:=\sqrt{2}\sigma^{-1}C_{1}^{1/2}$ and we are assuming that $0<\sigma<1$.
In turn, by (\ref{1d 73}), we obtain (\ref{1d v' near a}). Estimate
(\ref{1d v' near b}) can be obtained similarly. We omit the details. \hfill
\end{proof}

\begin{remark}
    The proof of \eqref{1d v' near a} and \eqref{1d v' near b} is adapted from \cite{niethammer1995}.
\end{remark}

Next, we strengthen the hypotheses on the Dirichlet data $\alpha_{\varepsilon
}$ and $\beta_{\varepsilon}$ and derive additional properties of minimizers.
In particular, we assume that $\beta_{\varepsilon}\rightarrow b$ (the case
$\beta_{\varepsilon}\rightarrow a$ is similar).

\begin{theorem}
\label{theorem 1d properties minimizers}Assume that $W$ satisfies
\eqref{W_Smooth}-\eqref{W' three zeroes}, that $\alpha_{-},\beta_{-}$ satisfy \eqref{alpha and beta minus}, and that $\omega$ satisfies
hypotheses \eqref{etaSmooth}, \eqref{eta close to eta0}, where%
\begin{equation}
0\leq\omega_{0}<\frac{1}{2C_{W}}\min\{\operatorname*{d}\nolimits_{W}%
(a,\alpha_{-}),\operatorname*{d}\nolimits_{W}(\beta_{-},b)\}\omega(0).
\label{800a}%
\end{equation}
Let $\alpha_{-}\leq\alpha_{\varepsilon},\,\beta_{\varepsilon}\leq b$ satisfy
\eqref{alpha epsilon and beta epsilon} and let $v_{\varepsilon}$ be the
minimizer of $G_{\varepsilon}$ obtained in Theorem \ref{theorem 1d EL}. Given
$k\in\mathbb{N}$ there exist $0<\varepsilon_{0}<1$, $\mu>0$, and $C>0$
depending only on $\alpha_{-}$, $\beta_{-}$, $k$, $A_{0}$, $B_{0}$, $T$,
$\omega$, $W$, such that, for all $0<\varepsilon<\varepsilon_{0}$, the
following properties hold:

\begin{enumerate}
\item[(i)] Either $v_{\varepsilon}>\beta_{-}$ in $I$ or if $R_{\varepsilon}$
is the first time in $[0,T]$ such that $v_{\varepsilon}=\beta_{-}$, then%
\begin{equation}
R_{\varepsilon}\leq C\varepsilon. \label{1d R epsilon}%
\end{equation}

\item[(ii)] Either $v_{\varepsilon}\geq\beta_{\varepsilon}-\varepsilon^{k}$ in
$I$ or if $T_{\varepsilon}$ is the first time such that $v_{\varepsilon}%
=\beta_{\varepsilon}-\varepsilon^{k}$, then
\begin{equation}
T_{\varepsilon}\leq C\varepsilon|\log\varepsilon|. \label{1d T epsilon}%
\end{equation}
Moreover, if $R_{\varepsilon}$ exists, then $R_{\varepsilon}<T_{\varepsilon}$
and $v_{\varepsilon}(t)\in\lbrack\beta_{-},\beta_{\varepsilon}-\varepsilon
^{k}]$ for $t\in\lbrack R_{\varepsilon},T_{\varepsilon}]$.

\item[(iii)] If $T_{\varepsilon}$ exists, then $v_{\varepsilon}\geq b-\tau
_{0}\varepsilon^{1/2}$ in $[T_{\varepsilon},T]$, where $\tau_{0}$ is the
constant in Theorem \ref{theorem monotonicity}.
\end{enumerate}
\end{theorem}

By Corollary \ref{corollary 1d minimizer} and the fact that $\beta=b$, we have
that the minimizer of $G^{(1)}$ is given by the constant function $b$. In
turn, by Theorem \ref{theorem 1d first gamma} and by standard properties of
$\Gamma$-convergence (see \cite[Theorem 1.21]{braides-book2002}), minimizers
$v_{\varepsilon}$ of $G_{\varepsilon}$ converge in \thinspace$L^{1}(I)$ to $b$
as $\varepsilon\rightarrow0^{+}$, i.e.,
\begin{equation}
v_{\varepsilon}\rightarrow b\quad\text{in }\,L^{1}(I)\quad\text{as
}\varepsilon\rightarrow0^{+}. \label{1d minimizers converge}%
\end{equation}

We are now prepared to prove Theorem \ref{theorem 1d properties minimizers}.
For every measurable subset $E\subseteq I$ and for every $v\in$\thinspace
$H^{1}(I)$ satisfying (\ref{1d dirichlet}), we define the localized energy
\begin{equation}
G_{\varepsilon}^{(1)}(v;E):=\int_{E}\left(  \frac{1}{\varepsilon
}W(v)+\varepsilon(v^{\prime})^{2}\right)  \omega~dt. \label{G1 epsilon local}%
\end{equation}

\begin{proof}[Proof of Theorem~\ref{theorem 1d properties minimizers}]
Throughout the proof, the constants $0<\varepsilon_{0}<1$ and $C>0$ depend
only on $\alpha_{-}$, $\beta_{-}$, $k$, $A_{0}$, $B_{0}$, $T$, $\omega$, $W$.
By Theorem \ref{theorem 1d second gamma limsup} there exists $\tilde
{v}_{\varepsilon}\in$\thinspace$H^{1}(I)$ satisfying \eqref{1d dirichlet} such
that
\[
G_{\varepsilon}^{(2)}(\tilde{v}_{\varepsilon})\leq\int_{0}^{l}%
2W(p_{\varepsilon}(t))t\,dt\,\omega^{\prime}(0)+Ce^{-2\sigma l}\left(  2\sigma
l+1\right)  +C\varepsilon^{2\gamma}l+C\varepsilon^{\gamma_{1}}|\log
\varepsilon|^{\gamma_{0}}%
\]
for all $0<\varepsilon<\varepsilon_{0}$. Fixing $l$ and using the fact that
$v_{\varepsilon}$ is a minimizer of $G_{\varepsilon}$, we have that
\begin{equation}
G_{\varepsilon}^{(1)}(v_{\varepsilon})\leq G_{\varepsilon}^{(1)}(\tilde
{v}_{\varepsilon})\leq G^{(1)}(b)+C\varepsilon=\operatorname*{d}%
\nolimits_{W}(\alpha,b)\omega(0)+C\varepsilon\label{energies bounded}%
\end{equation}
for all $0<\varepsilon<\varepsilon_{0}$, where $C>0$ is independent of $\varepsilon$. We extend $v_{\varepsilon
}$ to be $\alpha_{\varepsilon}$ for $t<0$ and $\beta_{\varepsilon}$ for $t>T$.

\textbf{Step 1: }Since%
\[
\operatorname*{d}\nolimits_{W}(a,\alpha_{-})=\lim_{\delta\rightarrow0^{+}%
}2\int_{a+\delta}^{\alpha_{-}}W^{1/2}(r)\,dr,
\]
we can find $0<\delta_{1}<\alpha_{-}-a$ so small that
\begin{equation}
\operatorname*{d}\nolimits_{W}(a+\delta,\alpha_{-})=2\int_{a+\delta}%
^{\alpha_{-}}W^{1/2}(r)\,dr\geq\operatorname*{d}\nolimits_{W}(a,\alpha
_{-})-\epsilon_{0}, \label{700}%
\end{equation}
for all $0<\delta\leq\delta_{1}$, where
\begin{equation}
\epsilon_{0}<\frac{1}{4}\operatorname*{d}\nolimits_{W}(a,\alpha_{-}).
\label{800}%
\end{equation}
We claim that there is $\varepsilon_{0}>0$ such that for all $0<\varepsilon
<\varepsilon_{0}$, the set%
\[
A_{\varepsilon}:=\{t\in\lbrack0,T]:~a\leq v_{\varepsilon}(t)\leq a+\delta
_{1}\}
\]
is empty. To see this, assume by contradiction that there exists $t_{\varepsilon}\in(0,T)$ such
that $v_{\varepsilon}(t_{\varepsilon})=a+\delta_{1}$. Since $v_{\varepsilon}$
is continuous, $v_{\varepsilon}([0,T])\supseteq\lbrack a+\delta_{1}%
,\alpha_{\varepsilon}\vee\beta_{\varepsilon}]$. Hence, we can find a closed
interval $I_{\varepsilon}$ such that $v_{\varepsilon}(I_{\varepsilon
})=[a+\delta_{1},\alpha_{\varepsilon}\wedge\beta_{\varepsilon}]$ and a closed
interval $J_{\varepsilon}$ such that $v_{\varepsilon}(J_{\varepsilon}%
)=[\alpha_{\varepsilon}\wedge\beta_{\varepsilon},\alpha_{\varepsilon}\vee
\beta_{\varepsilon}]$. Then by (\ref{energies bounded}), and
(\ref{eta close to eta0}),%
\begin{align*}
\operatorname*{d}\nolimits_{W}(\alpha,b)\omega(0)+C\varepsilon &  \geq
G_{\varepsilon}^{(1)}(v_{\varepsilon})\geq G_{\varepsilon}^{(1)}%
(v;I_{\varepsilon}\cup J_{\varepsilon})\\
&  \geq\int_{I_{\varepsilon}\cup J_{\varepsilon}}2W^{1/2}(v_{\varepsilon
})|v_{\varepsilon}^{\prime}|\omega\,dt\geq\int_{I_{\varepsilon}\cup
J_{\varepsilon}}2W^{1/2}(v_{\varepsilon})|v_{\varepsilon}^{\prime}%
|\,dt(\omega(0)-\omega_{0})\\
&  =\left(  \operatorname*{d}\nolimits_{W}(a+\delta_1,\alpha_{\varepsilon}%
\wedge\beta_{\varepsilon})+\operatorname*{d}\nolimits_{W}(\alpha_{\varepsilon
}\wedge\beta_{\varepsilon},\alpha_{\varepsilon}\vee\beta_{\varepsilon
})\right)  (\omega(0)-\omega_{0}).
\end{align*}
In view of (\ref{distance definition}), $\operatorname*{d}\nolimits_{W}%
(\cdot,r)$ and $\operatorname*{d}\nolimits_{W}(s,\cdot)$ are Lipschitz
continuous with Lipschitz constant $L=\max_{[a,b]}\sqrt{W}$. Hence, by
(\ref{alpha epsilon and beta epsilon}),
\[
\operatorname*{d}\nolimits_{W}(\alpha_{\varepsilon}\wedge\beta_{\varepsilon
},\alpha_{\varepsilon}\vee\beta_{\varepsilon})\geq\operatorname*{d}%
\nolimits_{W}(\alpha,\beta)-L(A_{0}\varepsilon^{\gamma}+B_{0}\varepsilon
^{\gamma}),
\]
and so, using the fact that $\alpha_{-}<\alpha_{\varepsilon}$, $\beta
_{\varepsilon}$, we have%
\[
\operatorname*{d}\nolimits_{W}(\alpha,b)\omega(0)+C\varepsilon\geq\left(
\operatorname*{d}\nolimits_{W}(a+\delta_1,\alpha_{-})+\operatorname*{d}%
\nolimits_{W}(\alpha,b)\right)  (\omega(0)-\omega_{0})-C\varepsilon^{\gamma}%
\]
or, equivalently,%
\begin{align*}
\operatorname*{d}\nolimits_{W}(a+\delta_1,\alpha_{-})\omega(0)  &  \leq\left(
\operatorname*{d}\nolimits_{W}(a+\delta_1,\alpha_{-})+\operatorname*{d}%
\nolimits_{W}(\alpha,b)\right)  \omega_{0}+C\varepsilon\\
&  \leq C_{W}\omega_{0}+C\varepsilon<\frac{1}{2}\operatorname*{d}%
\nolimits_{W}(a,\alpha_{-})\omega(0),
\end{align*}
for $0<\varepsilon<\varepsilon_{0}$, where in the last inequality we used
(\ref{eta 0 alpha-}), and where $\varepsilon_{0}$ is so small that%
\[
C\varepsilon_{0}\leq\frac{1}{4}\operatorname*{d}\nolimits_{W}(a,\alpha
_{-})\omega(0).
\]
Hence, also by (\ref{700}),%
\[
\operatorname*{d}\nolimits_{W}(a,\alpha_{-})-\epsilon_{0}\leq\operatorname*{d}%
\nolimits_{W}(a+\delta_1,\alpha_{-})<\frac{1}{2}\operatorname*{d}\nolimits_{W}%
(a,\alpha_{-}),
\]
which contradicts (\ref{800}).

\textbf{Step 2: }To prove (\ref{1d R epsilon}), observe that by Step 1, for
all $t\in\lbrack0,R_{\varepsilon})$ we have that $v_{\varepsilon}%
(t)\in\lbrack a+\delta_{1},\beta_{-})$. Hence,
\[
G_{\varepsilon}^{(1)}(v_{\varepsilon};[0,R_{\varepsilon}])\geq\varepsilon
^{-1}R_{\varepsilon}\min_{[0,T]}\omega\min_{\lbrack a+\delta_{1},\beta_{-}%
]}W,
\]
and so (\ref{1d R epsilon}) follows from (\ref{energies bounded}).

\textbf{Step 3: }To prove Item (ii), we consider three separate cases.

\textbf{Substep 3 a. }Assume first that either $\alpha_{\varepsilon}<\beta
_{-}$ or $\alpha_{\varepsilon}=\beta_{-}$ and $v_{\varepsilon}^{\prime}(0)>0$.
If $\alpha_{\varepsilon}<\beta_{-}$, since $v_{\varepsilon}(t)\in\lbrack
a+\delta_{1},\beta_{-})$ for all $t\in\lbrack0,R_{\varepsilon})$, we have that
$v_{\varepsilon}^{\prime}(R_{\varepsilon})\geq0$. On the other hand, by
(\ref{1d v' near b}), $v^{\prime}_{\varepsilon}(R_{\varepsilon})>0$, which in turn implies,
again by (\ref{1d v' near b}), that $v_{\varepsilon}^{\prime}(t)>0$ for all
$t\geq R_{\varepsilon}$ such that $v_{\varepsilon}(t)\leq b-\tau
_{0}\varepsilon^{1/2}$. Similarly, if $\alpha_{\varepsilon}=\beta_{-}$ and
$v^{\prime}_{\varepsilon}(0)>0$, then by (\ref{1d v' near b}) that $v_{\varepsilon}^{\prime
}(t)>0$ for all $t\geq R_{\varepsilon}=0$ such that $v_{\varepsilon}(t)\leq
b-\tau_{0}\varepsilon^{1/2}$.

Hence, in both cases, there exists a maximal interval $I_{\varepsilon}$ of the
set $B_{\varepsilon}$ defined in (\ref{1d set B epsilon}) whose left endpoint
is $R_{\varepsilon}$. Let $S_{\varepsilon}$ be the right endpoint of
$I_{\varepsilon}$. If $v_{\varepsilon}(S_{\varepsilon})=\beta_{\varepsilon
}-\varepsilon^{k}$, then $S_{\varepsilon}=T_{\varepsilon}$ and so
(\ref{1d T epsilon}) follows from (\ref{1d diam I epsilon}) and
(\ref{1d R epsilon}).

If $v_{\varepsilon}(S_{\varepsilon})=\beta_{-}$, then since $v_{\varepsilon
}^{\prime}(t)>0$ for all $t\geq R_{\varepsilon}$ such that $v_{\varepsilon
}(t)\leq b-\tau_{0}\varepsilon^{1/2}$, there exists $P_{\varepsilon}%
\in(R_{\varepsilon},S_{\varepsilon})$ such that $v_{\varepsilon}%
(P_{\varepsilon})=b-\tau_{0}\varepsilon^{1/2}$. It follows that
$v_{\varepsilon}([R_{\varepsilon},P_{\varepsilon}])=[\beta_{-},b-\tau
_{0}\varepsilon^{1/2}]$, while $v_{\varepsilon}([P_{\varepsilon}%
,S_{\varepsilon}])\supseteq\lbrack\beta_{-},b-\tau_{0}\varepsilon^{1/2}]$.
Then by (\ref{energies bounded}), and (\ref{eta close to eta0}), we have 
\begin{align*}
\operatorname*{d}\nolimits_{W}(\alpha,b)\omega(0)+C\varepsilon &  \geq
G_{\varepsilon}^{(1)}(v_{\varepsilon})\geq G_{\varepsilon}^{(1)}%
(v;[0,P_{\varepsilon}]\cup\lbrack P_{\varepsilon},S_{\varepsilon}])\\
&  \geq\int_{\lbrack0,P_{\varepsilon}]\cup\lbrack P_{\varepsilon
},S_{\varepsilon}]}2W^{1/2}(v_{\varepsilon})|v_{\varepsilon}^{\prime}%
|\omega\,dt\\
&  \geq\int_{\lbrack0,P_{\varepsilon}]\cup\lbrack P_{\varepsilon
},S_{\varepsilon}]}2W^{1/2}(v_{\varepsilon})|v_{\varepsilon}^{\prime
}|\,dt(\omega(0)-\omega_{0})\\
&  =\left(  \operatorname*{d}\nolimits_{W}(\alpha_{\varepsilon},b-\tau
_{0}\varepsilon^{1/2})+\operatorname*{d}\nolimits_{W}(\beta_{-},b-\tau
_{0}\varepsilon^{1/2})\right)  (\omega(0)-\omega_{0}).
\end{align*}
As in Step 1, using the fact that $\operatorname*{d}\nolimits_{W}(\cdot,r)$
and $\operatorname*{d}\nolimits_{W}(s,\cdot)$ are Lipschitz continuous and
(\ref{alpha epsilon and beta epsilon}), it follow that%
\[
\operatorname*{d}\nolimits_{W}(\alpha,b)\omega(0)+C\varepsilon\geq
(\operatorname*{d}\nolimits_{W}(\alpha,b)+\operatorname*{d}\nolimits_{W}%
(\beta_{-},b)-L(A_{0}\varepsilon^{\gamma}+2\tau_{0}\varepsilon^{1/2}%
))(\omega(0)-\omega_{0}),
\]
or, equivalently,
\[
\lbrack\operatorname*{d}\nolimits_{W}(\alpha,b)+\operatorname*{d}%
\nolimits_{W}(\beta_{-},b)]\omega_{0}+C(\varepsilon^{\gamma}+\varepsilon
^{1/2})\geq\operatorname*{d}\nolimits_{W}(\beta_{-},b)\omega(0),
\]
which contradicts (\ref{800a}), provided we take $0<\varepsilon<\varepsilon
_{0}$ with $\varepsilon_{0}$ sufficiently small (depending only on $\beta_{-}$
and $W$).

On the other hand, if $v_{\varepsilon}(t)>\beta_{-}$ for all $t\in I$, then
$I_{\varepsilon}=[0,T_{\varepsilon}]$ is a maximal interval of the set
$B_{\varepsilon}$ defined in (\ref{1d set B epsilon}), and so
(\ref{1d T epsilon}) follows from (\ref{1d diam I epsilon}).

\textbf{Substep 3 b. }Assume first that $\alpha_{\varepsilon}=\beta_{-}$ and
$v_{\varepsilon}^{\prime}(0)\leq0$. Then%
\[
(\omega v_{\varepsilon}^{\prime})^{\prime}(0)=\omega(0) W^{\prime}(\beta_{-})<0
\]
and so in both cases $ v_{\varepsilon}^{\prime}(t)<0$ for all $t>0$ small. It follows from
(\ref{1d v' near a}), that $v_{\varepsilon}^{\prime}(t)<0$ for all $t>0$ such that
$v_{\varepsilon}(t)\geq a+\tau_{0}\varepsilon^{1/2}$. Since $v_{\varepsilon
}(T)=\beta_{\varepsilon}$, this implies that there exist $L_{\varepsilon
}<M_{\varepsilon}<N_{\varepsilon}$ such that $v_{\varepsilon}%
([0,L_{\varepsilon}])\supseteq\lbrack a+\tau_{0}\varepsilon^{1/2},\beta_{-}]$
and $v_{\varepsilon}([M_{\varepsilon},N_{\varepsilon}])\supseteq\lbrack
\beta_{-},\beta_{\varepsilon}]$. Then by (\ref{energies bounded}), and
(\ref{eta close to eta0}), we obtain
\begin{align*}
\operatorname*{d}\nolimits_{W}(\alpha,b)\omega(0)+C\varepsilon &
=\operatorname*{d}\nolimits_{W}(\beta_{-},b)\omega(0)+C\varepsilon\geq
G_{\varepsilon}^{(1)}(v_{\varepsilon})\geq G_{\varepsilon}^{(1)}%
(v;[0,L_{\varepsilon}]\cup\lbrack M_{\varepsilon},N_{\varepsilon}])\\
&  \geq\int_{\lbrack0,L_{\varepsilon}]\cup\lbrack M_{\varepsilon
},N_{\varepsilon}]}2W^{1/2}(v_{\varepsilon})|v_{\varepsilon}^{\prime}%
|\omega\,dt\\
&  \geq\int_{\lbrack0,L_{\varepsilon}]\cup\lbrack M_{\varepsilon
},N_{\varepsilon}]}2W^{1/2}(v_{\varepsilon})|v_{\varepsilon}^{\prime
}|\,dt(\omega(0)-\omega_{0})\\
&  =\left(  \operatorname*{d}\nolimits_{W}(a+\tau_{0}\varepsilon^{1/2}%
,\beta_{-})+\operatorname*{d}\nolimits_{W}(\beta_{-},\beta_{\varepsilon
})\right)  (\omega(0)-\omega_{0}).
\end{align*}
Using the fact that $\operatorname*{d}\nolimits_{W}(\cdot,r)$ and
$\operatorname*{d}\nolimits_{W}(s,\cdot)$ are Lipschitz continuous and
(\ref{alpha epsilon and beta epsilon}), it follow that%
\[
\operatorname*{d}\nolimits_{W}(\beta_{-},b)\omega(0)+C\varepsilon
\geq(\operatorname*{d}\nolimits_{W}(a,\beta_{-})+\operatorname*{d}%
\nolimits_{W}(\beta_{-},b)-L(B_{0}\varepsilon^{\gamma}+2\tau_{0}%
\varepsilon^{1/2}))(\omega(0)-\omega_{0}),
\]
or, equivalently,
\[
\lbrack\operatorname*{d}\nolimits_{W}(a,\beta_{-})+\operatorname*{d}%
\nolimits_{W}(\beta_{-},b)]\omega_{0}+C(\varepsilon^{\gamma}+\varepsilon
^{1/2})\geq\operatorname*{d}\nolimits_{W}(a,\beta_{-})\omega(0)\geq
\operatorname*{d}\nolimits_{W}(a,\alpha_{-})\omega(0),
\]
which contradicts (\ref{800a}), provided we take $0<\varepsilon<\varepsilon
_{0}$ with $\varepsilon_{0}$ sufficiently small (depending only on $\alpha
_{-}$ and $W$).

This contradiction shows that if $\alpha=\beta_{-}$, then $v_{\varepsilon
}^{\prime}(0)>0$ and so we are back to Substep 3 a.

\textbf{Substep 3 c. }Finally, we consider the case in which $\alpha
_{\varepsilon}>\beta_{-}$. We claim that $v_{\varepsilon}%
>\beta_{-}$ in $I$. Indeed, assume by contradiction that $R_{\varepsilon}$
exists. Then $v_{\varepsilon}^{\prime}(R_{\varepsilon})\leq0$ and so by (\ref{1d v' near a}%
), $v_{\varepsilon}^{\prime}(t)<0$ for all $t>0$ such that $v_{\varepsilon}(t)\geq
a+\tau_{0}\varepsilon^{1/2}$. Since $v_{\varepsilon}(T)=\beta_{\varepsilon}$,
this implies that there exist $R_{\varepsilon}<L_{\varepsilon}<M_{\varepsilon
}<N_{\varepsilon}$ such that $v_{\varepsilon}([R_{\varepsilon},L_{\varepsilon
}])\supseteq\lbrack a+\tau_{0}\varepsilon^{1/2},\beta_{-}]$ and
$v_{\varepsilon}([M_{\varepsilon},N_{\varepsilon}])\supseteq\lbrack\beta
_{-},\beta_{\varepsilon}]$. Then by (\ref{energies bounded}), and
(\ref{eta close to eta0}), we have
\begin{align*}
\operatorname*{d}\nolimits_{W}(\alpha,b)\omega(0)+C\varepsilon &  \geq
G_{\varepsilon}^{(1)}(v_{\varepsilon})\geq G_{\varepsilon}^{(1)}%
(v;[R,L_{\varepsilon}]\cup\lbrack M_{\varepsilon},N_{\varepsilon}])\\
&  \geq\int_{\lbrack0,L_{\varepsilon}]\cup\lbrack M_{\varepsilon
},N_{\varepsilon}]}2W^{1/2}(v_{\varepsilon})|v_{\varepsilon}^{\prime}%
|\omega\,dt\\
&  \geq\int_{\lbrack0,L_{\varepsilon}]\cup\lbrack M_{\varepsilon
},N_{\varepsilon}]}2W^{1/2}(v_{\varepsilon})|v_{\varepsilon}^{\prime
}|\,dt(\omega(0)-\omega_{0})\\
&  =\left(  \operatorname*{d}\nolimits_{W}(a+\tau_{0}\varepsilon^{1/2}%
,\beta_{-})+\operatorname*{d}\nolimits_{W}(\beta_{-},\beta_{\varepsilon
})\right)  (\omega(0)-\omega_{0}).
\end{align*}
Using the fact that $\operatorname*{d}\nolimits_{W}(\cdot,r)$ and
$\operatorname*{d}\nolimits_{W}(s,\cdot)$ are Lipschitz continuous and
(\ref{alpha epsilon and beta epsilon}), it follow that%
\[
\operatorname*{d}\nolimits_{W}(\alpha,b)\omega(0)+C\varepsilon\geq
(\operatorname*{d}\nolimits_{W}(a,\beta_{-})+\operatorname*{d}\nolimits_{W}%
(\beta_{-},b)-L(B_{0}\varepsilon^{\gamma}+2\tau_{0}\varepsilon^{1/2}%
))(\omega(0)-\omega_{0}),
\]
or, equivalently,
\[
\lbrack\operatorname*{d}\nolimits_{W}(a,\beta_{-})+\operatorname*{d}%
\nolimits_{W}(\beta_{-},b)]\omega_{0}+C(\varepsilon^{\gamma}+\varepsilon
^{1/2})\geq\operatorname*{d}\nolimits_{W}(a,\alpha)\omega(0)\geq
\operatorname*{d}\nolimits_{W}(a,\alpha_{-})\omega(0),
\]
which contradicts (\ref{800a}), provided we take $0<\varepsilon<\varepsilon
_{0}$ with $\varepsilon_{0}$ sufficiently small (depending only on $\alpha
_{-}$ and $W$).

\textbf{Step 4. }To prove Item (iii), assume that $T_{\varepsilon}$ exists.
Assume by contradiction that there exists a first time $Q_{\varepsilon
}>T_{\varepsilon}$ such that $v_{\varepsilon}=b-\tau_{0}\varepsilon^{1/2}$.
Then $v_{\varepsilon}^{\prime}(Q_{\varepsilon})\leq0$, but so by
(\ref{1d v' near a}) and (\ref{1d v' near b}), $v^{\prime}_{ \varepsilon}(t)<0$ for all
$t\geq Q_{\varepsilon}$ such that $v_{\varepsilon}(t)\geq a+\tau
_{0}\varepsilon^{1/2}$. Since $v_{\varepsilon}(T)=\beta_{\varepsilon}$, there
exists a first time $S_{\varepsilon}$ such that $v_{\varepsilon}=a+\tau
_{0}\varepsilon^{1/2}$. Hence, $v_{\varepsilon}([0,T_{\varepsilon}%
])\supseteq\lbrack\alpha_{\varepsilon},\beta_{\varepsilon}-\varepsilon^{k}]$
and $v_{\varepsilon}([Q_{\varepsilon},S_{\varepsilon}])\supseteq\lbrack
a+\tau_{0}\varepsilon^{1/2},b-\tau_{0}\varepsilon^{1/2}]$. Then by
(\ref{energies bounded}), and (\ref{eta close to eta0}), we have 
\begin{align*}
\operatorname*{d}\nolimits_{W}(\alpha,b)\omega(0)+C\varepsilon &  \geq
G_{\varepsilon}^{(1)}(v_{\varepsilon})\geq G_{\varepsilon}^{(1)}%
(v;[0,T_{\varepsilon}]\cup\lbrack Q_{\varepsilon},S_{\varepsilon}])\\
&  \geq\int_{\lbrack0,T_{\varepsilon}]\cup\lbrack Q_{\varepsilon
},S_{\varepsilon}]}2W^{1/2}(v_{\varepsilon})|v_{\varepsilon}^{\prime}%
|\omega\,dt\\
&  \geq\int_{\lbrack0,T_{\varepsilon}]\cup\lbrack Q_{\varepsilon
},S_{\varepsilon}]}2W^{1/2}(v_{\varepsilon})|v_{\varepsilon}^{\prime
}|\,dt(\omega(0)-\omega_{0})\\
&  =\left(  \operatorname*{d}\nolimits_{W}(\alpha_{\varepsilon},\beta
_{\varepsilon}-\varepsilon^{k})+\operatorname*{d}\nolimits_{W}(a+\tau
_{0}\varepsilon^{1/2},b-\tau_{0}\varepsilon^{1/2})\right)  (\omega
(0)-\omega_{0}).
\end{align*}
Using the fact that $\operatorname*{d}\nolimits_{W}(\cdot,r)$ and
$\operatorname*{d}\nolimits_{W}(s,\cdot)$ are Lipschitz continuous and
(\ref{alpha epsilon and beta epsilon}), it follow that%
\[
\operatorname*{d}\nolimits_{W}(\alpha,b)\omega(0)+C\varepsilon\geq
(\operatorname*{d}\nolimits_{W}(\alpha,b)+\operatorname*{d}\nolimits_{W}%
(a,b)-C(\varepsilon^{k}+\varepsilon^{\gamma}+\varepsilon^{1/2}))(\omega
(0)-\omega_{0}),
\]
or, equivalently,
\[
\lbrack\operatorname*{d}\nolimits_{W}(a,\beta_{-})+\operatorname*{d}%
\nolimits_{W}(\beta_{-},b)]\omega_{0}+C(\varepsilon^{k}+\varepsilon^{\gamma
}+\varepsilon^{1/2})\geq\operatorname*{d}\nolimits_{W}(a,b)\omega
(0)\geq\operatorname*{d}\nolimits_{W}(a,\alpha_{-})\omega(0),
\]
which contradicts (\ref{800a}), provided we take $0<\varepsilon<\varepsilon
_{0}$ with $\varepsilon_{0}$ sufficiently small (depending only on $\alpha
_{-}$, $\beta_{-}$, and $W$). \hfill
\end{proof}

\subsection{Second-Order $\Gamma$-liminf}

In this subsection, we prove the $\liminf$ counterpart of Theorem
\ref{theorem 1d second gamma limsup}.

\begin{theorem}
[Second-Order $\Gamma$-Liminf]\label{theorem liminf second 1d}Assume that $W$ satisfies \eqref{W_Smooth}-\eqref{W' three zeroes}, that that $\alpha_{-}$
satisfies \eqref{alpha and beta minus}, and that $\omega$ satisfies \eqref{etaSmooth}, \eqref{eta close to eta0}, and
\eqref{eta 0 alpha-}. Let $\alpha_{-}\leq\alpha_{\varepsilon},\,\beta
_{\varepsilon}\leq b$ satisfy \eqref{alpha epsilon and beta epsilon} and let
$v_{\varepsilon}$ be the minimizer of $G_{\varepsilon}$ obtained in Theorem
\ref{theorem 1d EL}. Then there exist $0<\varepsilon_{0}<1$, $C>0$, and
$l_{0}>1$, depending only on $\alpha_{-}$, $A_{0}$, $B_{0}$, $T$, $\omega$,
and $W$, such that%
\[
G_{\varepsilon}^{(2)}(v_{\varepsilon})\geq2\omega^{\prime}(0)\int_{0}%
^{l}W^{1/2}(w_{\varepsilon})w_{\varepsilon}^{\prime}s\,ds-Ce^{-l\mu}\left(
l\mu+1\right)  -C\varepsilon^{1/2} l^2-C\varepsilon^{\gamma_{1}}|\log
\varepsilon|^{2+\gamma_{0}}%
\]
for all $0<\varepsilon<\varepsilon_{0}$ and $l>l_{0}$, where $w_{\varepsilon
}(s):=v_{\varepsilon}(\varepsilon s)$ for $s\in\lbrack0,T\varepsilon^{-1}]$
satisfies%
\[
\lim_{\varepsilon\rightarrow0^{+}}\int_{0}^{l}W^{1/2}(w_{\varepsilon
})w_{\varepsilon}^{\prime}s\,ds=\int_{0}^{l}W^{1/2}(z_{\alpha})z_{\alpha
}^{\prime}s\,ds
\]
for every $l>0$ and where $z_{\alpha}$ solves the Cauchy problem
\eqref{cauchy problem z alpha}. In particular,%
\[
\liminf_{\varepsilon\rightarrow0^{+}}G_{\varepsilon}^{(2)}(v_{\varepsilon
})\geq2\omega^{\prime}(0)\int_{0}^{\infty}W^{1/2}(z_{\alpha})z_{\alpha
}^{\prime}s\,ds.
\]

\end{theorem}

Note that Theorems \ref{theorem 1d second gamma limsup} and
\ref{theorem liminf second 1d} together provide the second-order asymptotic
development by $\Gamma$-convergence for the functionals $G_{\varepsilon}$
defined in (\ref{1d functional G}). To prove Theorem
\ref{theorem liminf second 1d}, it is convenient to rescale the functionals
$G_{\varepsilon}$. We define
\begin{equation}
H_{\varepsilon}(w):=\int_{0}^{T\varepsilon^{-1}}(W(w(s))+(w^{\prime}%
(s))^{2})\omega_{\varepsilon}(s)\,ds \label{rescaledProblemFormulation}%
\end{equation}
for all $w\in$\thinspace$H^{1}((0,T\varepsilon^{-1}))$ such that
\begin{equation}
w(0)=\alpha_{\varepsilon},\quad w(T\varepsilon^{-1})=\beta_{\varepsilon},
\label{eqn:RescaledMassConstraint}%
\end{equation}
where
\begin{equation}
\omega_{\varepsilon}(s):=\omega(\varepsilon s). \label{rescaledEtaDefinition}%
\end{equation}
Note that if $v_{\varepsilon}$ is the minimizers of $G_{\varepsilon}$ obtained
in Theorem \ref{theorem 1d EL}, then
\begin{equation}
w_{\varepsilon}(s):=v_{\varepsilon}(\varepsilon s),\quad s\in\lbrack
0,T\varepsilon^{-1}] \label{w epsilon}%
\end{equation}
is a minimizer of $H_{\varepsilon}$.

We prove that the functions $w_{\varepsilon}$ necessarily converge.

\begin{lemma}
\label{profilesConverge}Assume that the hypotheses of Theorem
\ref{theorem liminf second 1d} hold. Let $w_{\varepsilon}$ be as in
\eqref{w epsilon}. Then $w_{\varepsilon}\rightharpoonup z_{\alpha}$ in
$H^{1}((0,l))$ for every $l\in\mathbb{N}$, where $z_{\alpha}$ solves the
Cauchy problem \eqref{cauchy problem z alpha}. Moreover, the family
\[
|w_{\varepsilon}^{\prime}(t)|\leq C\quad\text{for all }t\in(0,T\varepsilon
^{-1})
\]
and all $0<\varepsilon<1$, where the constant $C>0$ depends only on $\omega$,
$T$, $a$, $b$, and $W$.
\end{lemma}

\begin{proof}
Extend $w_{\varepsilon}$ to be $\beta_{\varepsilon}$ for $t\geq T\varepsilon
^{-1}$. The fact that the family $\{w_{\varepsilon}^{\prime}\}_{\varepsilon}$
is uniformly bounded in $L^{\infty}(\mathbb{R}_{+})$ follows from Corollary
\ref{corollary bounded derivative}. Furthermore, we have that the functions
$w_{\varepsilon}$ are bounded in $L^{\infty}(\mathbb{R}_{+})$ by
(\ref{1d truncated}). Let $\varepsilon_{n}\rightarrow0^{+}$. After a
diagonalization argument, we can find a subsequence $\{\varepsilon_{n_{k}%
}\}_{k}$ of $\{\varepsilon_{n}\}_{n}$ and $w_{0}\in H_{\operatorname*{loc}%
}^{1}(\mathbb{R}_{+})$ such that
\begin{equation}
w_{\varepsilon_{n_{k}}}\rightharpoonup w_{0}\text{ in }H_{\operatorname*{loc}%
}^{1}(\mathbb{R}_{+}). \label{rescaledWeakConvergence}%
\end{equation}
For simplicity, in what follows, we write $\varepsilon$ in place of
$\varepsilon_{n_{k}}$.

Since $w_{\varepsilon}(0)=\alpha_{\varepsilon}\rightarrow\alpha$,   we have that
that $w_{0}(0)=\alpha$. By Theorem \ref{theorem 1d EL} and (\ref{w epsilon}),
we obtain
\begin{equation}%
\begin{cases}
2(w_{\varepsilon}^{\prime}\omega_{\varepsilon})^{\prime}-W^{\prime
}(w_{\varepsilon})\omega_{\varepsilon}=0\quad\text{on }(0,T\varepsilon
^{-1}),\\
w_{\varepsilon}(0)=\alpha_{\varepsilon},\quad w_{\varepsilon}(T\varepsilon
^{-1})=\beta_{\varepsilon}.
\end{cases}
\label{rescaledELEquation}%
\end{equation}
Hence for every $\phi\in C_{c}^{\infty}(\mathbb{R}_{+})$ and for $\varepsilon$
small enough we find that
\[
\int_{0}^{T\varepsilon^{-1}}2w_{\varepsilon}^{\prime}\omega_{\varepsilon}%
\phi^{\prime}+W^{\prime}(w_{\varepsilon})\omega_{\varepsilon}\phi\,ds=0.
\]
Letting $\varepsilon\rightarrow0$ and using \eqref{rescaledEtaDefinition} and
\eqref{rescaledWeakConvergence} gives
\[
\int_{\mathbb{R}}2w_{0}^{\prime}\omega(0)\phi^{\prime}+W^{\prime}(w_{0}%
)\omega(0)\phi\,ds=0,
\]
which then shows that $w_{0}$ solves the initial value problem
\begin{equation}
\left\{
\begin{array}
[c]{l}%
2w_{0}^{\prime\prime}=W^{\prime}(w_{0})\quad\text{in }\mathbb{R}_{+},\\
w_{0}(0)=\alpha.
\end{array}
\right.  \label{1DLimitODE}%
\end{equation}
Furthermore, by (\ref{1d truncated}) we know that $a\leq w_{0}\leq b$, which
by \eqref{1DLimitODE} implies that $|w_{0}^{\prime\prime}|\leq C$. Also, by
(\ref{energies bounded}), the fact that $H_{\varepsilon}(w_{\varepsilon
})=G_{\varepsilon}(v_{\varepsilon})$,
\begin{align*}
\omega(0)\int_{0}^{l}((w_{0}^{\prime})^{2}+W(w_{0}))\,ds  &  \leq
\lim_{\varepsilon\rightarrow0}\int_{0}^{l}((w_{\varepsilon}^{\prime}%
)^{2}+W(w_{\varepsilon}))\omega_{\varepsilon}\,ds\\
&  \leq\lim_{\varepsilon\rightarrow0^{+}}H_{\varepsilon}(w_{\varepsilon
})=\operatorname*{d}\nolimits_{W}(\alpha,b)\omega(0)
\end{align*}
for every $l\in\mathbb{N}$, and thus
\begin{equation}
\int_{0}^{\infty}((w_{0}^{\prime})^{2}+W(w_{0}))\,ds\leq\operatorname*{d}%
\nolimits_{W}(\alpha,b). \label{rescaledLimitEnergyBound}%
\end{equation}
If $\alpha=b$, then this inequality implies that $w_{0}\equiv b$. Otherwise,
if $\alpha<b$, (\ref{rescaledLimitEnergyBound}) combined with the fact that
$|w_{0}^{\prime\prime}(t)|\leq C$ for all $t\in\mathbb{R}_{+}$ (by
\eqref{1DLimitODE}) implies that $\lim_{s\rightarrow+\infty}w_{0}^{\prime
}(s)=0$. In turn, $|w_{0}^{\prime}(t)|\leq C$ for all $t\in\mathbb{R}_{+}$,
and since%
\[
\liminf_{t\rightarrow\infty}W(w_{0})=0
\]
in view (\ref{rescaledLimitEnergyBound}), we have that%
\[
\text{either}\quad\lim_{t\rightarrow\infty}w_{0}(t)=a\quad\text{or\quad}%
\lim_{t\rightarrow\infty}w_{0}(t)=b.
\]
By integrating \eqref{1DLimitODE} we find that
\begin{equation}
(w_{0}^{\prime})^{2}=W(w_{0}). \label{rescaledODE2}%
\end{equation}
We now distinguish two cases. If $\beta_{-}<\alpha<b$, we define $R_{\varepsilon
}=0$. On the other hand, if $\alpha\leq\beta_{-}$, then by Theorem
\ref{theorem 1d properties minimizers}, we have that $R_{\varepsilon}\leq
C\varepsilon$, where $R_{\varepsilon}$ is the first time in $[0,T_{\varepsilon
}]$ such that $v_{\varepsilon}=\max\{\alpha_{\varepsilon},\beta_{-}\}$. Hence,
in both cases $w_{\varepsilon}(\varepsilon^{-1}R_{\varepsilon})=v_{\varepsilon
}(R_{\varepsilon})\geq\beta_{-}$. Since $\varepsilon^{-1}R_{\varepsilon}\leq
C$, by extracting a subsequence, we can assume that $\varepsilon
^{-1}R_{\varepsilon}\rightarrow s_{0}$. In turn, $w_{0}(s_{0})=\max
\{a,\beta_{-}\}$. It follows from (\ref{rescaledODE2}), that $w_{0}^{\prime
}(s_{0})=W^{1/2}(\max\{a,\beta_{-}\})>0$. \ Hence, $w_{0}$ is increasing after
$s_{0}$ and so it is the unique solution to the Cauchy problem%
\[
\left\{
\begin{array}
[c]{l}%
w_{0}^{\prime}=W^{1/2}(w_{0}),\\
w_{0}(s_{0})=\max\{a,\beta_{-}\}.
\end{array}
\right.
\]
By uniqueness, it follows that $a<w_{0}(s)<b$ for all $s$, which means that
$w_{0}$ is strictly increasing. In turn,%
\[
\left\{
\begin{array}
[c]{l}%
w_{0}^{\prime}=W^{1/2}(w_{0}),\\
w_{0}(0)=\alpha,
\end{array}
\right.
\]
and%
\[
\lim_{t\rightarrow\infty}w_{0}(t)=b.
\]
This shows that $w_{0}=z_{\alpha}$. Using the fact that $\{\varepsilon
_{n}\}_{n}$ was an arbitrary sequence, the statement of the lemma follows. \hfill
\end{proof}

Next, we will use the previous lemmas to derive a second-order liminf
inequality, which immediately implies Theorem \ref{theorem liminf second 1d}.

\begin{proof}
[Proof of Theorem \ref{theorem liminf second 1d}]By Theorem
\ref{theorem 1d properties minimizers}, we have $v_{\varepsilon
}(T_{\varepsilon})\geq\beta_{\varepsilon}-\varepsilon^{k}$ for all
$0<\varepsilon<\varepsilon_{0}$, where $T_{\varepsilon}=0$ if $v_{\varepsilon
}>\beta_{ \varepsilon}-\varepsilon^{k}$ in $I$ and otherwise $T_{\varepsilon}$ is the
first time that $v_{\varepsilon}=\beta_{\varepsilon}-\varepsilon^{k}$,
and we have $T_{\varepsilon}\leq C\varepsilon|\log\varepsilon|$. Moreover,
$v_{\varepsilon}(t)\in\lbrack\beta_{-},\beta_{\varepsilon}-\varepsilon^{k}]$
for $t\in\lbrack R_{\varepsilon},T_{\varepsilon}]$, where $R_{\varepsilon
}<T_{\varepsilon}$ is either the first time such that $v_{\varepsilon}%
=\beta_{-}$ or $R_{\varepsilon}=0$ and $v_{\varepsilon}>\beta_{-}$ in $I$.

Here, $\varepsilon_{0}$ and $C$ depend only on $\alpha_{-}$, $\beta_{-}$,
$A_{0}$, $B_{0}$, $T$, $\omega$, $W$. In what follows, we will take
$\varepsilon_{0}$ smaller and $C$ larger, if necessary, preserving the same
dependence on the parameters of the problem.

Setting $l_{\varepsilon}:=\varepsilon^{-1}T_{\varepsilon}$, we have that
\begin{equation}
w_{\varepsilon}(l_{\varepsilon})\geq\beta_{\varepsilon}-\varepsilon^{k},
\label{b epsilon}%
\end{equation}
for $0<\varepsilon<\varepsilon_{0}$, where
\begin{equation}
l_{\varepsilon}\leq C|\log\varepsilon|. \label{KepsilonBound}%
\end{equation}
By \eqref{rescaledProblemFormulation} we have
\begin{align*}
&  \frac{H_{\varepsilon}(w_{\varepsilon})-\operatorname*{d}\nolimits_{W}%
(\alpha,b)\omega(0)}{\varepsilon}\\
&  =\varepsilon^{-1}\int_{0}^{l_{\varepsilon}}(W^{1/2}(w_{\varepsilon
})-w_{\varepsilon}^{\prime})^{2}\omega_{\varepsilon}\,ds+2\varepsilon^{-1}%
\int_{0}^{l_{\varepsilon}}W^{1/2}(w_{\varepsilon})w_{\varepsilon}^{\prime
}(\omega_{\varepsilon}-\omega(0))\,ds\\
&  \quad+\varepsilon^{-1}\int_{l_{\varepsilon}}^{T\varepsilon^{-1}}\left(
W(w_{\varepsilon})+(w_{\varepsilon}^{\prime})^{2}\right)  \omega_{\varepsilon
}\,ds+\varepsilon^{-1}\omega(0)\left(  2\int_{0}^{l_{\varepsilon}}%
W^{1/2}(w_{\varepsilon})w_{\varepsilon}^{\prime}\,ds-\operatorname*{d}%
\nolimits_{W}(\alpha,b)\right) \\
&  \geq2\varepsilon^{-1}\int_{0}^{l_{\varepsilon}}W^{1/2}(w_{\varepsilon
})w_{\varepsilon}^{\prime}(\omega_{\varepsilon}-\omega(0))\,ds+\varepsilon
^{-1}\omega(0)\left(  2\int_{0}^{l_{\varepsilon}}W^{1/2}(w_{\varepsilon
})w_{\varepsilon}^{\prime}\,ds-\operatorname*{d}\nolimits_{W}(\alpha,b)\right)
\\
&  =:\mathcal{A}+\mathcal{B}.
\end{align*}
To estimate $\mathcal{B}$, observe that by the change of variables
$r=w_{\varepsilon}(s)$, we obtain
\begin{align*}
2\int_{0}^{l_{\varepsilon}}W^{1/2}(w_{\varepsilon})w_{\varepsilon}^{\prime
}\,ds  &  =2\int_{\alpha_{\varepsilon}}^{w_{\varepsilon}(l_{\varepsilon}%
)}W^{1/2}(r)\,dr\\
&  =2\int_{\alpha}^{b}W^{1/2}(r)\,dr-2\int_{\alpha_{\varepsilon}}^{\alpha
}W^{1/2}(r)\,dr-2\int_{w_{\varepsilon}(l_{\varepsilon})}^{b}W^{1/2}(r)\,dr.
\end{align*}
By (\ref{alpha epsilon and beta epsilon}),%
\[
2\left\vert \int_{\alpha_{\varepsilon}}^{\alpha}W^{1/2}(r)\,dr\right\vert
\leq2\max_{[a,b]}W^{1/2}|\alpha_{\varepsilon}-\alpha|\leq C\varepsilon
^{\gamma}.
\]
On the other hand, by (\ref{W near b}), (\ref{alpha epsilon and beta epsilon}%
), and (\ref{b epsilon}),%
\[
2\int_{w_{\varepsilon}(l_{\varepsilon})}^{b}W^{1/2}(r)\,dr\leq C\int%
_{w_{\varepsilon}(l_{\varepsilon})}^{b}\left(  b-r\right)
\,dr=C(b-w_{\varepsilon}(l_{\varepsilon}))^{2}\leq C(\varepsilon
^{2k}+\varepsilon^{2\gamma}).
\]
Hence,
\[
\mathcal{B}\geq-C(\varepsilon^{\gamma}+\varepsilon^{2k}).
\]
To estimate $\mathcal{A}$, we use Taylor's formula and assumption
\eqref{etaSmooth} to get
\[
|\omega_{\varepsilon}(s)-\omega(0)-\varepsilon s\omega^{\prime}(0)|\leq
|\omega^{\prime}|_{C^{0,d}}\varepsilon^{1+d}|s|^{1+d}.
\]
Using (\ref{1d truncated}), Lemma \ref{profilesConverge} and
\eqref{KepsilonBound}, we have 
\[
\left\vert \varepsilon^{-1}\int_{0}^{l_{\varepsilon}}W^{1/2}(w_{\varepsilon
})w_{\varepsilon}^{\prime}\varepsilon^{1+d}|s|^{1+d}\,ds\right\vert \leq
C\varepsilon^{d}|\log\varepsilon|^{2+d}.
\]
Thus, we find that
\begin{align}
\mathcal{A}  &  \geq2\omega^{\prime}(0)\int_{0}^{l_{\varepsilon}}%
W^{1/2}(w_{\varepsilon})w_{\varepsilon}^{\prime}s\,ds-C\varepsilon^{d}%
|\log\varepsilon|^{2+d}\label{711}\\
&  =2\omega^{\prime}(0)\int_{0}^{l}W^{1/2}(w_{\varepsilon})w_{\varepsilon
}^{\prime}s\,ds+2\omega^{\prime}(0)\int_{l}^{l_{\varepsilon}}W^{1/2}%
(w_{\varepsilon})w_{\varepsilon}^{\prime}s\,ds-C\varepsilon^{d}|\log
\varepsilon|^{2+d}\nonumber\\
&  =:\mathcal{A}_{1}+\mathcal{A}_{2}-C\varepsilon^{d}|\log\varepsilon
|^{2+d},\nonumber
\end{align}
where $l$ is fixed.

To estimate $\mathcal{A}_{2}$, we distinguish two cases. If $l_{\varepsilon
}\leq l$, then we use the fact that $v_{\varepsilon}(t)\geq b-\tau
_{0}\varepsilon^{1/2}$ for all $t\in\lbrack T_{\varepsilon},T]$ to obtain
\[
0\leq b-v_{\varepsilon}(t)\leq\tau\varepsilon^{1/2}.
\]
In turn, by (\ref{W near b}),
\[
W^{1/2}(w_{\varepsilon}(s))\leq\sigma^{-1}(b-w_{\varepsilon}(s))\leq
C\varepsilon^{1/2}%
\]
for all $s\in\lbrack l_{\varepsilon},l]$. Hence, also by Lemma
\ref{profilesConverge},%
\[
\mathcal{A}_{2}\geq-2|\omega^{\prime}(0)|\int_{l_{\varepsilon}}^{l}%
W^{1/2}(w_{\varepsilon})|w_{\varepsilon}^{\prime}|s\,ds\geq-C\varepsilon
^{1/2} l^2.
\]
On the other hand, if $l_{\varepsilon}>l$, since $v_{\varepsilon}(t)\in
\lbrack\beta_{-},\beta_{\varepsilon}-\varepsilon^{k}]$ for $t\in\lbrack
R_{\varepsilon},T_{\varepsilon}]$, where $R_{\varepsilon}<T_{\varepsilon}$ is
either the first time such that $v_{\varepsilon}=\beta_{-}$ or $R_{\varepsilon
}=0$ and $v_{\varepsilon}>\beta_{-}$ in $I$, we have that $[R_{\varepsilon
},T_{\varepsilon}]$ is a maximal interval of the set $ B^k_{\varepsilon}$ defined
in (\ref{1d set B epsilon}), and so by  Theorem~\ref{thorem barrier 1d}, and
(\ref{alpha epsilon and beta epsilon}),%
\begin{align*}
b-v_{\varepsilon}(t)  &  \leq(b-v_{\varepsilon}(R_{\varepsilon}))e^{-\mu
(t-R_{\varepsilon})\varepsilon^{-1}}+(b-v_{\varepsilon}(T_{\varepsilon
}))e^{-\mu(T_{\varepsilon}-t)\varepsilon^{-1}}\\
&  \leq(b-v_{\varepsilon}(R_{\varepsilon}))e^{-\mu(t-R_{\varepsilon
})\varepsilon^{-1}}+(b-\beta_{\varepsilon}+\varepsilon^{k})\\
&  \leq be^{-\mu(t-R_{\varepsilon})\varepsilon^{-1}}+B_{0}\varepsilon^{\gamma
}+\varepsilon^{k}%
\end{align*}
for $t\in\lbrack R_{\varepsilon},T_{\varepsilon}]$. By Theorem
\ref{theorem 1d properties minimizers}, we have that $R_{\varepsilon}\leq
C\varepsilon$. It follows that $r_{\varepsilon}:=\varepsilon^{-1}%
R_{\varepsilon}\leq C$ and
\[
0\leq b-w_{\varepsilon}(s)\leq be^{-\mu(s-r_{\varepsilon})}+B_{0}%
\varepsilon^{\gamma}+\varepsilon^{k}%
\]
for all $s\in\lbrack r_{\varepsilon},l_{\varepsilon}]$ and all $0<\varepsilon
<\varepsilon_{0}$. Using (\ref{W near b}) and Lemma \ref{profilesConverge}, we
have 
\begin{align*}
\mathcal{A}_{2}  &  \geq-C\int_{l}^{l_{\varepsilon}}(b-w_{\varepsilon
})s\,ds\geq-C\rho\int_{l}^{\infty}e^{-\mu(s-r_{\varepsilon})}%
s\,ds+C(\varepsilon^{\gamma}+\varepsilon^{k})l_{\varepsilon}^{2}\\
&  \geq-Ce^{\mu r_{\varepsilon}}e^{-l\mu}\left(  l\mu+1\right)  -C(\varepsilon
^{\gamma}+\varepsilon^{k})l_{\varepsilon}^{2}\\
&  \geq-Ce^{-l\mu}\left(  l\mu+1\right)  -C(\varepsilon^{\gamma}%
+\varepsilon^{k})\log^{2}\varepsilon
\end{align*}
where we used (\ref{KepsilonBound}) and the fact that $r_{\varepsilon}\leq C$
and we take $l>C\geq r_{\varepsilon}$. Using this estimate in (\ref{711})
gives%
\[
\mathcal{A}\geq2\omega^{\prime}(0)\int_{0}^{l}W^{1/2}(w_{\varepsilon
})w_{\varepsilon}^{\prime}s\,ds-Ce^{-l\mu}\left(  l\mu+1\right)
-C l^2\varepsilon^{1/2}-C\varepsilon^{\gamma_{1}}|\log\varepsilon|^{2+d},
\]
where $\gamma_{1}=\min\{d,\gamma,k\}$. Combining the estimates for
$\mathcal{A}$ and $\mathcal{B}$ gives%
\begin{equation}\begin{split}
\frac{H_{\varepsilon}(w_{\varepsilon})-\operatorname*{d}\nolimits_{W}%
(\alpha,b)\omega(0)}{\varepsilon}\geq 2\omega^{\prime}(0)&\int_{0}^{l}%
W^{1/2}(w_{\varepsilon})w_{\varepsilon}^{\prime}s\,ds-Ce^{-l\mu}\left(
l\mu+1\right)  \\ &-C l^2\varepsilon^{1/2}-C\varepsilon^{\gamma_{1}}|\log
\varepsilon|^{2+d} \end{split}\label{limsup H unifform}%
\end{equation}
for all $0<\varepsilon<\varepsilon_{0}$ and all $l>C$.

By \eqref{rescaledWeakConvergence}, we can write
\[
\lim_{\varepsilon\rightarrow0^{+}}\int_{0}^{l}W^{1/2}(w_{\varepsilon
})w_{\varepsilon}^{\prime}s\,ds=\int_{0}^{l}W^{1/2}(z_{\alpha})z_{\alpha
}^{\prime}s\,ds.
\]
Taking first $\varepsilon\rightarrow0^{+}$ and then $l\rightarrow\infty$ and
using the Lebesgue dominated convergence theorem and (\ref{estimate z alpha})
gives
\begin{equation}
\liminf_{\varepsilon\rightarrow0^{+}}\frac{H_{\varepsilon}(w_{\varepsilon
})-\operatorname*{d}\nolimits_{W}(\alpha,b)\omega(0)}{\varepsilon}\geq
2\int_{0}^{\infty}W^{1/2}(z_{\alpha})z_{\alpha}^{\prime}s\,ds\,\omega^{\prime
}(0). \label{liminf H epsilon}%
\end{equation}
Since $H(w_{\varepsilon})=G_{\varepsilon}^{(1)}(v_{\varepsilon})$, this
concludes the proof. \hfill
\end{proof}

\section{Properties of Minimizers of $F_{\varepsilon}$}

\label{section minimizers}In this section,  we study qualitative properties of
critical points and minimizers of the functional $F_{\varepsilon}$ given in
(\ref{functional cahn-hilliard}) and subject to the Dirichlet boundary
conditions (\ref{dirichlet boundary conditions}).

The following theorem is the analog of Lemma 4.3 in Sternberg and Zumbrun
\cite{sternberg-zumbrun1998}. Here, we replace the mass constraint with
Dirichlet boundary conditions.

 Recall \eqref{Omega delta}.

\begin{theorem}
\label{theorem decay sz}Let $\Omega\subset\mathbb{R}^{N}$ be an open, bounded,
connected set with boundary of class $C^{2,d}\ $, $0<d\leq1$. Assume that $W$
satisfies \eqref{W_Smooth}-\eqref{W' three zeroes} and that
$g_{\varepsilon}$ satisfy \eqref{bounds g},
\eqref{g epsilon smooth}-\eqref{g epsilon -g bound}. Let $u_{\varepsilon}\in
H^{1}(\Omega)$ be a critical point of \eqref{functional cahn-hilliard} subject
to the Dirichlet boundary condition \eqref{dirichlet boundary conditions}.
Then
\begin{equation}
a\leq u_{\varepsilon}(x)\leq b\quad\text{for all }x\in\Omega.
\label{u between a and b}%
\end{equation}
Moreover, for every
\begin{equation}
0<\rho<b-c, \label{rho restrictions}%
\end{equation}
there exist $\mu_{\rho}>0$ and $C_{\rho}>0$, independent of $\varepsilon$,
such that for all $\varepsilon$ sufficiently small the following estimates
hold%
\begin{equation}
0\leq b-u_{\varepsilon}(x)\leq C_{\rho}e^{-\mu_{\rho}\operatorname*{dist}%
(x,K_{\rho})/(2\varepsilon)}\quad\text{for }x\in\Omega\setminus K_{\rho},
\label{decay estimate b}%
\end{equation}
where $K_{\rho}:=\{x\in\Omega:\,u_{\varepsilon}(x)\leq b-\rho\}\cup
\Omega_{\varepsilon|\log\varepsilon|}$.
\end{theorem}

In the definition of the set $K_\rho$ the choice of $\Omega_{\varepsilon|\log\varepsilon|}$ is somewhat arbitrary. In what follows (see Theorem \ref{theorem uniform convergence level sets}), the choice $\Omega_r$ for $r>0$ small would have sufficed.

The proof relies on the following proposition, which is essentially due to
Sternberg and Zumbrun \cite[Proposition 4.1]{sternberg-zumbrun1998}.

\begin{proposition}
\label{proposition sz}Let $\Omega\subset\mathbb{R}^{N}$ be an open, bounded,
connected set with $C^{1}$ boundary and let $K\subset\Omega$ be a compact set.
Suppose that $v:\overline{\Omega}\rightarrow\mathbb{R}$ is a function in
$C^{2}(\Omega)\cap C^{1}(\overline{\Omega})$ satisfying the conditions%
\begin{equation}
\left\{
\begin{array}
[c]{ll}%
\varepsilon^{2}\Delta v\geq\mu^{2}v & \text{in }\Omega\setminus K,\\
v\leq M & \text{on }\partial\Omega\cup\partial K,
\end{array}
\right.  \label{sz1a}%
\end{equation}
where $\mu>0$ and $M$ is a positive constant (not necessarily independent of
$\varepsilon$). Then there exists a constant $C_{0}$ independent of
$\varepsilon$ such that%
\[
v(x)\leq C_{0}Me^{-\mu\operatorname*{dist}(x,K)/(2\varepsilon)}\quad\text{for
}x\in\Omega\setminus K
\]
for all $\varepsilon>0$ sufficiently small.
\end{proposition}

\begin{proof}
By the maximum principle, $v\leq M$ in $\Omega\setminus K$. Let%
\[
K_{\varepsilon}:=\{x\in\Omega:\,\operatorname*{dist}(x,K)\leq\varepsilon\}.
\]
Consider the radial function $\phi(x):=e^{-\mu|x|/\varepsilon}$. Letting
$r=\left\vert x\right\vert $, we have that%
\begin{equation}
\varepsilon^{2}\Delta\phi=\varepsilon^{2}\left(  \frac{\partial^{2}\phi
}{\partial r^{2}}+\frac{N-1}{r}\frac{\partial\phi}{\partial r}\right)
=\mu^{2}e^{-\mu r/\varepsilon}-\frac{N-1}{r}\varepsilon\mu e^{-\mu
r/\varepsilon}\leq\mu^{2}\phi. \label{sz1}%
\end{equation}
Define%
\[
q(x):=M_{1}\varepsilon^{-N}\int_{K_{\varepsilon}}e^{-\mu|x-y|/\varepsilon
}dy,\quad x\in\overline{\Omega\setminus K_{\varepsilon}},
\]
where $M_{1}>0$ is to be determined. If $x\in\Omega\setminus K_{\varepsilon}$,
then we can differentiate under the integral sign and use (\ref{sz1}) to find%
\begin{equation}
\varepsilon^{2}\Delta q\leq\mu^{2}q\quad\text{in }\Omega\setminus
K_{\varepsilon}. \label{sz2}%
\end{equation}
If $x_{0}\in\overline{\Omega}$ and $\operatorname*{dist}(x_{0},K)\geq
\varepsilon$, then there exists $y_{0}\in K$ such that $|x_{0}-y_{0}%
|=\operatorname*{dist}(x_{0},K)\geq\varepsilon$. In particular, $\Omega\cap
B(y_{0},\varepsilon)\subseteq K_{\varepsilon}$, and hence,%
\begin{equation}
q(x_{0})\geq M_{1}\varepsilon^{-N}\int_{\Omega\cap B(y_{0},\varepsilon
)}e^{-\mu|x_{0}-y|/\varepsilon}dy\geq M_{1}\varepsilon^{-N}e^{-\mu
\operatorname*{dist}(x,K)/\varepsilon}e^{-\mu}|\Omega\cap B(y_{0}%
,\varepsilon)|, \label{sz2a}%
\end{equation}
where we used the fact that if $y\in B(y_{0},\varepsilon)$, then
$|y-x_{0}|\leq|y-y_{0}|+|y_{0}-x_{0}|<\operatorname*{dist}(x_{0}%
,K)+\varepsilon$. If $B(y_{0},\varepsilon/2)\subseteq\Omega$, then
$|\Omega\cap B(y_{0},\varepsilon)|\geq|B(y_{0},\varepsilon/2)|=\alpha
_{N}2^{-N}\varepsilon^{N}$, where $\alpha_{N}=|B(0,1)|$. Otherwise, there
exists $y_{1}\in B(y_{0},\varepsilon/2)\cap\partial\Omega$. Since
$\partial\Omega$ is of class $C^{1}$, it is Lipschitz continuous, hence, by
taking by taking $\varepsilon$ sufficiently small, we can find a cone
$K_{y_{1},\varepsilon}$ with vertex $y_{1}$
and vertex angle depending on the Lipschitz constant associated to $\Omega$
such that $K_{y_{1},\varepsilon}\cap B(y_1,\varepsilon/2)\subseteq\Omega\cup\{y_{1}\}$. Since
$y_{1}\in B(y_{0},\varepsilon/2)$, we have that $K_{y_{1},\varepsilon
}\cap B(y_1,\varepsilon/2)\subseteq (\Omega\cap  B(y_{0},\varepsilon))\cup \{y_1\}$, and so
\[
\mathcal{L}^{N}(\Omega\cap B(y_{0},\varepsilon))\geq\mathcal{L}^{N}%
(K_{y_{1},\varepsilon} \cap B(y_1,\varepsilon/2))=c_{0}\varepsilon^{N}.
\]
This shows that $\mathcal{L}^{N}(\Omega\cap B(y_{0},\varepsilon))\geq
\min\{c_{0},\alpha_{N}2^{-N}\}\varepsilon^{ N}$. Take
\[
M_{1}:=Me^{2\mu}/\min\{c_{0},\alpha_{N}2^{-N}\}.
\]
Observe that if $x_{0}\in\partial K_{\varepsilon}\cap\Omega$, then
$\operatorname*{dist}(x_{0},K)=\varepsilon$, and so by (\ref{sz2a}),
\begin{equation}
q(x_{0})\geq M_{1}e^{-2\mu}\min\{c_{0},\alpha_{N}2^{-N}\}\geq M, \label{sz3}%
\end{equation}
while if $x_{0}\in\partial\Omega\setminus K_{\varepsilon}$, then
\begin{equation}
q(x_{0})\geq M_{1}e^{-\mu\operatorname*{dist}(x,K)/\varepsilon}e^{-\mu}%
\min\{c_{0},\alpha_{N}2^{-N}\}\geq Me^{\mu-\mu\operatorname*{dist}%
(x,K)/\varepsilon} \label{sz4}%
\end{equation}
Next we estimate $q$ from above on $\Omega\setminus K_{\varepsilon}$. If
$x_{0}\in\Omega\setminus K_{\varepsilon}$, then $|x_{0}-y|\geq
\operatorname*{dist}(x_{0},K_{\varepsilon})$ for all $y\in K_{\varepsilon}$,
and so,%
\begin{align}
q(x_{0})  &  =M_{1}\varepsilon^{-N}\int_{K_{\varepsilon}}e^{-\mu
|x_{0}-y|/(2\varepsilon)}e^{-\mu|x_{0}-y|/(2\varepsilon)}dy\nonumber\\
&  \leq M_{1}\varepsilon^{-N}e^{-\mu\operatorname*{dist}(x_{0},K_{\varepsilon
})/(2\varepsilon)}\int_{K_{\varepsilon}}e^{-\mu|x_{0}-y|/(2\varepsilon
)}dy\label{sz5}\\
&  \leq M_{1}\varepsilon^{-N}e^{-\mu\operatorname*{dist}(x_{0},K_{\varepsilon
})/(2\varepsilon)}\int_{\mathbb{R}^{N}\setminus B(x_{0},\operatorname*{dist}%
(x_{0},K_{\varepsilon}))}e^{-\mu|x_{0}-y|/(2\varepsilon)}dy\nonumber\\
&  =M_{1}\varepsilon^{-N}e^{-\mu\operatorname*{dist}(x_{0},K_{\varepsilon
})/(2\varepsilon)}\beta_{N}\int_{\operatorname*{dist}(x_{0},K_{\varepsilon}%
)}^{\infty}e^{-\mu r/(2\varepsilon)}r^{N-1}dr\nonumber\\
&  \leq M_{1}e^{-\mu\operatorname*{dist}(x_{0},K_{\varepsilon})/(2\varepsilon
)}\beta_{N}\int_{0}^{\infty}e^{-\mu t/2}t^{N-1}dt=:MC_{0}e^{-\mu
\operatorname*{dist}(x_{0},K_{\varepsilon})/(2\varepsilon)},\nonumber
\end{align}
where  we set $\beta_N:= \mathcal{H}^{N-1}(\mathbb{S}^{N-1})$, we used spherical coordinates, and made the change of variables
$t:= r/\varepsilon$.

If we now define $w:=v-q$, by (\ref{sz1a}) and (\ref{sz2}), we have that
$\varepsilon^{2}\Delta w\geq\mu^{2}w$ in in $\Omega\setminus K_{\varepsilon}$,
while by the fact that $v\leq M$ in $\Omega\setminus K$ and (\ref{sz3}),
$w\leq0$ in $\partial K_{\varepsilon}\cap\Omega$. Finally, if $x\in
\partial\Omega\setminus K_{\varepsilon}$, by (\ref{sz1a}) and (\ref{sz4}),
\[
w(x)\leq M-Me^{\mu
-\mu\operatorname*{dist}(x,K)/\varepsilon}\leq0.
\]
Hence, we have shown that%
\[
\left\{
\begin{array}
[c]{ll}%
\varepsilon^{2}\Delta w\geq\mu^{2}w & \text{in }\Omega\setminus K_{\varepsilon
},\\
w\leq0 & \text{on }\partial(\Omega\setminus K_{\varepsilon}).
\end{array}
\right.
\]
It follows from the maximum principle that $w\leq0$ in $\Omega\setminus
K_{\varepsilon}$, that is,%
\[
v(x)\leq q(x)\leq MC_{0}e^{-\mu\operatorname*{dist}(x,K_{\varepsilon
})/(2\varepsilon)},\quad x\in\Omega\setminus K_{\varepsilon}.
\]
Finally, observe that if $x\in\Omega\setminus K$, then
$\operatorname*{dist}(x,K)\leq\operatorname*{dist}%
(x,K_{\varepsilon})+\varepsilon$, and so%
\[
v(x)\leq MC_{0}e^{-\mu/2}e^{-\mu\operatorname*{dist}(x,K)/(2\varepsilon
)},\quad x\in\Omega\setminus K_{\varepsilon}.
\]
On the other hand, if $x\in K\setminus K_{\varepsilon}$, then
$\operatorname*{dist}(x,K)\leq\varepsilon$ and so, using the fact that $v\leq
M$ in $\Omega\setminus K$, we have
\[
v(x)\leq M\leq M\frac{e^{-\mu\operatorname*{dist}(x,K)/(2\varepsilon)}%
}{e^{-\mu/2}},
\]
which concludes the proof. \hfill
\end{proof}

We turn to the proof of Theorem \ref{theorem decay sz}.

\begin{proof}
[Proof of Theorem \ref{theorem decay sz}]To prove (\ref{u between a and b}),
assume that there exists $x_{0}\in\Omega$ such that $u_{\varepsilon}%
(x_{0})\geq b$. Assume first that $u_{\varepsilon}(x_{0})>b$. Since
$W^{\prime}(s)>0$ for $s>b$ and $u_{\varepsilon}\leq b$ on $\partial\Omega$,
we can assume that $u_{\varepsilon}$ achieves its maximum at $x_{0}$. But
then,%
\[
0\geq\Delta u_{\varepsilon}(x_{0})=\frac{1}{2\varepsilon^{2}}W^{\prime
}(u_{\varepsilon}(x_{0}))>0.
\]
Similarly, we can conclude that $u_{\varepsilon}\geq a$.

Next, let $v:=b-u_{\varepsilon}$. In $\Omega\setminus K_{\rho}$, we have that%
\[
\varepsilon^{2}\Delta v=-\frac{1}{2}\frac{W^{\prime}(u_{\varepsilon}%
)}{b-u_{\varepsilon}}v=\frac{1}{2}\frac{W^{\prime}(b)-W^{\prime}%
(u_{\varepsilon})}{b-u_{\varepsilon}}v\geq\mu_{\rho}^{2}v,
\]
where%
\[
\mu_{\rho}^{2}:=\frac{1}{2}\sup_{b-\rho\leq s<b}\frac{W^{\prime}(b)-W^{\prime
}(s)}{b-s}>0
\]
by (\ref{WPrime_At_Wells}), (\ref{W' three zeroes}), and
(\ref{rho restrictions}). Taking $M:= b$, we can apply Proposition
\ref{proposition sz} to obtain (\ref{decay estimate b}).\hfill
\end{proof}

\begin{remark}
If $W$ is symmetric with respect to $c$, or, more generally, if $c:=\frac
{a+b}{2}$, then
\[
a<u_{\varepsilon}(x)<b
\]
for all $x\in\Omega$. To see this, suppose first that $a=-1$, $c=0$, and
$b=1$. Assume that there exists $x_{0}\in\Omega$ such that $u_{\varepsilon
}(x_{0})=\pm1$. Let%
\[
v_{\varepsilon}:= u_{\varepsilon}^{2}-1.
\]
Then for all $x\in\Omega$ such that $-1<u_{\varepsilon}(x)<1$,
\begin{align*}
\Delta v_{\varepsilon}  &  =\Delta(u_{\varepsilon}^{2})=2\sum_{i=1}^{N}%
\frac{\partial}{\partial x_{i}}\left(  u_{\varepsilon}\frac{\partial
u_{\varepsilon}}{\partial x_{i}}\right)  =2\sum_{i=1}^{N}\left(
\frac{\partial u_{\varepsilon}}{\partial x_{i}}\right)  ^{2}+2u_{\varepsilon
}\Delta u_{\varepsilon}\\
&  =2|\nabla u_{\varepsilon}|^{2}+\frac{u_{\varepsilon}}{\varepsilon^{2}%
}W^{\prime}(u_{\varepsilon})\geq\frac{u_{\varepsilon}}{\varepsilon^{2}}%
\frac{W^{\prime}(u_{\varepsilon})}{u_{\varepsilon}^{2}-1}v_{\varepsilon}.
\end{align*}
Since $W^{\prime}(s)>0$ for $-1<s<0$ and $W^{\prime}(s)<0$ for $0<s<1$, we
have that $\frac{sW^{\prime}(s)}{s^{2}-1}\geq0$. Moreover,
\[
\frac{sW^{\prime}(s)}{s^{2}-1}=\frac{s}{s+1}\frac{W^{\prime}(s)-W^{\prime}%
(1)}{s-1}\rightarrow\frac{1}{2}W^{\prime\prime}(1)>0
\]
as $s\rightarrow1^{-}$ and
\[
\frac{sW^{\prime}(s)}{s^{2}-1}=\frac{s}{s-1}\frac{W^{\prime}(s)-W^{\prime
}(-1)}{s+1}\rightarrow\frac{1}{2}W^{\prime\prime}(-1)>0
\]
as $s\rightarrow-1^{+}$. Hence, by defining%
\[
c_{\varepsilon}(x):=\left\{
\begin{array}
[c]{cc}%
\frac{u_{\varepsilon}}{\varepsilon^{2}}\frac{W^{\prime}(u_{\varepsilon}%
)}{u_{\varepsilon}^{2}-1} & \text{if }-1<u_{\varepsilon}(x)<1,\\
\frac{1}{2\varepsilon^{2}}W^{\prime\prime}(1) & \text{if }u_{\varepsilon
}(x)=1,\\
\frac{1}{2 \varepsilon^{2}}W^{\prime\prime}(-1) & \text{if }u_{\varepsilon}(x)=-1,
\end{array}
\right.
\]
we have that%
\[
\Delta v_{\varepsilon}\geq c_{\varepsilon}(x)v_{\varepsilon}(x),
\]
where $c_{\varepsilon}(x)\geq0$. Moreover, $v_{\varepsilon}=u_{\varepsilon
}^{2}-1=g_{\varepsilon}^{2}-1<0$ on $\partial\Omega$. Since $v_{\varepsilon
}(x_{0})=0$, it follows from \cite[Theorem 4, Chapter 6]{evans-book2010} that
$v_{\varepsilon}$ is constant in $\Omega$, which is a contradiction.

To remove the additional condition that $a=-1$ and $b=1$, it suffices to
replace $W$ with
\[
\bar{W}(r):= W\left(  \frac{b-a}{2}r+\frac{a+b}{2}\right)
\]
and $u_{\varepsilon}$ with $\bar{u}_{\varepsilon}:=\frac{2}{b-a}u_{\varepsilon
}-\frac{a+b}{b-a}$.
\end{remark}

\bigskip

\begin{theorem}
\label{theorem gradient estimate}Let $\Omega\subset\mathbb{R}^{N}$ be an open,
bounded, connected set with boundary of class $C^{2,d}\ $, $0<d\leq1$. Assume
that $W$ satisfies \eqref{W_Smooth}-\eqref{W' three zeroes} and
that $g_{\varepsilon}$ satisfy \eqref{bounds g},
\eqref{g epsilon smooth}-\eqref{g epsilon -g bound}. Let $u_{\varepsilon}\in
H^{1}(\Omega)$ be a minimizer of \eqref{functional cahn-hilliard} subject to
the Dirichlet boundary condition \eqref{dirichlet boundary conditions}. Then
there exists a constant $C>0$ depending only on $N$ such that
\[
|\nabla u_{\varepsilon}(x)|\leq\frac{C}{\varepsilon}\quad\text{for all }%
x\in\Omega\setminus\Omega_{\varepsilon}.
\]
Moreover, under the additional hypothesis that $g_{\varepsilon}\in C^{2}(\overline
{\Omega})$, with
\[
\Vert\nabla^{2}g_{\varepsilon}\Vert_{L^{\infty}(\Omega)}\leq\frac
{M}{\varepsilon^{2}},
\]
for some constant $M>0$, there exists a constant $C>0$ depending only on $\Omega$ and $M$, such that
\[
|\nabla u_{\varepsilon}(x)|\leq\frac{C}{\varepsilon}\quad\text{for all }%
x\in\Omega.
\]

\end{theorem}

The following lemma is due to Bethuel, Brezis, and H\'{e}lein \cite[Lemma
A.1, Lemma A.2]{bethuel-brezis-helein1993}.

\begin{lemma}
\label{lemma bbh interior}Let $\Omega\subset\mathbb{R}^{N}$ be an open, bounded set, and let $f\in L^{\infty}(\Omega)$. Assume that $u\in H^{1}%
(\Omega)\cap L^{\infty}(\Omega)$ is a weak solution to
\[
\Delta u=f\quad\text{in }\Omega.
\]
Then for every $x\in\Omega$,
\[
|\nabla u(x)|^{ 2}\leq C\left(  \Vert u\Vert_{L^{\infty}(\Omega)}\Vert
f\Vert_{L^{\infty}(\Omega)}+\frac{1}{\operatorname*{dist}^{2}(x,\partial
\Omega)}\Vert u\Vert_{L^{\infty}(\Omega)}^{2}\right)  ,
\]
where $C>0$ is a constant depending only on $N$.

Moreover, if $u\in H_{0}^{1}(\Omega)$, then
\[
\Vert\nabla u\Vert_{L^{\infty}(\Omega)}^{2}\leq C\Vert u\Vert_{L^{\infty
}(\Omega)}\Vert f\Vert_{L^{\infty}(\Omega)},
\]
where $C>0$ is a constant depending only on $\Omega$.
\end{lemma}

We turn to the proof of Theorem \ref{theorem gradient estimate}.

\begin{proof}
[Proof of Theorem \ref{theorem gradient estimate}]By Lemma
\ref{lemma bbh interior}, for every $x\in\Omega\setminus\Omega_{\varepsilon}$,%
\begin{align*}
|\nabla u_{\varepsilon}(x)|^{2}  &  \leq C\left(  \Vert u_{\varepsilon}%
\Vert_{L^{\infty}(\Omega)}\left\Vert \frac{1}{2\varepsilon^{2}}W^{\prime
}(u_{\varepsilon})\right\Vert _{L^{\infty}(\Omega)}+\frac{1}%
{\operatorname*{dist}^{2}(x,\partial\Omega)}\Vert u_{\varepsilon}%
\Vert_{L^{\infty}(\Omega)}^{2}\right) \\
&  \leq C\left(  \max\{|a|,|b|\}\max_{[a,b]}|W^{\prime}|\frac{1}%
{2\varepsilon^{2}}+\frac{1}{\varepsilon^{2}}\max\{|a|^{2},|b|^{2}\}\right)  ,
\end{align*}
where we used the facts that $a\leq u_{\varepsilon}\leq b$
(\ref{u between a and b}).

 To prove the last statement, observe that the function $v_{\varepsilon}:=u_{\varepsilon
}-g_{\varepsilon}\in H_{0}^{1}(\Omega)$ is a weak solution to
\[
\left\{
\begin{array}
[c]{ll}%
\Delta v_{\varepsilon}=\frac{1}{2\varepsilon^{2}}W^{\prime}(u_{\varepsilon
})-\Delta g_{\varepsilon} & \text{in }\Omega,\\
v_{\varepsilon}=0 & \text{on }\partial\Omega.
\end{array}
\right.
\]
It now suffices to apply the second part of Lemma
\ref{lemma bbh interior}.

\hfill
\end{proof}

Given the functional%
\[
F(u):=\int_{\Omega}(W(u)+|\nabla u|^{2})\,dx,
\]
we say that a function $u\in H^{1}(\Omega)$ is a \emph{local minimizer} of $F$
if for every $U\Subset\Omega$ and all $w\in H^{1}(\Omega)$ with support
contained in $U$, we have that $F(u+v)\geq F(u)$. The following theorem is a
special case of a result of Caffarelli and Cordoba
\cite{caffarelli-cordoba1995} (we refer to the paper for the general statement).

\begin{theorem}
\label{theorem caffarelli cordoba}Assume that $W$ satisfies hypotheses
\eqref{W_Smooth}-\eqref{W' three zeroes}. Let $u\in H^{1}(B(0,R))$, with
$a\leq u\leq b$, $R>2$, be a local minimizer of
\begin{equation}
F(v):=\int_{B(0,R)}(W(v)+|\nabla v|^{2})\,dx,\quad v\in H^{1}(B(0,R)),
\label{functional F}%
\end{equation}
and assume that for $a<\lambda<b$ there exists $c_{\lambda}>0$ such that%
\[
\mathcal{L}^{N}(B(0,1)\cap\{u>\lambda\})>c_{0}.
\]
Then there exists $c_{1}>0$ (depending only on $\lambda$, $c_{0}$, $N$, and
$W$) such that
\[
\mathcal{L}^{N}(B(0,r)\cap\{u>\lambda\})>c_{1}r^{N}%
\]
for every $1<r<R$.
\end{theorem}

\begin{remark}
\label{remark caffarelli-cordoba}A similar estimate continues to
hold if we replace $\{u>\lambda\}$ with $\{u<\lambda\}$ in Theorem
\ref{theorem caffarelli cordoba}. To see this, define $\bar{W}(s):=W(a+b-s)$,
and observe that if $u\in H^{1}(B(0,R))$ is a local minimizer of
\eqref{functional F}, then $v:=-u+a+b$, is a local minimizer of
\[
\bar{F}(w):=\int_{B(0,R)}(\bar{W}(w)+|\nabla w|^{2})\,dx,\quad v\in
H^{1}(B(0,R)).
\]
Moreover, $\{u<\lambda\}=\{v>a+b-\lambda\}$, where $a+b-\lambda\in(a,b)$.
Hence, it suffices to apply Theorem \ref{theorem caffarelli cordoba} to
$\bar{F}$.
\end{remark}

\begin{theorem}
\label{theorem uniform convergence level sets}Let $\Omega\subset\mathbb{R}%
^{N}$ be an open, bounded, connected set with boundary of class $C^{2,d}\ $,
$0<d\leq1$. Assume that $W$ satisfies
\eqref{W_Smooth}-\eqref{W' three zeroes} and that $g_{\varepsilon}$ satisfy \eqref{bounds g},
\eqref{g epsilon smooth}-\eqref{g epsilon -g bound}. Suppose also that
\eqref{u0=b} holds. Let $u_{\varepsilon}\in H^{1}(\Omega)$ be a minimizer of
\eqref{functional cahn-hilliard} subject to the Dirichlet boundary condition
\eqref{dirichlet boundary conditions}. Then%
\[
u_{\varepsilon}\rightarrow b\quad\text{in }L^{1}(\Omega).
\]
Moreover, for every\ $a<\lambda<b$ and for every $\delta>0$, there exists
$\varepsilon_{\delta}>0$ such that%
\begin{align} \label{incl:levelsetsinOmegadelta}
\{u_{\varepsilon}\leq\lambda\}\subseteq\Omega_{\delta}
\end{align}
for all $0<\varepsilon<\varepsilon_{\delta}$.
\end{theorem}

\begin{proof}
The fact that $u_{\varepsilon}\rightarrow b$ in $L^{1}(\Omega)$ follows from
(\ref{u0=b}) and standard properties of $\Gamma$-convergence (see
\cite[Theorem 1.21]{braides-book2002}). Next, we prove \eqref{incl:levelsetsinOmegadelta}. Given $a<\lambda<b$ and $R>0$, assume by contradiction that there exist
$\varepsilon_{n}\rightarrow0^{+}$ and $x_{n}\in\Omega\setminus\Omega_{2R}$
such that $u_{\varepsilon_{n}}(x_{n})\leq\lambda$. By compactness, we can
assume that $x_{n}\rightarrow x_{0}$. Define $v_{n}(y):=u_{\varepsilon_{n}%
}(x_{n}+\varepsilon_{n}y)$, $y\in B(0,R/\varepsilon_{n})$. By a change of
variables, and the minimality of $u_{\varepsilon_{n}}$, we have that $v_{n}$
is a local minimizer of
\[
F(v):=\int_{B(0,R/\varepsilon_{n})}(W(v)+|\nabla v|^{2})\,dx.
\]
By Theorem \ref{theorem gradient estimate} applied to $u_{\varepsilon}$ in
$\Omega\setminus\Omega_{\varepsilon}$, there exists $C_{0}>0$ such that
\[
|\nabla v_{n}(y)|\leq C_{0}\quad\text{for all }y\in B(0,R/\varepsilon_{n})
\]
provided $0<\varepsilon<2R$. Given $\lambda<\lambda_{1}<b$, since
$v_{n}(0)\leq\lambda$, it follows that
\[
v_{n}(y)\leq v_{n}(0)+C_{0}|y|\leq\lambda+C_{0}|y|<\lambda_{1}%
\]
for all $y\in B(0,(\lambda_{1}-\lambda)/C_{0})$, where, without loss of
generality, we assume that $(\lambda_{1}-\lambda)/C_{0}<1$. Hence,%
\[
\mathcal{L}^{N}(B(0,1)\cap\{v_{n}<\lambda_{1}\})\geq\mathcal{L}^{N}%
(B(0,(\lambda-\lambda_{1})/C_{0}))=c_{0}.
\]
It follows from Remark \ref{remark caffarelli-cordoba} that there exists
$c_{1}>0$ (depending only on $\lambda$, $c_{0}$, $N$, and $W$) such that
\[
\mathcal{L}^{N}(B(0,r)\cap\{v_{n}<\lambda_{1}\})>c_{1}r^{N}%
\]
for every $1<r<R/\varepsilon_{n}$. By the change of variables $x := x_{n}%
+\varepsilon_{n}y$, we find%
\[
\mathcal{L}^{N}(B(x_{n},\varepsilon_{n}r)\cap\{u_{\varepsilon_{n}}<\lambda
_{1}\})>c_{1}(\varepsilon_{n}r)^{N}%
\]
for every $1<r<R/\varepsilon_{n}$. As a consequence,
\[
\mathcal{L}^{N}(B(x_{n},R)\cap\{u_{\varepsilon_{n}}<\lambda_{1}\})\geq
c_{1}R^{N}.
\]
But this is a contradiction, since $B(x_{n},R)\subseteq B(x_{0},2R)\subseteq
\Omega$, and $u_{\varepsilon}\rightarrow b$ in $L^{1}(\Omega)$.\hfill
\end{proof}

\begin{theorem}
\label{theorem boundary estimates}Let $\Omega\subset\mathbb{R}^{N}$ be an
open, bounded, connected set with boundary of class $C^{2,d}\ $, $0<d\leq1$.
Assume that $W$ satisfies
\eqref{W_Smooth}-\eqref{W' three zeroes} and that $g_{\varepsilon}$ satisfy \eqref{bounds g},
\eqref{g epsilon smooth}-\eqref{g epsilon -g bound}. Let
$0<\delta<<1$ and suppose that
\eqref{u0=b} holds.  Then there exist $\mu>0$ and $C>0$, independent of
$\varepsilon$ and $\delta$, such that for all $\varepsilon$ sufficiently small
the following estimate holds%
\begin{equation}
0\leq b-u_{\varepsilon}(x)\leq Ce^{-\mu\delta/\varepsilon}\quad\text{for }%
x\in\Omega\setminus\Omega_{2\delta}. \label{decay b}%
\end{equation}

\end{theorem}

\begin{proof}
Fix $\rho$ as in (\ref{rho restrictions}). By Theorem \ref{theorem decay sz},
there exist $\mu>0$ and $C>0$, independent of $\varepsilon$, such that for all
$\varepsilon$ sufficiently small the following estimates hold%
\begin{equation}
0\leq b-u_{\varepsilon}(x)\leq Ce^{-\mu\operatorname*{dist}(x,K_{\rho
})/(2\varepsilon)}\quad\text{for }x\in\Omega\setminus K_{\rho}, \label{cc1}%
\end{equation}
where $K_{\rho}:=\{x\in\overline{\Omega}:\,u_{\varepsilon}(x)\leq b-\rho
\}\cup\Omega_{\varepsilon|\log\varepsilon|}$. By Theorem
\ref{theorem uniform convergence level sets}, there exists $\varepsilon
_{\delta,\rho}>0$ such that%
\begin{equation}
\{u_{\varepsilon}\leq b-\rho\}\subseteq\Omega_{\delta} \label{cc2}%
\end{equation}
for all $0<\varepsilon<\varepsilon_{\delta,\rho}$. Thus,
\[
0\leq b-u_{\varepsilon}(x)\leq Ce^{-\mu\delta/(2\varepsilon)}%
\]
for all $x\in\Omega\setminus\Omega_{2\delta}$. \hfill
\end{proof}

\section{Second-Order $\Gamma$-Limit}

\label{section main theorems}In this section, we finally prove Theorem
\ref{theorem main}.

\begin{theorem}
[Second-Order $\Gamma$-Limsup]\label{theorem limsup}Assume that $\Omega\subset
\mathbb{R}^{N}$ is an open, bounded, connected set and that its boundary
$\partial\Omega$ is of class $C^{2,d}$, $0<d\leq1$. Assume that $W$ satisfies \eqref{W_Smooth}-\eqref{W' three zeroes} and that
$g_{\varepsilon}$ satisfy \eqref{bounds g},
\eqref{g epsilon smooth}-\eqref{g epsilon -g bound}. Suppose also that
\eqref{u0=b} holds. Then there exists $\{u_{\varepsilon}\}_{\varepsilon}$ in
$H^{1}(\Omega)$ such that $\operatorname*{tr}u_{\varepsilon}=g_{\varepsilon}$
on $\partial\Omega$, $u_{\varepsilon}\rightarrow b$ in $L^{1}(\Omega)$, and%
\[
\limsup_{\varepsilon\rightarrow0^{+}}\mathcal{F}_{\varepsilon}^{(2)}%
(u_{\varepsilon})\leq\int_{\partial\Omega}\kappa(y)\int_{0}^{\infty}%
2W^{1/2}(z_{g(y)}(s))z_{g(y)}^{\prime}(s)s\,dsd\mathcal{H}^{N-1}(y)
\]
where $z_{g(y)}$ solves the Cauchy problem \eqref{cauchy problem z alpha}
with $\alpha=g(y)$.
\end{theorem}

\begin{proof}
By Lemma \ref{lemma diffeomorphism}, for $\delta>0$ sufficiently small the
function $\Phi:\partial\Omega\times\lbrack0,\delta]\rightarrow\overline
{\Omega}_{\delta}$ is of class $C^{1\,d}$. In turn, the function%
\[
\omega(y,t):=\det J_{\Phi}(y,t)
\]
is of class $C^{1,d}$ and
\[
\omega_{1}:=\min_{y\in\partial\Omega}\omega(y,0)>0.
\]
Fix
\begin{equation}
0<\omega_{0}<\frac{1}{4}\frac{C_{W}-\operatorname*{d}\nolimits_{W}%
(a,\alpha_{-})}{C_{W}}\omega_{1}.\label{omega0}%
\end{equation}
By taking $\delta>0$ sufficiently small, we can assume that%
\begin{equation}
|\omega(y,t_{1})-\omega(y,t_{2})|\leq\omega_{0}\label{omega uc}%
\end{equation}
for all $y\in\partial\Omega$ and all $t_{1},t_{2}\in\lbrack0,\delta]$.

Let $\delta_{\varepsilon}\rightarrow0^{+}$ as $\varepsilon\rightarrow0^{+}$,
and for each $y\in\overline{\Omega}$ define
\begin{equation}
\Psi_{\varepsilon}(y,r):=\int_{g_{\varepsilon}(y)}^{r}\frac{\varepsilon
}{(\delta_{\varepsilon}+W(s))^{1/2}}\,ds, \label{Psi epsilon}%
\end{equation}
and
\begin{equation}
0 \leq T_{\varepsilon}(y):=\Psi_{\varepsilon}(y,b). \label{T epsilon y def}%
\end{equation}
Note that $T_{\varepsilon}\in C^{1}(\overline{\Omega})$ with%
\begin{align*}
T_{\varepsilon}(y)  &  \leq\int_{g_{-}}^{b}\frac{\varepsilon}{(\delta
_{\varepsilon}+W(s))^{1/2}}\,ds\\
&  \leq-\frac{\sigma}{2}\varepsilon\log(\sigma^{2}\delta_{\varepsilon}%
)+\sigma\varepsilon\log(1+2(b-a))
\end{align*}
by (\ref{bounds g}) and Proposition \ref{proposition asymptotic behavior}.
Hence, there exist $C_{0}>0$ and $\varepsilon_{0}>0$, depending only on $W$
such that
\begin{equation}
T_{\varepsilon}(y)\leq C_{0}\varepsilon|\log\delta_{\varepsilon}|
\label{T epsilon y}%
\end{equation}
for all $0<\varepsilon<\varepsilon_{0}$ and all $y\in\partial\Omega$.

For each fixed $y\in\partial\Omega$, let $v_{\varepsilon}(y,\cdot
):[0,T_{\varepsilon}(y)]\rightarrow\lbrack g_{\varepsilon}(y),b]$ be the
inverse of $\Psi_{\varepsilon}(y,\cdot)$. Then $v_{\varepsilon}\left(
y,0\right)  =g_{\varepsilon}(y)$, $v_{\varepsilon}(y,T_{\varepsilon}(y))=b$,
and
\begin{equation}
\frac{\partial v_{\varepsilon}}{\partial t}(y,t)=\frac{(\delta_{\varepsilon
}+W(v_{\varepsilon}\left(  y,t\right)  ))^{1/2}}{\varepsilon}
\label{partial v epsilon t}%
\end{equation}
for $t\in\lbrack0,T_{\varepsilon}(y)]$. Assume first that $g_{\varepsilon}\in
C^{1}(\partial\Omega)$. Then by standard results on the smooth dependence of
solutions on a parameter (see, e.g. \cite[Section 2.4]{gerald-book2012}), we
have that $v_{\varepsilon}$ is of class $C^{1}$ in the variables $(y,t)$.
Extend $v_{\varepsilon}(y,t)$ to be equal to $b$ for $t>T_{\varepsilon}(y)$.

We have
\[
v_{\varepsilon}(y,\Psi_{\varepsilon}(y,r))=r
\]
for all $g_{\varepsilon}(y)\leq r\leq b$. For every $y\in\partial\Omega$ and
every tangent vector $\tau$ to $\partial\Omega$ at $y$, differentiating in the
direction $\tau$ gives
\[
\frac{\partial v_{\varepsilon}}{\partial\tau}(y,\Psi_{\varepsilon}%
(y,r))+\frac{\partial v_{\varepsilon}}{\partial t}(y,\Psi_{\varepsilon
}(y,r))\frac{\partial\Psi_{\varepsilon}}{\partial\tau}(y,r)=0.
\]
Hence,
\[
\frac{\partial v_{\varepsilon}}{\partial\tau}(y,t)+\frac{\partial
v_{\varepsilon}}{\partial t}(y,t)\frac{\partial\Psi_{\varepsilon}}%
{\partial\tau}(y,r)=0
\]
for all $y\in\partial\Omega$ and $t\in\lbrack0,T_{\varepsilon}(y))$.

By (\ref{Psi epsilon}),
\[
\frac{\partial\Psi_{\varepsilon}}{\partial\tau}(y,r)=-\frac{\varepsilon
}{(\delta_{\varepsilon}+W(g_{\varepsilon}(y)))^{1/2}}\frac{\partial
g_{\varepsilon}}{\partial\tau}(y),
\]
and so by (\ref{partial v epsilon t}), we have%
\begin{align*}
\frac{\partial v_{\varepsilon}}{\partial\tau}(y,t)  &  =-\frac{\partial
v_{\varepsilon}}{\partial t}(y,t)\frac{\partial\Psi_{\varepsilon}}%
{\partial\tau}(y,r)\\
&  =\frac{(\delta_{\varepsilon}+W(v_{\varepsilon}\left(  y,t\right)  ))^{1/2}%
}{(\delta_{\varepsilon}+W(g_{\varepsilon}(y)))^{1/2}}\frac{\partial
g_{\varepsilon}}{\partial\tau}(y)
\end{align*}
for $t\in\lbrack0,T_{\varepsilon}(y))$, while $\frac{\partial v_{\varepsilon}%
}{\partial\tau}(y,t)=0$ for $t>T_{\varepsilon}(y)$. Observe that if
$g_{\varepsilon}(y)\geq c$, then since $W$ is decreasing for $c\leq s\leq b$
and $v_{\varepsilon}(y,\cdot)$ is increasing, we have that $W(v_{\varepsilon
}\left(  y,t\right)  )\leq W(g_{\varepsilon}\left(  y\right)  )$. Hence,
$\left\vert \frac{\partial v_{\varepsilon}}{\partial\tau}(y,t)\right\vert
\leq\left\vert \frac{\partial g_{\varepsilon}}{\partial\tau}(y)\right\vert $.
On the other hand, if $g_{\varepsilon}(y)\leq c$, then by (\ref{bounds g}),
\[
(\delta_{\varepsilon}+W(g_{\varepsilon}(y)))^{1/2}\geq\min_{\lbrack g_{-}%
,c]}W^{1/2}=:W_{0}>0.
\]
Since $a\leq v_{\varepsilon}(y,t)\leq b$, in both cases, we have%
\begin{equation}
\left\vert \frac{\partial v_{\varepsilon}}{\partial\tau}(y,t)\right\vert
\leq\left\{
\begin{array}
[c]{ll}%
C\left\vert \frac{\partial g_{\varepsilon}}{\partial\tau}(y)\right\vert  &
\text{if }y\in\partial\Omega\text{ and }t\in\lbrack0,T_{\varepsilon}(y)),\\
0 & \text{if }y\in\partial\Omega\text{ and }t\in(T_{\varepsilon}(y),\delta].
\end{array}
\right.  \label{partial v epsilon y}%
\end{equation}
If $g_{\varepsilon}\in H^{1}(\partial\Omega)$, a density argument shows that
$v_{\varepsilon}\in H^{1}(\partial\Omega\times(0,\delta))$ and that
(\ref{partial v epsilon t}) and (\ref{partial v epsilon y}) continues to hold a.e.

Set%
\begin{equation}
u_{\varepsilon}(x):=\left\{
\begin{array}
[c]{ll}%
v_{\varepsilon}(\Phi^{-1}(x)) & \text{if }x\in\Omega_{\delta},\\
b & \text{if }x\in\Omega\setminus\Omega_{\delta},
\end{array}
\right.  . \label{u epsilon bl}%
\end{equation}
Then $u_{\varepsilon}\in H^{1}(\Omega)$, with
\begin{equation}
|\nabla u_{\varepsilon}(x)|^{2}\leq\left\vert \frac{\partial v_{\varepsilon}%
}{\partial t}(\Phi^{-1}(x))\right\vert ^{2}+C\Vert\nabla y\Vert_{L^{\infty
}(\Omega_{\delta})}^{2}\left\vert \nabla_{\tau}v_{\varepsilon}(\Phi
^{-1}(x))\right\vert ^{2}, \label{299b}%
\end{equation}
where we used the facts that $\Phi^{-1}(x)=(y(x),\operatorname*{dist}%
(x,\partial\Omega))$, $\left\vert \nabla\operatorname*{dist}(x,\partial
\Omega)\right\vert =1$, and $\tau\cdot\nabla\operatorname*{dist}%
(x,\partial\Omega)=0$ for every vector $\tau$ such that $\tau\cdot\nu(y)=0$.

In view of Lemma \ref{lemma diffeomorphism}, we can use the change of
variables $x=\Phi(y,t)$ and Tonelli's theorem to write%
\begin{align*}
\mathcal{F}_{\varepsilon}^{(2)}(u_{\varepsilon}) &  =\int_{\partial\Omega}%
\int_{0}^{\delta}\left(  \frac{1}{\varepsilon^{2}}W(u_{\varepsilon}%
(\Phi(y,t)))+|\nabla u_{\varepsilon}(\Phi(y,t))|^{2}\right)  \omega
(y,t)\,dtd\mathcal{H}^{N-1}(y)\\
&  \quad-\frac{1}{\varepsilon}\int_{\partial\Omega}d(g(y),b)\,d\mathcal{H}%
^{N-1}(y)\\
&  \leq\left(  \int_{\partial\Omega}\int_{0}^{\delta}\left(  \frac
{1}{\varepsilon^{2}}W(v_{\varepsilon}(y,t))+\left\vert \frac{\partial
v_{\varepsilon}}{\partial t}(y,t)\right\vert ^{2}\right)  \omega
(y,t)\,dtd\mathcal{H}^{N-1}(y)\right.  \\
&  \quad\left.  -\frac{1}{\varepsilon}\int_{\partial\Omega}%
d(g(y),b)\,d\mathcal{H}^{N-1}(y)\right)  \\
&  \quad+C\Vert\nabla y\Vert_{L^{\infty}(\Omega_{\delta})}^{2}\int%
_{\partial\Omega}\int_{0}^{\delta}\left\vert \nabla_{\tau}v_{\varepsilon
}(y,t)\right\vert ^{2}\omega(y,t)\,dtd\mathcal{H}^{N-1}(y)=:\mathcal{A}%
+\mathcal{B}.
\end{align*}
Taking $\delta_{\varepsilon}$ as in (\ref{1d delta epsilon}), by Theorem
\ref{theorem 1d second gamma limsup}, there exist constants $0<\varepsilon
_{0}<1$, $C,C_{0}>0$, and $\gamma_{0},\gamma_{1}>0$, depending only on
$\alpha_{-}$, $A_{0}$, $B_{0}$, $T$, $\omega$, and $W$, such that
\begin{align*}
&  \int_{0}^{\delta}\left(  \frac{1}{\varepsilon^{2}}W(v_{\varepsilon
}(y,t))+\left\vert \frac{\partial v_{\varepsilon}}{\partial t}(y,t)\right\vert
^{2}\right)  \omega(y,t)\,dt-\frac{1}{\varepsilon}\operatorname*{d}%
\nolimits_{W}(b,g(y))\\
&  \leq\int_{0}^{l}2W(p_{\varepsilon}(y,t))t\,dt\,\frac{\partial\omega
}{\partial t}(y,0)+Ce^{-2\sigma l}\left(  2\sigma l+1\right)  +C\varepsilon
^{2\gamma}l+C\varepsilon^{\gamma_{1}}|\log\varepsilon|^{\gamma_{0}}%
\end{align*}
for all $0<\varepsilon<\varepsilon_{0}$ and all $l>0$, where $p_{\varepsilon
}(y,t)=v_{\varepsilon}(y,\varepsilon t)$, $p_{\varepsilon}(y,\cdot)\rightarrow
z_{g(y)}$ pointwise in $[0,\infty)$, where $z_{\alpha}$ solves the Cauchy
problem \eqref{cauchy problem z alpha}. Hence, by Lemma
\ref{lemma diffeomorphism},%
\begin{align*}
\mathcal{A}\leq\int_{\partial\Omega}\kappa(y)\int_{0}^{l}2W(p_{\varepsilon
}(y,t))t\,dtd\mathcal{H}^{N-1}(y)&+Ce^{-2\sigma l}\left(  2\sigma l+1\right)
\\&+C\varepsilon^{2\gamma}l+C\varepsilon^{\gamma_{1}}|\log\varepsilon
|^{\gamma_{0}}%
\end{align*}
for all $0<\varepsilon<\varepsilon_{0}$ and all $l>0$. Since $p_{\varepsilon
}(y,t)\rightarrow z_{g(y)}(t)$ for all $t\in\lbrack0,l]$ and $a\leq
p_{\varepsilon}(y,t)\leq b$, we can apply the Lebesgue dominated convergence
theorem to obtain%
\begin{align*}
\lim_{\varepsilon\rightarrow0^{+}}\int_{\partial\Omega}\kappa(y)\int_{0}%
^{l}2W(p_{\varepsilon}(y,t))t\,dt\,&d\mathcal{H}^{N-1}(y)\\&=\int_{\partial\Omega
}\kappa(y)\int_{0}^{l}2W(z_{g(y)}(t))t\,dt\,d\mathcal{H}^{N-1}(y).
\end{align*}
Hence,
\[
\limsup_{\varepsilon\rightarrow0^{+}}\mathcal{A}\leq\int_{\partial\Omega
}\kappa(y)\int_{0}^{l}2W(z_{g(y)}(t))t\,dt\,d\mathcal{H}^{N-1}(y)+Ce^{-2\sigma
l}\left(  2\sigma l+1\right)  .
\]
By (\ref{estimate z alpha}) and\ the Lebesgue dominated convergence theorem,
the right-hand side converges to%
\[
\int_{\partial\Omega}\kappa(y)\int_{0}^{\infty}2W(z_{g(y)}%
(t))t\,dt\,d\mathcal{H}^{N-1}(y).
\]

On the other hand, by (\ref{T epsilon y}) and (\ref{partial v epsilon y}),
\begin{align}
\mathcal{B}  &  \leq C\Vert\nabla y\Vert_{L^{\infty}(\Omega_{\delta})}^{2}%
\int_{\partial\Omega}\left\vert \nabla_{\tau}g_{\varepsilon}(y)\right\vert
^{2}\int_{0}^{T_{\varepsilon}(y)}\omega(y,t)\,dt\,d\mathcal{H}^{N-1}%
(y)\nonumber\\
&  \leq C\varepsilon|\log\varepsilon|\Vert\omega\Vert_{L^{\infty}%
(\partial\Omega\times\lbrack0,\delta])}\int_{\partial\Omega}\left\vert
\partial_{\tau}g_{\varepsilon}(y)\right\vert ^{2}d\mathcal{H}^{N-1}(y)=o(1)
\label{299c}%
\end{align}
by (\ref{g epsilon bound derivatives}).

In conclusion, we have shown that%
\begin{equation}
\mathcal{F}_{\varepsilon}^{(2)}(u_{\varepsilon})\leq\int_{\partial\Omega
}\kappa(y)\int_{0}^{\infty}2W^{1/2}(z_{g(y)}(s))z_{g(y)}^{\prime
}(s)s\,ds\,d\mathcal{H}^{N-1}(y)+o(1).\nonumber
\end{equation}

\textbf{Step 2: }We claim that
\[
u_{\varepsilon}\rightarrow u_{0}\quad\text{in }L^{1}(\Omega).
\]
In view of Lemma \ref{lemma diffeomorphism}, we can use the change of
variables $x:=\Phi(y,t)$ and Tonelli's theorem to write%
\begin{align*}
\int_{\Omega}|u_{\varepsilon}-u_{0}|\,dx  &  =\int_{\partial\Omega}\int%
_{0}^{\delta}|u_{\varepsilon}(\Phi(y,t)))-b|\omega(y,t)\,dtd\mathcal{H}%
^{N-1}(y)\\
&  =\int_{\partial\Omega}\int_{0}^{T_{\varepsilon}(y)}|v_{\varepsilon
}(y,t)-b|\omega(y,t)\,dtd\mathcal{H}^{N-1}(y)\\
&  \leq C\varepsilon|\log\varepsilon|,
\end{align*}
where we used the fact that $v_{\varepsilon}(y,t)=b$ for $t\geq T_{\varepsilon
}(y)$ and (\ref{T epsilon y}).\hfill
\end{proof}

For every measurable set $E\subseteq\Omega$, we define the localized energy%
\[
E_{\varepsilon}(u,E):=\int_{E}\left(  \frac{1}{\varepsilon^{2}}W(u)+|\nabla
u|^{2}\right)  \,dx,\quad u\in H^{1}(\Omega).
\]

\begin{theorem}
[Second-Order $\Gamma$-Liminf]\label{theorem liminf}Assume that $\Omega\subset
\mathbb{R}^{N}$ is an open, bounded, connected set and that its boundary
$\partial\Omega$ is of class $C^{2,d}$, $0<d\leq1$. Assume that $W$ satisfies \eqref{W_Smooth}-\eqref{W' three zeroes} and that
$g_{\varepsilon}$ satisfy \eqref{bounds g},
\eqref{g epsilon smooth}-\eqref{g epsilon -g bound}. Suppose also that
\eqref{u0=b} holds. Then%
\[
\liminf_{\varepsilon\rightarrow0^{+}}\mathcal{F}_{\varepsilon}^{(2)}%
(u_{\varepsilon})\geq\int_{\partial\Omega}\kappa(y)\int_{0}^{\infty}%
2W^{1/2}(z_{g(y)}(s))z_{g(y)}^{\prime}(s)s\,dsd\mathcal{H}^{N-1}(y),
\]
where $z_{\alpha}$ solves the Cauchy problem \eqref{cauchy problem z alpha}
with $\alpha=g(y)$.
\end{theorem}

\begin{proof}
We choose $\omega$ and $\delta$ as in the proof of Theorem
\ref{theorem limsup}. By Theorem \ref{theorem boundary estimates} (with
$\Omega_{\delta}$ and $\Omega_{2\delta}$ replaced by $\Omega_{\delta/2}$ and
$\Omega_{\delta}$, respectively), we can assume that
\begin{equation}
0\leq b-u_{\varepsilon}(x)\leq Ce^{-\mu\delta/\varepsilon}\quad\text{for }%
x\in\Omega\setminus\Omega_{\delta}\label{902}%
\end{equation}
for all $0<\varepsilon<\varepsilon_{\delta}$.

Write
\begin{align*}
F_{\varepsilon}^{(2)}(u_{\varepsilon}) &  =E_{\varepsilon}(u_{\varepsilon
},\Omega\setminus\Omega_{\delta})\\
&  \quad+\left(  E_{\varepsilon}(u_{\varepsilon},\Omega_{\delta})-\frac
{1}{\varepsilon}\int_{\partial\Omega}\operatorname*{d}\nolimits_{W}%
(g,b)\,d\mathcal{H}^{N-1}\right)  \\
&  \quad=:\mathcal{A}+\mathcal{B}.
\end{align*}
Since $\mathcal{A}\geq0$, it remains to study $\mathcal{B}$. In view of Lemma
\ref{lemma diffeomorphism}, we can use the change of variables $x=\Phi(y,t)$
and Tonelli's theorem to write%
\[
E_{\varepsilon}(u_{\varepsilon},\Omega_{\delta})=\int_{\partial\Omega}\int%
_{0}^{\delta}\left(  \frac{1}{\varepsilon^{2}}W(u_{\varepsilon}(\Phi
(y,t)))+|\nabla u_{\varepsilon}(\Phi(y,t))|^{2}\right)  \omega
(y,t)\,dtd\mathcal{H}^{N-1}(y).
\]
Since $u_{\varepsilon}\in C^{1}(\overline{\Omega})$, if we define%
\[
\tilde{u}_{\varepsilon}(y,t):=u_{\varepsilon}(y+t\nu(y)),
\]
we have that%
\[
\frac{\partial\tilde{u}_{\varepsilon}}{\partial t}(y,t)=\frac{\partial
u_{\varepsilon}}{\partial\nu(y)}(y+t\nu(y)),
\]
and so,%
\begin{align}
&  E_{\varepsilon}(u_{\varepsilon},\Omega_{\delta})-\frac{1}{\varepsilon}%
\int_{\partial\Omega}\operatorname*{d}\nolimits_{W}(g,b)\,d\mathcal{H}%
^{N-1}\label{900}\\
&  \geq\int_{\partial\Omega}\left[  \int_{0}^{\delta}\left(  \frac
{1}{\varepsilon^{2}}W(\tilde{u}_{\varepsilon}(y,t))+\left\vert \frac
{\partial\tilde{u}_{\varepsilon}}{\partial t}(y,t)\right\vert ^{2}\right)
\omega(y,t)\,dt-\frac{1}{\varepsilon}\operatorname*{d}\nolimits_{W}%
(g(y),b)\right]  d\mathcal{H}^{N-1}(y).\nonumber
\end{align}
For $y\in\partial\Omega$, in view of (\ref{902}), we have that%
\begin{equation}
b-C_{\rho}e^{-\mu_{\rho}\delta/(2\varepsilon)}\leq\tilde{u}_{\varepsilon
}(y,\delta)\leq b.\label{903}%
\end{equation}
Let $v_{\varepsilon}^{y}\in H^{1}([0,\delta])$ be the minimizer of the
functional
\[
v\mapsto\int_{0}^{\delta}\left(  \frac{1}{\varepsilon^{2}}W(v(t))+|v^{\prime
}(t)|^{2}\right)  \omega(y,t)\,dt
\]
defined for all $v\in H^{1}([0,\delta])$ such that $v(0)=g_{\varepsilon}(y)$
and $v(\delta)=\tilde{u}_{\varepsilon}(y,\delta)$. In view of
(\ref{g epsilon -g bound}) and (\ref{903}), we can apply Theorem
\ref{theorem liminf second 1d} to find $0<\varepsilon_{0}<1$, $C>0$, and
$l_{0}>1$, depending only on $\alpha_{-}$, $a$, $b$, $\delta$, $\omega$, and
$W$ such that%
\begin{align*}
\psi_{\varepsilon}(y) &  :=\int_{0}^{\delta}\left(  \frac{1}{\varepsilon^{2}%
}W(\tilde{u}_{\varepsilon}(y,t))+\left\vert \frac{\partial\tilde
{u}_{\varepsilon}}{\partial t}(y,t)\right\vert ^{2}\right)  \omega
(y,t)\,dt-\frac{1}{\varepsilon}\operatorname*{d}\nolimits_{W}(b,g(y))\\
&  \geq\int_{0}^{\delta}\left(  \frac{1}{\varepsilon^{2}}W(v_{\varepsilon}%
^{y}(t))+|(v_{\varepsilon}^{y})^{\prime}(t)|^{2}\right)  \omega(y,t)\,dt-\frac
{1}{\varepsilon}\operatorname*{d}\nolimits_{W}(b,g(y))\\
&  \geq2\frac{\partial\omega}{\partial t}(y,0)\int_{0}^{l}W^{1/2}%
(w_{\varepsilon})w_{\varepsilon}^{\prime}s\,ds-Ce^{-l\mu}\left(
l\mu+1\right)  \\&\hspace{5cm}-C l^2\varepsilon^{1/2}-C\varepsilon^{\gamma_{1}}|\log
\varepsilon|^{2+\gamma_{0}}=:\phi_{\varepsilon}(y),
\end{align*}
for all $0<\varepsilon<\varepsilon_{0}$ and $l>l_{0}$, where $w_{\varepsilon
}(s):=v_{\varepsilon}(\varepsilon s)$ for $s\in\lbrack0,\delta\varepsilon
^{-1}]$ satisfies%
\begin{equation}
\lim_{\varepsilon\rightarrow0^{+}}\int_{0}^{l}W^{1/2}(w_{\varepsilon
})w_{\varepsilon}^{\prime}s\,ds=\int_{0}^{l}W^{1/2}(z_{ g(y)})z_{ g(y)
}^{\prime}s\,ds\label{904}%
\end{equation}
for every $l>0$ and where $z_{ g(y)}$ solves the Cauchy problem
\eqref{cauchy problem z alpha} with $\alpha=g(y)$. By Corollary
\ref{corollary bounded derivative}, there exists a constant $C>0$ depending
only on $\alpha_{-}$, $a$, $b$, $\delta$, $\omega$, and $W$ such that
$|w_{\varepsilon}(t)|\leq C$ for all $t\in\lbrack0,\delta\varepsilon^{-1}]$
and for all $0<\varepsilon<\varepsilon_{0}$. Hence, $|\phi_{\varepsilon
}(y)|\leq C_{l}$ for all $y\in\partial\Omega$ and for all $0<\varepsilon
<\varepsilon_{0}$. Since $\psi_{\varepsilon}-C_{l}\geq0$, we can apply Fatou's
lemma to obtain%
\begin{align*}
 & \liminf_{\varepsilon\rightarrow0^{+}}\int_{\partial\Omega}\psi_\varepsilon(y) d\mathcal{H}^{N-1}(y)
\geq\int_{\partial\Omega}\liminf_{\varepsilon\rightarrow0^{+}}\psi_\varepsilon(y) d\mathcal{H}^{N-1}(y)\\
& \hspace{4.1cm}  \geq\int_{\partial\Omega}\liminf_{\varepsilon\rightarrow0^{+}}%
\phi_{\varepsilon}(y)\,d\mathcal{H}^{N-1}(y)\\
&  =\int_{\partial\Omega}2\kappa(y)\left(  \int_{0}^{l}W^{1/2}(z_{g(y)}%
)z_{g(y)}^{\prime}s\,ds-Ce^{-l\mu}\left(  l\mu+1\right)  \right)
\,d\mathcal{H}^{N-1}(y).
\end{align*}
%\begin{align*}
%&  \liminf_{\varepsilon\rightarrow0^{+}}\int_{\partial\Omega}\left[  \int%
%_{0}^{\delta}\left(  \frac{1}{\varepsilon^{2}}W(\tilde{u}_{\varepsilon
%}(y,t))+\left\vert \frac{\partial\tilde{u}_{\varepsilon}}{\partial
%t}(y,t)\right\vert ^{2}\right)  \omega(y,t)\,dt-\frac{1}{\varepsilon
%}\operatorname*{d}\nolimits_{W}(g(y),b)\right]  d\mathcal{H}^{N-1}(y)\\
%&  \geq\int_{\partial\Omega}\liminf_{\varepsilon\rightarrow0^{+}}\left[
%\int_{0}^{\delta}\left(  \frac{1}{\varepsilon^{2}}W(\tilde{u}_{\varepsilon
%}(y,t))+\left\vert \frac{\partial\tilde{u}_{\varepsilon}}{\partial
%t}(y,t)\right\vert ^{2}\right)  \omega(y,t)\,dt-\frac{1}{\varepsilon
%}\operatorname*{d}\nolimits_{W}(g(y),b)\right]  d\mathcal{H}^{N-1}(y)\\
%&  \geq\int_{\partial\Omega}\liminf_{\varepsilon\rightarrow0^{+}}%
%\phi_{\varepsilon}(y)\,d\mathcal{H}^{N-1}(y)\\
%&  =\int_{\partial\Omega}2\kappa(y)\left(  \int_{0}^{l}W^{1/2}(z_{g(y)}%
%)z_{g(y)}^{\prime}s\,ds-Ce^{-l\mu}\left(  l\mu+1\right)  \right)
%\,d\mathcal{H}^{N-1}(y).
%\end{align*}
Letting $l\rightarrow\infty$ and using the Lebesgue monotone convergence
theorem for the first term gives%
\begin{align*}
&  \liminf_{\varepsilon\rightarrow0^{+}}\int_{\partial\Omega}\psi_\varepsilon(y)  d\mathcal{H}^{N-1}(y)\\
&  \geq\int_{\partial\Omega}2\kappa(y)\int_{0}^{\infty}W^{1/2}(z_{g(y)}%
(s))z_{g(y)}^{\prime}(s)s\,dsd\mathcal{H}^{N-1}(y).
\end{align*}
Recalling the definition of $\psi_\varepsilon$ concludes the proof. \hfill
\end{proof}

\section{Note Added to Proof}

When this paper was almost complete, we became aware of the paper by Alikakos
and Fusco \cite{alikakos-fusco2023} (and consequently of
\cite{alikakos-zhiyuan2024}, \cite{gazoulis2024}, \cite{sandier-sternberg2024}%
), where they studied the case $g_{\varepsilon}\equiv z_{0}$, where
$z_{0}\notin W^{-1}(\{0\})$, in the vectorial case, that is, when
$W:\mathbb{R}^{m}\rightarrow\lbrack0,\infty)$ with $m\geq1$, and $W$ has a
finite number of wells. In \cite[Lemma 3.1 and Theorem 3.3]%
{alikakos-fusco2023}, the authors proved that there exists $z_{1}\in
W^{-1}(\{0\})$ such that minimizers $u_{\varepsilon}$ of $F_{\varepsilon}$
satisfy the bound
\begin{equation}
\varepsilon\sigma^{+}\mathcal{H}^{N-1}(\partial\Omega)(1-C_{1}\varepsilon
^{1/3})\leq F_{\varepsilon}(u_{\varepsilon})\leq\varepsilon\sigma
^{+}\mathcal{H}^{N-1}(\partial\Omega)+C_{2}\varepsilon^{2},
\label{estimate alikakos-fusco0}%
\end{equation}
where $\sigma^{+}$ is the vectorial version of $\operatorname*{d}%
\nolimits_{W}(z_{0},z_{1})$ and $C_{1}$ and $C_{2}$ are positive constants
independent of $\varepsilon$. Using this estimate, they were able to show
that
\begin{equation}
|u_{\varepsilon}(x)-z_{1}|\leq Ke^{-k(\operatorname*{dist}(x,\partial
\Omega)-C\varepsilon^{1/[3(N-1)]})_{+}/\varepsilon},\quad x\in\Omega,
\label{estimate alikakos-fusco}%
\end{equation}
where $C$, $K$, $k$ are positive constants independent of $\varepsilon$.

In the scalar case $m=1$ we are able to replace
(\ref{estimate alikakos-fusco0}) with the sharp bound (\ref{sharp bound}).

\section*{Acknowledgements}

The research of I. Fonseca was partially supported the National Science
Foundation under grants No. DMS-2205627, DMS-2108784 and DMS-23423490, and that of 
G. Leoni
under grant No. DMS-2108784. The research of L. Kreutz was supported by the DFG through the Emmy Noether Programme (project number 509436910). 

G. Leoni would like to thank R. Murray and I. Tice for useful conversations on
the subject of this paper.
 The authors would like to thank Pascal  Steinke and Francesco Colasanto for carefully reading the manuscript.

\bibliographystyle{abbrv}
\bibliography{modica-mortola-refs}

\begin{thebibliography}{10}

\bibitem{alikakos-fusco2023}
N.~D. Alikakos and G.~Fusco.
\newblock Sharp lower bounds for the vector {A}llen-{C}ahn energy and
  qualitative properties of minimizers under no symmetry hypotheses.
\newblock {\em Bull. Hellenic Math. Soc.}, 67:12--58, 2023.

\bibitem{alikakos-zhiyuan2024}
N.~D. Alikakos and Z.~Geng.
\newblock On the triple junction problem without symmetry hypotheses.
\newblock {\em Arch. Ration. Mech. Anal.}, 248(2):Paper No. 24, 58, 2024.

\bibitem{anzellotti-baldo1993}
G.~Anzellotti and S.~Baldo.
\newblock Asymptotic development by {$\Gamma$}-convergence.
\newblock {\em Appl. Math. Optim.}, 27(2):105--123, 1993.

\bibitem{anzellotti-baldo-orlandi1996}
G.~Anzellotti, S.~Baldo, and G.~Orlandi.
\newblock {$\Gamma$}-asymptotic developments, the {C}ahn-{H}illiard functional,
  and curvatures.
\newblock {\em J. Math. Anal. Appl.}, 197(3):908--924, 1996.

\bibitem{baldi2001}
A.~Baldi et~al.
\newblock Weighted bv functions.
\newblock {\em Houston J. Math}, 27(3):683--705, 2001.

\bibitem{baldo1990}
S.~Baldo.
\newblock Minimal interface criterion for phase transitions in mixtures of
  {C}ahn-{H}illiard fluids.
\newblock {\em Ann. Inst. H. Poincar\'e{} C Anal. Non Lin\'eaire}, 7(2):67--90,
  1990.

\bibitem{bellettini-nayam-novaga2015}
G.~Bellettini, A.-H. Nayam, and M.~Novaga.
\newblock {$\Gamma$}-type estimates for the one-dimensional {A}llen-{C}ahn's
  action.
\newblock {\em Asymptot. Anal.}, 94(1-2):161--185, 2015.

\bibitem{bethuel-brezis-helein1993}
F.~Bethuel, H.~Brezis, and F.~H\'elein.
\newblock Asymptotics for the minimization of a {G}inzburg-{L}andau functional.
\newblock {\em Calc. Var. Partial Differential Equations}, 1(2):123--148, 1993.

\bibitem{braides-book2002}
A.~Braides.
\newblock {\em {$\Gamma$}-convergence for beginners}, volume~22 of {\em Oxford
  Lecture Series in Mathematics and its Applications}.
\newblock Oxford University Press, Oxford, 2002.

\bibitem{caffarelli-cordoba1995}
L.~A. Caffarelli and A.~C\'ordoba.
\newblock Uniform convergence of a singular perturbation problem.
\newblock {\em Comm. Pure Appl. Math.}, 48(1):1--12, 1995.

\bibitem{cristoferi-gravina2021}
R.~Cristoferi and G.~Gravina.
\newblock Sharp interface limit of a multi-phase transitions model under
  nonisothermal conditions.
\newblock {\em Calc. Var. Partial Differential Equations}, 60(4):Paper No. 142,
  62, 2021.

\bibitem{dalmaso-book1993}
G.~Dal~Maso.
\newblock {\em An introduction to {$\Gamma$}-convergence}.
\newblock Progress in Nonlinear Differential Equations and their Applications,
  8. Birkh\"auser Boston, Inc., Boston, MA, 1993.

\bibitem{dalmaso-fonseca-leoni2015}
G.~Dal~Maso, I.~Fonseca, and G.~Leoni.
\newblock Second order asymptotic development for the anisotropic
  {C}ahn-{H}illiard functional.
\newblock {\em Calc. Var. Partial Differential Equations}, 54(1):1119--1145,
  2015.

\bibitem{dephilippis-maggi2015}
G.~De~Philippis and F.~Maggi.
\newblock Regularity of free boundaries in anisotropic capillarity problems and
  the validity of {Y}oung's law.
\newblock {\em Arch. Ration. Mech. Anal.}, 216(2):473--568, 2015.

\bibitem{evans-book2010}
L.~Evans.
\newblock {\em Partial differential equations}, volume~19 of {\em Graduate
  Studies in Mathematics}.
\newblock American Mathematical Society, Providence, RI, 1998.

\bibitem{fonseca-kreutz-leoni2025II}
I.~Fonseca, L.~Kreutz, and G.~Leoni.
\newblock Second-order {$\Gamma$}-limit for the {C}ahn--{H}illiard functional
  with {D}irichlet boundary conditions, {II}, 2025.

\bibitem{fonseca-liu2017}
I.~Fonseca and P.~Liu.
\newblock The weighted ambrosio--tortorelli approximation scheme.
\newblock {\em SIAM Journal on Mathematical Analysis}, 49(6):4491--4520, 2017.

\bibitem{fonseca-tartar1989}
I.~Fonseca and L.~Tartar.
\newblock The gradient theory of phase transitions for systems with two
  potential wells.
\newblock {\em Proceedings of the Royal Society of Edinburgh: Section A
  Mathematics}, 111(1-2):89--102, 1989.

\bibitem{gazoulis2024}
D.~Gazoulis.
\newblock On the {$ \Gamma $}-convergence of the {A}llen-{C}ahn functional with
  boundary conditions, 2024.

\bibitem{lee-book2013}
J.~M. Lee.
\newblock {\em Introduction to smooth manifolds}, volume 218 of {\em Graduate
  Texts in Mathematics}.
\newblock Springer, New York, second edition, 2013.

\bibitem{leoni-murray2016}
G.~Leoni and R.~Murray.
\newblock Second-{O}rder {$\Gamma$}-limit for the {C}ahn--{H}illiard
  {F}unctional.
\newblock {\em Arch. Ration. Mech. Anal.}, 219(3):1383--1451, 2016.

\bibitem{leoni-murray2019}
G.~Leoni and R.~Murray.
\newblock Local minimizers and slow motion for the mass preserving
  {A}llen-{C}ahn equation in higher dimensions.
\newblock {\em Proc. Amer. Math. Soc.}, 147(12):5167--5182, 2019.

\bibitem{modica1987}
L.~Modica.
\newblock The gradient theory of phase transitions and the minimal interface
  criterion.
\newblock {\em Arch. Rational Mech. Anal.}, 98(2):123--142, 1987.

\bibitem{modica-mortola1977}
L.~Modica and S.~Mortola.
\newblock Un esempio di {$\Gamma$}-convergenza.
\newblock {\em Boll. Un. Mat. Ital. B (5)}, 14(1):285--299, 1977.

\bibitem{niethammer1995}
B.~S. Niethammer.
\newblock Existence and uniqueness of radially symmetric stationary points
  within the gradient theory of phase transitions.
\newblock {\em European Journal of Applied Mathematics}, 6(1):45–67, 1995.

\bibitem{owen-rubinstein-sternberg1990}
N.~C. Owen, J.~Rubinstein, and P.~Sternberg.
\newblock Minimizers and gradient flows for singularly perturbed bi-stable
  potentials with a {D}irichlet condition.
\newblock {\em Proc. Roy. Soc. London Ser. A}, 429(1877):505--532, 1990.

\bibitem{sandier-sternberg2024}
E.~Sandier and P.~Sternberg.
\newblock Allen-cahn solutions with triple junction structure at infinity,
  2024.

\bibitem{sternberg1988}
P.~Sternberg.
\newblock The effect of a singular perturbation on nonconvex variational
  problems.
\newblock {\em Arch. Rational Mech. Anal.}, 101(3):209--260, 1988.

\bibitem{sternberg-zumbrun1998}
P.~Sternberg and K.~Zumbrun.
\newblock Connectivity of phase boundaries in strictly convex domains.
\newblock {\em Archive for Rational Mechanics and Analysis}, 141(4):375--400,
  1998.

\bibitem{gerald-book2012}
G.~Teschl.
\newblock {\em Ordinary differential equations and dynamical systems}, volume
  140 of {\em Graduate Studies in Mathematics}.
\newblock American Mathematical Society, Providence, RI, 2012.

\end{thebibliography}

\end{document}